\documentclass[11pt,reqno]{amsart}

\usepackage{graphicx,color,xcolor} 
\usepackage{amssymb,amsmath,amsthm,amsfonts,mathrsfs}
\usepackage{properfan}
\usepackage{hyperref}

\makeatletter
\def\part{\@startsection{part}{0}%
  \z@{\linespacing\@plus\linespacing}{.5\linespacing}%
  {\normalfont\large\scshape\centering}} 
\makeatother

\nc{\NC}{\on{NC}}
\nc{\e}{\tsl{e}}
\nc{\Pol}{\on{Pol}}
\renewcommand{\tikzdot}{\hackcenter{\begin{tikzpicture}[scale=0.375]
    \draw (0,0) -- (0,1.5);
    \fill (0,0.75) circle (6pt);
\end{tikzpicture}}}
\renewcommand{\tikzcrossing}{\hackcenter{\begin{tikzpicture}[scale=0.375]
    \draw (0,0) -- (1.5,1.5);
    \draw (1.5,0) -- (0,1.5);
\end{tikzpicture}}}
\newcommand{\tikzcrossbr}{\hackcenter{\begin{tikzpicture}[scale=0.375]
    \draw[black](0,0)--(1.5,1.5);
    \draw[red](1.5,0)--(0,1.5);
\end{tikzpicture}}}
\newcommand{\tikzcrossrb}{\hackcenter{\begin{tikzpicture}[scale=0.375]
    \draw[red](0,0)--(1.5,1.5);
    \draw[black](1.5,0)--(0,1.5);
\end{tikzpicture}}}
\newcommand{\tikzcrossdistant}{\hackcenter{\begin{tikzpicture}[scale=0.375]
    \draw[black](0,0)--(1.5,1.5);
    \draw[cyan](1.5,0)--(0,1.5);
\end{tikzpicture}}}
\nc{\Wlie}{\mathfrak{W}}
\tikzset{squigs/.style={decorate, decoration={snake,amplitude=.25mm,segment length=0.95mm,
                       post length=0mm,pre length=0mm}}} 

\nc{\CW}{\cal{W}}
\renewcommand{\f}{\tsl{f}}
\nc{\Seq}{\on{Seq}}
\nc{\Cor}{\on{Cor}}

\nc{\Inc}{\on{Inc}}
\nc{\E}{\tsl{E}}

\nc{\sha}{{\easycyrsymbol{\cyrsh}}}
\newcommand{\h}{\tsl{h}}

\renewcommand{\te}[1]{\textrm{#1}}

\title{The coherent KLR sheaf} 
\author{Matthew Hase-Liu} 
\address{Department of Mathematics, MIT, Cambridge, MA 02139, USA}
\email{\href{mailto:mhaseliu@mit.edu}{mhaseliu@mit.edu}}

\author{Fan Zhou} 
\address{Department of Mathematics, Columbia University, New York, NY 10027, USA}
\email{\href{mailto:fz2326@columbia.edu}{fz2326@columbia.edu}}

\pgfplotsset{compat=1.14}
\begin{document}

\maketitle

\vspace*{-0.25in}

\vspace*{0.15in}




\begin{abstract}
    We construct the KLR algebra (in type $A$) as affine sections of a coherent sheaf $\mathcal{E}$ on a product of projective spaces $\mathbb{P}^n$. It was shown by Elias-Qi that the Lie algebra $\mathfrak{sl}_2$ acts on such KLR algebras; our construction geometrically explains this action as stemming from the $\mathfrak{sl}_2$-action on the twisting sheaves $\mathcal{O}(k)$ of $\mathbb{P}^1$ and gives the maximal finite-dimensional $\mathfrak{sl}_2$-submodule of KLR as global sections of $\mathcal{E}$. We furthermore extend the Elias-Qi action to a Witt action. Kleshchev-Loubert-Miemietz showed that the KLR of type $A$ has a graded affine cellular structure; our construction also gives a geometric incarnation of this.
\end{abstract}

\textcolor{white}{$\Db$}\vspace{-1em}

\tableofcontents

\section{Introduction}
The Khovanov-Lauda-Rouquier (KLR) algebras, also called quiver Hecke algebras, were introduced independently by Khovanov-Lauda in \cite{KL09} and \cite{KLII} and by Rouquier in \cite{rouquier20082}. They are certain graded algebras which categorify the upper half of the quantum group. To be more precise, for each $\BN$-linear combination $\alpha$ of simple roots of a Lie algebra $\glie$, the KLR algebra $\CR_\alpha$ is a diagrammatically defined graded algebra such that its Grothendieck group is isomorphic to the $\alpha$-th weight space of the upper half of Lusztig's integral form of the quantum group\footnote{Classically this is not how most people phrase things; here we implicitly appeal to the finite homological dimension of KLR. },
\[K_0(\Mod_\te{fg} \CR_\alpha)\cong {{U}}_q^+(\glie)^\alpha.\]
The KLR algebras are moreover equipped with a monoidal (or horizontal) product
\[\CR_\alpha\otimes\CR_{\alpha'}\lto\CR_{\alpha+\alpha'},\] 
so that the Grothendieck group of the direct sum $\CR=\bigoplus_\alpha\CR_\alpha$ receives a ring structure coming from the induction tensor product of modules, $\Ind_{\CR_\alpha\otimes\CR_{\alpha'}}^{\CR_{\alpha+\alpha'}} M\otimes N$. Then
\[K_0(\Mod_\te{fg} \CR)\cong {{U}}_q^+(\glie).\]   The KLR algebras were originally shown to be nontrivial in \cite{KL09} by showing that they act faithfully on 

Moreover, certain modules over the KLR decategorify to certain special bases of the quantum group. Khovanov-Lauda \cite{KL09} conjectured and Varagnolo-Vasserot and Rouquier \cite{varagnolo2011canonical},\cite{rouquier2012quiver} proved that Grothendieck classes of the projective indecomposables and the simple modules correspond to the Lusztig canonical and dual canonical bases, respectively. 

It was later shown by \cite{kleshchev2013affine} (and \cite{kleshchevloubert2015affine} in all finite types) that the KLR algebra $\CR_\alpha$ admits a affine cellular structure in the sense of \cite{koenig2012affine}, and that the representation theory of KLR in type $A$ is affine highest weight (or affine quasi-hereditary) in the sense of \cite{Kleshchev_2015}. This was also shown by \cite{BrundanKleshchevMcNamara_2014},\cite{Kato_2014}. \cite{BrundanKleshchevMcNamara_2014},\cite{Kato_2014},\cite{McNamaraKLRI} then showed that the standard modules and proper standard modules for this structure decategorify to the PBW and dual PBW bases, respectively. 

Shortly after \cite{KL09}, Varagnolo-Vasserot \cite{varagnolo2011canonical} proved Khovanov-Lauda's conjecture that the projective indecomposables correspond to the canonical bases by constructing the KLR algebra as equivariant Ext algebras of the direct sum of Lusztig sheaves on a certain quiver representation space. This perverse/constructible geometry is not something we consider in the present paper; however it would be very interesting to study the relationship between the two.

Much later, Elias-Qi \cite{EQ23} constructed a $\sl_2$-action on KLR which acts by derivations, but they note that there was not yet a geometric explanation for this action. It is natural to search then for a geometric origin of this action; we do so in this paper. Elias-Qi \cite{EQ23} considered the notion of an ``$\sl_2$-core'', namely the maximal finite-dimensional $\sl_2$-submodule. They showed that this $\sl_2$-core of KLR is a (monoidal) subalgebra, and they conjectured that it injects into the semisimplification of KLR, which is to say that it intersects the graded Jacobson radical trivially. Henceforth we refer to this conjecture as the \textit{Elias-Qi conjecture}. 

Grlj-Lauda \cite{lauda2025action} recently constructed a half-Witt algebra action on the KLR (in fact the construct it on the KLR 2-category); in one color, there exist parameters for which their action agrees with that of Elias-Qi \cite{EQ23}, but even in two colors, and even when restricting to a $\sl_2$-action, the two actions do not match. We construct a half-Witt algebra action on KLR in this paper differing from the \cite{lauda2025action} one; our action agrees with the action from \cite{EQ23}, and moreover can be extended to a full Witt algebra action after localizing the KLR appropriately.

Part of the significance of such $\sl_2$ or Witt actions is their ties to link homology. Khovanov-Rozansky showed in \cite{khovanov2016positive} that the half-Witt $\Wlie_{\ge 0}$ acts on triply graded link homology, and later \cite{guerinroz2025action},\cite{qi2022symmetries},\cite{qi2023symmetries} showed that this action arises from a half-Witt action on $\gl_N$-foams. The action of \cite{lauda2025action} is compatible with foamation, but differs from the \cite{EQ23} action; it is hence natural to study the relationship between the Witt action we give in this paper (and indeed the $\sl_2$-action of \cite{EQ23}) and link homology, but we have not done so in this paper.

\subsection{Main results}
We work in type $A_N$; let $I$ denote the index set of vertices for the corresponding quiver, and let us fix $\alpha=\sum_a n_a\alpha_a \in \BN I$. Let 
\[U\coloneqq \prod_{a\in I}\BA^{n_a}\subset \prod_{a\in I} \BP^{n_a}\eqqcolon X.\]
We produce a coherent sheaf, $\CE$, which we dub the ``coherent KLR sheaf'', on the product of projective spaces $X$, by using Serre's twisting sheaves on copies of $\BP^1$. To be more precise, we construct a sheaf $\ol\CP$ on $Y=(\BP^1)^n$ as a direct sum of box products of certain line bundles; then $\CE$ is the subsheaf of $\shend(\pi_*\ol\CP)$ preserving a certain sheaf $\CQ$, where $\pi\colon Y\lto X$ is the symmetrizing map. The line bundles $\CO(k)$ of $\BP^1$ admit a natural $\sl_2$-action, which produces a natural $\sl_2$-action on $\ol\CP$ and therefore $\shend(\pi_*\ol\CP)$; it turns out $\CE$ is preserved under this action.
\begin{mytheorem}{A}[\ref{thm:E=R}, \ref{thm:A_N monoidal}, \ref{thm:A_N sl2 action}, \ref{thm:globalsectioniscore}]\label{thm:A}
    The coherent sheaf of algebras $\CE$ on $X=\prod_{a\in I} \BP^{n_a}$ has
    \[\CE(U)\cong\CR_\alpha,\]
    and there is a geometric monoidal product on $\CE$ corresponding to the usual monoidal product on $\CR_\alpha$. As a result, one may identify the image of $\CR_\alpha$ inside $\End_{\Sym}(\Pol)$ as 
    \[\CR_\alpha=\{\phi\in\End_{\Sym}(\Pol):\phi(Q)\subseteq Q\},\]
    where $Q=\bigoplus_\beta Q_\beta\Pol_\beta$ for certain polynomials $Q_\beta$.
    
    Moreover there is a natural $\sl_2$-action $\sl_2\actson \CE$ coming from the Borel-Weil-Bott $\sl_2$-actions on line bundles over $\BP^1$; over the affine chart $U$, this recovers the $\sl_2$-action of Elias-Qi \cite{EQ23}, and over the entire space $X$ we obtain the $\sl_2$-core:
    \[\CE(X)=\Cor_{\sl_2}(\CR_{\alpha}).\]
\end{mytheorem}

Here it is worth noting that the affine chart $U$ is \textit{not} $\SL_2$-stable! Nevertheless, $\CE(U)$ inherits a natural $\sl_2$-action. Indeed, for any algebraic group $G$ over $\BC$ acting on a scheme $X$ and $\CF$ a $G$-equivariant $\mathcal O_X$-module, differentiating the equivariant structure gives an action of $\mathfrak g=\Lie(G)$ by first-order differential operators on $\CF$. Consequently, $\Gamma(U,\CF)$ is endowed with a natural $\mathfrak g$-action for every open subset $U\subset X$, whether or not $U$ is $G$-stable.

In fact, by using this (coherent) geometric interpretation, we were able to extend the $\sl_2$-action of Elias-Qi further.
\begin{mytheorem}{B}[\ref{thm:halfwittacts}, \ref{thm:fullwittacts}]\label{thm:B}
    There is a half-Witt-action $\Wlie_{\ge-1}\actson \CR_\alpha$ extending the $\sl_2$-action of Elias-Qi \cite{EQ23}. 

    By inverting the dots of KLR to obtain a diagrammatic algebra $\CR_\alpha^\times$ (which one might call a ``Laurent KLR algebra'') with ``negative/hollow dots'', one can extend the action $\Wlie_{\ge-1}$ to $\Wlie\actson \CR_\alpha^\times$. 
\end{mytheorem}

Finally, the representation theory of KLR algebras of type $A$ can be very nicely phrased in terms of the affine cellular structure of \cite{kleshchev2013affine}. We give a geometric analogue of this as well.
\begin{mytheorem}{C}[\ref{thm:M is Verma on affine}, \ref{thm:cellular E}]\label{thm:C}
    Fix a total order on the simple roots $\alpha_1>\alpha_2>\cdots>\alpha_N$, giving a lexicographic order $\pi_0>\cdots>\pi_\ell$ on the root partitions of $\alpha$. There is a filtration
    \[0\subset\CE_{\ge\pi_0}\subset\cdots\subset\CE_{\ge\pi_\ell}=\CE\] 
    such that evaluating on the affine chart $U$ gives the affine cellular filtration of Kleshchev-Loubert-Miemietz \cite{kleshchev2013affine}. As a result, one may similarly identify the image of the cell ideal inside $\End_{\Sym}(\Pol)$ as
    \[\CR_{\ge\pi}=\{\phi\in\End_{\Sym}(\Pol):\phi(Q)\subseteq Q_{\ge\pi}\}\] 
    for a certain explicit subspace $Q_{\ge\pi}$.
    
    Moreover this filtration is stable under the $\sl_2$-action, as well as under the $\Wlie_{\ge-1}$- or $\Wlie$-action, when restricted to appropriate affines.
\end{mytheorem}

While we were unable to prove the Elias-Qi conjecture in this paper, we believe the results of this paper bring us closer to a proof, because our coherent geometry elucidates where Elias-Qi's $\sl_2$-action is coming from. 

Geometrically, the significance of our construction is as follows. The affine chart $U$ is naturally a space of colored effective divisors on $\BA^1$. Indeed, we have \[U=\prod_{a\in I}\BA^{n_a}\cong \prod_{a\in I}\Sym^{n_a}\BA^1.\] So  a point of $U$ records, for each color $a$, an unordered collection of $n_a$ points of $\BA^1$, or, equivalently, a monic polynomial of degree $n_a$ whose roots are precisely those points.

This description also gives a precise connection with the construction of Varagnolo--Vasserot. The coordinate ring of $U$ is naturally identified with the equivariant cohomology ring \[\CO(U)\cong H^\bullet_{\prod_{a\in I}\GL_{n_a}}(\text{pt}; \BC).\] Concretely, the generators of the left-hand side are elementary symmetric functions in the points of each colored divisor, whereas the generators of the right-hand side are the universal Chern classes: the $i$th elementary symmetric function is mapped to the $i$th Chern class under this identification, and both have degree $2i$ (under the KLR grading). Varagnolo--Vasserot \cite{varagnolo2011canonical} realize $\CR_\alpha$ as an equivariant Ext algebra over the right-hand side. Our construction compactifies the spectrum of this ring, rather than the quiver representation space itself on which their Lusztig sheaves live. In particular, since $\BP^{n_a}\cong \Sym^{n_a}\BP^1$, we may regard \[X\cong \prod_{a\in I}\Sym^{n_a}\BP^1\] as the corresponding projective compactification, obtained by allowing the points of the colored divisors to go to $\infty$. 

The color-wise addition maps on symmetric powers, which take unions of
effective divisors, together with the necessary line bundle twists,
provide the geometric origin of the monoidal product in Theorem \ref{thm:A}. Also, the diagonal $\SL_2$-action
is induced by the usual action of $\SL_2$ on $\BP^1$ by projective
changes of coordinate. The equality
\[
\CE(X)=\Cor_{\sl_2}(\CR_\alpha)
\]
in Theorem \ref{thm:A} says that the locally finite part of $\CR_\alpha$ under the $\sl_2$-action consists precisely of the elements which extend regularly over the entire compactified space. Similarly, the half-Witt and full Witt actions of
Theorem \ref{thm:B} arise from polynomial vector fields on $\BA^1$ and Laurent vector fields on $\BG_m$, respectively, together with their induced actions on the relevant line bundles. Finally, the geometric content of Theorem \ref{thm:C} is that the KLM cellular filtration is realized by support conditions on colored divisors: the cell layer indexed by a root partition is supported where the coincidences prescribed by its interval roots occur.

\begin{remark} While we work over $\BC$ throughout this article, essentially all of the results continue to hold over other fields. In particular, the construction of $\CE_\alpha$, its
identification with $R_\alpha$ on the dominant affine chart, the
(geometric) monoidal product, and the cellular filtration are defined
integrally and hence extend over an arbitrary field $k$. Likewise, the 
$\SL_2$-linearization is independent of the characteristic or field.

However, the identification of global sections with the $\sl_2$-core requires a modification in positive
characteristic. By viewing $R_{\alpha,k}$ inside the meromorphic sections of
$\CE_{\alpha,k}$, let $\Cor_{\SL_2}(R_{\alpha,k})$ denote the sum of its
finite-dimensional algebraic $\SL_{2,k}$-subrepresentations. Equivalently, in the present geometric setting this is the locally finite part for the full distribution algebra $\operatorname{Dist}(\SL_{2,k})$, including its divided power operators. The same
geometric argument then gives
\[\CE_{\alpha,k}(X_{\alpha,k})=\Cor_{\SL_2}(R_{\alpha,k}).
\]
In characteristic zero, $\operatorname{Dist}(\SL_2)=U(\sl_2)$, so this recovers the usual $\sl_2$-core. In characteristic $p$, using the distribution algebra is essential, since Frobenius $p$th powers are invisible to the ordinary infinitesimal $\sl_2$-action but are detected by its divided power operators. Finally, the half-Witt and
full Witt actions extend in their present normalization when
$\operatorname{char}(k)\ne 2$.
\end{remark}

\subsection{The motivating example for $\NH_1$}
To give the reader a very brief glimpse of our story, we give here the example in type $A_1$ with 1 string. This example is admittedly very trivial, perhaps even too trivial to glean anything useful; but this is where our journey started, and so perhaps it is worth sharing. We will not define the KLR or the nil-Hecke yet, as the definition is not necessary for this example.

Let us consider $\NH_1=\BC[x]$; following the notation of \cite{KMZ}, we denote by $\wh\NC_1$ its graded semisimplification. $\BC[x]$ can be considered as the sections over $\BA^1$ of the sheaf of endomorphisms of $\CO_{\BP^1}$. We would like to witness that 
\[\wh\NC_1=\BC=\BC[x]\cap\BC[x^{-1}]=\NH_1^+\cap\NH_1^-\subset\BC[x,x^{-1}]=\NH_1^\circ.\]

Let us consider two copies of $\BA^1$, $U^+=\Spec\BC[x]$ and $U^-=\Spec\BC[x^{-1}]$, which should be thought of as Spec of symmetric polynomials in one variable. These spaces can be glued together along $U^\circ=\Spec\BC[x,x^{-1}]$ to form $X=\BP^1$. 

The rank $1!$ free sheaf $\CO_{U^\pm}$ on $U^\pm=\Spec\BC[x^{\pm 1}]$ has the property that
\[\shom(\CO_{U^\pm},\CO_{U^\pm})(U^\pm)=\NH_1^\pm.\] 
These two sheaves $\CO_{U^\pm}$ can be glued together with the trivial transition function to give the structure sheaf $\CO_{\BP^1}$, which has sections over $U^+\cap U^-$
\[\CO_{\BP^1}(U^+\cap U^-)=\BC[x,x^{-1}]=\NH_1^\circ.\] 
Then the global sections of $\CO_{\BP^1}$ are those which can be defined on both $U^+$ and $U^-$, where the intersection is taken inside sections over $U^+\cap U^-$, which in turn lives inside the stalk at the generic point:
\[\CO_{\BP^1}(\BP^1)=\CO_{\BP^1}(U^+)\cap\CO_{\BP^1}(U^-)\subset \CO_{\BP^1}(U^+\cap U^-)\subset \CO_{\BP^1,\eta}.\] 
This geometrically realizes that $\wh\NC_1=\NH_1^+\cap \NH_1^-\subset\NH_1^\circ$. 

\subsection{Outline}
In the interest of making this paper accessible and easy to read, we have split the paper into two parts. In Part \ref{part:I} we discuss the 1-color case, namely that of the nil-Hecke; only in Part \ref{part:II} do we tackle the full multi-color KLR. 

\subsection{Future directions}
There are many interesting follow-up questions one could ask, which we intend to study in future papers. What are monoidal generators and relations for the $\sl_2$-core, $\CE(X)$? Could one say something about algebraic stratifications from this point of view, and perhaps even find a Koszul nil-algebra such as in \cite{zhou2024bgg,zhou2026bgg}? Could the structures in this paper be used to carry out the minimal categorification program of \cite{KMZ} in higher rank? Could the coherent program of the present paper be extended to other (affine cellular) algebras appearing in categorification, such as Soergel calculus? Does this coherent geometry give rise to other Lie algebra actions? What is the relationship to Hodge theory? What is the relationship to the established geometric constructions of KLR, such as \cite{varagnolo2011canonical}? Perhaps the answers to many of these questions are related to one another. 

\subsection{Acknowledgements} 
We would like to thank Mikhail Khovanov and Joshua Sussan for providing many insights. We would also like to thank Alvaro Martinez, whose past work with F.Z. motivated parts of this paper. We would also like to thank Pablo Alvarez for introducing us to Kazhdan-Laumon and for very helpful discussions; Amal Mattoo for some helpful discussions regarding $\BP^1$; You Qi for a helpful discussion regarding ties to the existing geometric theory; and Felix Roz for some discussions on other $\sl_2$-actions. We would also like to thank Ben Elias and Sabin Cautis for helpful comments on a draft.

\subsubsection{AI disclosure}
We made use of LLMs in developing and checking proofs for this paper to great effect. Their assistance substantially accelerated the project, allowing us to complete in weeks work that would normally have taken months. The AI models we used were\footnote{Ranked here chronologically, but also in increasing order of effectiveness.}: Gemini 3.1 Pro, Fable 5 (for the few days it was available to us), Opus 5, and ChatGPT 5.6 Sol (and 6 Astra for a few days). With the exception of the idea of considering the space $Z_\pi$, which was suggested by Fable\footnote{Interestingly, though Fable wrote down the spaces $Z_\pi$ and the maps $\mu_\pi$, it did not realize the significance was the relation to the `depth' ring $\Lbd_\pi$ of the affine cellular structure of KLM, which we later realized.}, all key ideas in this paper were human-generated, from overarching philosophies, such as the consideration of coherent sheaves on projective spaces and the construction of the geometric monoidal product using concatenation, to technical philosophies, such as the construction of sheaves such as $\wh\CP$ from reverse-engineering the $\sl_2$-action and the construction of the KLR sheaf inside the endomorphism sheaf $\shend(\varpi_*\wh\CP)$ as the subsheaf preserving $\varpi_*\wh\CQ$. The proofs with nontrivial AI input are the proofs of Proposition \ref{prop:Q as cap} and Lemma \ref{lem:verma is generically simple}. Proposition \ref{prop: torsion_free_core_is_glob_sect} was also originally written to only handle what is needed for this paper; we fed this argument into AI and asked it to generalize it as much as possible, and the resulting proof it suggested is what is included here. Notably, the only nontrivial AI input in that proof was the suggestion to use Tag 01PQ from the Stacks Project; the key idea of translating by elements of the group and observing the connection with the global sections is due to us. We are, of course, solely responsible for the correctness of the paper, and all text in the paper is human-generated.  


\part{The $A_1$ case: nil-Hecke}\label{part:I}
In this Part we briefly describe our coherent story in the simplest case, namely that of 1-color KLR, also known as the nil-Hecke. This Part is considerably shorter and simpler than Part \ref{part:II}, where the full multi-color case is treated, but we hope that this simpler Part 1 will be more convincing in showcasing some of our ideas. Some proofs will be deferred to Part \ref{part:II} for a full proof of the general case. 

Let us briefly recall the nil-Hecke. The nil-Hecke algebra, denoted $\NH_n$, is a non-commutative algebra defined via generators and relations. In terms of algebraic symbols, it is generated by $\{\pd_i,x_i:1\le i\le n\}$ subject to the relations
\begin{align*}
    \pd_i x_j &= x_j\pd_i \quad \text{if $j\neq i,i+1$}, & x_i \pd_i - \pd_i x_{i+1}&=1,\\
    \pd_i\pd_j &= \pd_j\pd_i \quad \text{if $|i-j|>1$}, & \pd_i x_i - x_{i+1} \pd_i &=1,\\
    x_i x_j &=   x_j x_i, & \pd_i\pd_{i+1}\pd_i &= \pd_{i+1}\pd_i\pd_{i+1},\\
     & &\pd_i^2&= 0.
\end{align*}
However it is often easier to work with this algebra diagrammatically; that is to say, we draw algebra elements as linear combinations of pictures, and treat multiplication as stacking pictures on top of one another\footnote{Our convention will be that $xy$ means $x$ goes on top of $y$.}. The generators would be drawn as
\[ \hackcenter{\begin{tikzpicture}[scale=0.375]
        \node at (-2,1) {$x_i \, = \, $}; 
        \draw (0,0)--(0,2);
        \node at (1.6,1) { $\cdots$};
        \draw (3,0)--(3,2);
        \draw (5,0)--(5,2);
        \draw(7,0)--(7,2);
        \node at (8.6,1) {$\cdots$};
        \draw (10,0)--(10,2);
        \fill (5,1) circle (6pt);
        \node at (5,-0.5) {\tiny $i^\te{th}$};
    \end{tikzpicture}}\ ,\qquad 
    \hackcenter{\begin{tikzpicture}[scale=0.375]
      \node at (-2,1) {$\pd_i \, = \, $};
        \draw (0,0)--(0,2);
       \node at (1.6,1) { $\cdots$};
        \draw (3,0)--(3,2);
        \draw (5,0)--(7,2);
        \node at (5,-0.6) {\tiny $i^\te{th}$};
        \draw(7,0)--(5,2);
        \node at (7,-0.6) {\tiny $(i\!+\!1)^\te{th}$};
        \draw(9,0)--(9,2);
        \node at (10.6,1) {$\cdots$};
        \draw (12,0)--(12,2);
    \end{tikzpicture}},\]
and some of the relations would be drawn as:
\begin{align*}
    \begin{diagram}
        \draw (0,0) ..controls(2.5,2).. (0,4);
        \draw (2,0)..controls(-0.5,2)..(2,4);
    \end{diagram}
    &=0,\qquad \begin{diagram}
        \draw (0,0)--(4,4);
    \draw (4,0)--(0,4);
    \draw (2,0) ..controls(-0.5,2).. (2,4);
    \end{diagram}
    =
    \begin{diagram}
        \draw (0,0)--(4,4);
    \draw (4,0)--(0,4);
    \draw (2,0) ..controls(4.5,2).. (2,4);
    \end{diagram},\qquad 
    \begin{raisediagram}[-0.5em]
        \node at (0,1.5) {\tiny $\cdots$};
        \draw (3-2,0) ..controls(4-2,1.5).. (4-2,3);
        \draw (4-2,0)..controls(3-2,1.5)..(3-2,3);
        \node at (3,1.5) {\tiny $\cdots$};
        \draw (8-4,3)..controls(9-4,1.5)..(9-4,0);
        \draw (9-4,3)..controls(8-4,1.5)..(8-4,0);
        \node at (6,1.5) {\tiny $\cdots$};
        \node at (1,-0.5) {\tiny $i^\te{th}$};
        \node at (4,-0.5) {\tiny $j^\te{th}$};
    \end{raisediagram}=\begin{raisediagram}[-0.5em]
        \node at (0,1.5) {\tiny $\cdots$};
        \draw (3-2+3,0) ..controls(4-2+3,1.5).. (4-2+3,3);
        \draw (4-2+3,0)..controls(3-2+3,1.5)..(3-2+3,3);
        \node at (3,1.5) {\tiny $\cdots$};
        \draw (8-4-3,3)..controls(9-4-3,1.5)..(9-4-3,0);
        \draw (9-4-3,3)..controls(8-4-3,1.5)..(8-4-3,0);
        \node at (6,1.5) {\tiny $\cdots$};
        \node at (1,-0.5) {\tiny $i^\te{th}$};
        \node at (4,-0.5) {\tiny $j^\te{th}$};
    \end{raisediagram},
\end{align*}
\[\begin{diagram}
    \draw(0,0)--(2,2);
    \draw(0,2)--(2,0);
    \fill (0.4,1.6) circle (6pt);
\end{diagram}\:-\:\begin{diagram}
    \draw(0,0)--(2,2);
    \draw(0,2)--(2,0);
    \fill (1.6,0.4) circle (6pt);
\end{diagram}\ =\ \begin{diagram}
    \draw(0,0)--(2,2);
    \draw(0,2)--(2,0);
    \fill (0.4,0.4) circle (6pt);
\end{diagram}\:-\:\begin{diagram}
    \draw(0,0)--(2,2);
    \draw(0,2)--(2,0);
    \fill (1.6,1.6) circle (6pt);
\end{diagram}\ =\ \begin{diagram}
    \draw(0,0)--(0,2);
    \draw(2,0)--(2,2);
\end{diagram}\ .\]

There is moreover a ``monoidal product'' on the nil-Hecke algebras,
\[\otimes\colon \NH_n\otimes\NH_m\lto\NH_{n+m},\]
which diagrammatically simply places two diagrams horizontally side-by-side. Algebraically this sends for instance $\pd_i\otimes 1\lmto \pd_i$ and $1\otimes \pd_j\lmto \pd_{n+j}$. The induction tensor product along this monoidaly product, namely $\Ind_{\NH_n\otimes\NH_m}^{\NH_{n+m}}M\otimes N$, categorifies the multiplication in the quantum group $\dot{{U}}_q^+(\sl_2)$. 


In \cite{KMZ} a finite-dimensional semisimplification $\wh\NC_n$, standing for ``(double) nil-Coxeter'', of $\NH_n$ was constructed as a subalgebra, $\wh\NC_n\linj \NH_n$. This is so-called because it is generated by both positive crossings and ``negative crossings'', namely by $\pd_i^+=\pd_i$ and $\pd_i^-=-x_i\pd_ix_{i+1}$. It was shown that this subalgebra is moreover the $\sl_2$-core of \cite{EQ23}. 


In this Part we will construct the nil-Hecke algebra $\NH_n$ as affine sections of a coherent sheaf on $\BP^n$.

\section{The spaces and the sheaves}
Here we set the stage for our coherent geometry and introduce the relevent coherent objects. To avoid too many parentheses, we will often write for instance $\BP^{1,\times n}$ for $(\BP^1)^n$. 

\subsection{The base space}
Let $x_1,\cdotsc,x_n$ denote the dot variables for nil-Hecke. Though in principle these variables live on a cover space, for this exposition it is easier to start on the base. Let $\e_i$ denote the usual $i$-th elementary symmetric polynomial in the variables $x_1,\cdotsc,x_n$, and let $\e_i^-=\e_i(x_1^{-1},\cdotsc,x_n^{-1})$ denote the $i$-th elementary symmetric polynomial in the variables $x_1^{-1},\cdotsc,x_n^{-1}$. Let
\begin{align*}
    \Sym_n=\Sym_n^+&=\BC[\e_1,\cdotsc,\e_n],\\
    \Sym_n^-&=\BC[\e_1^-,\cdotsc,\e_n^-],\\
    \Sym_n^\times&=\BC[\e_1,\cdotsc,\e_n,\e_n^-].
\end{align*}
Note that
\begin{align*}
    \Sym_n^-&=\BC\Big[\frac{\e_{n-1}}{\e_n},\cdotsc,\frac{1}{\e_n} \Big].
\end{align*}

Let
\[X_n^+=\Spec\Sym_n^+,\qquad X_n^-=\Spec\Sym_n^-,\]
and let us glue these spaces along
\[X_n^\times=\Spec \Sym_n^\times\] 
to obtain
\[X_n^\circ=X_n^+\cup_{X_n^\times} X_n^-,\] 
a space covered by two affine charts.
\begin{proposition}\label{prop: Xn description}
    The space $ X_n^\circ$ can be described as
    \[ X_n^\circ=\BP^n\sm\BP^{n-2},\]
    where the removed copy of $\BP^{n-2}\subset \BP^n=\Proj\BC[\E_0,\cdotsc,\E_n]$ corresponds to the subspace where $\E_0=\E_n=0$. 
\end{proposition}
\begin{proof}
    Recall that $\BP^n$ consists of tuples $(a_0\colon\cdots\colon a_n)$ such that at least one $a_i$ is non-zero. The locus where $a_0\ne 0$ (resp. where $a_n\ne 0$) can be identified with $X_n^+$ (resp. with $X_n^-$). Their intersection is moreover precisely $X_n^\times$, and the complement of their union (where $a_0 = a_n = 0$) is isomorphic to a copy of $\BP^{n-2}$, as desired.
\end{proof}

Hence $X_n^\circ$ naturally is a open subvariety of
\[X_n\coloneqq \BP^n=\Proj\BC[\E_0,\cdotsc,\E_n],\] 
where we have chosen homogeneous coordinates $\E_i$, with the convention that
\[\e_i=\frac{\E_i}{\E_0}.\] 
Frequently we refer to $X_n^+$ as the ``dominant affine chart'' of $X_n^\circ$ and $X_n$.

Note well that $X_n^\circ $ is obtained from $X_n$ by removing a codimension-2 space; hence, when we work with line bundles on $X_n^\circ$, we will appeal to Hartog's theorem and simply consider $X_n$ instead.

\subsection{The cover space}
Let 
\begin{align*}
    \Pol_n=\Pol^+_n&=\BC[x_1,\cdotsc,x_n],\\
    \Pol_n^-&=\BC[x_1^{-1},\cdotsc,x_n^{-1}],\\
    \Pol_n^\times&=\BC[x_1,\cdotsc,x_n,x_1^{-1},\cdotsc,x_n^{-1}].
\end{align*}
Let
\[Y_n^+=\Spec\Pol_n^+,\qquad Y_n^-=\Spec\Pol_n^-,\] 
and let us glue them along 
\[Y_n^\times=\Spec\Pol_n^\times\]
to obtain the space
\[Y_n^\circ =Y_n^+\cup_{Y_n^\times} Y_n^-.\] 
Similarly to the base space case, this naturally sits inside 
\[Y_n=\BP^{1,\times n},\] 
as we shall see.
\begin{proposition}
    The space $ Y_n^\circ$ can be described as
    \[ Y_n^\circ=\BP^{1,\times n}\mathbin{\bigg\backslash}\bigcup_{n(n-1)} \BP^{1,\times n-2},\] 
    where the removed space is the union of points $(p_1,\cdotsc,p_n)$ such that $p_i=0,\ p_j=\infty$ for some $i,j$.

    Moreover, the natural $S_n$-cover $\pi\colon \BP^{1,\times n}\lto\BP^n$ restricts to an $S_n$-cover of $\pi\colon  Y_n^\circ \lto  X_n^\circ$.
\end{proposition}
\begin{proof}
    Recall that the map $\BP^{1,\times n}\to\BP^n$ is given by taking a tuple $([a_0\colon b_0],\ldots, [a_n\colon b_n])$ to the $n+1$ coefficients of the polynomial $(a_0t+b_0)\cdots(a_nt+b_n)$. The preimage of $ X_n^\circ$ under this map consists of tuples of points such that $a_0\cdots a_n$ and $b_0\cdots b_n$ are not simultaneously zero. The complement of this in $Y_n$ consists of tuples of points such that there are distinct indices $i$ and $j$ with $a_i=0$ and $b_j=0$. The result follows.
\end{proof}
Note again that $Y_n^\circ$ is obtained from $Y_n$ by removing a codimension-2 subspace. We may also think of $x_i$ as the coordinate on the $i$-th factor of $\BP^1$ in $Y_n$. 

Let
\[\pi\colon Y_n\lto X_n\cong \Sym^n\BP^1 \] 
denote the symmetrization map, which we have seen restricts to a cover of $X_n^\circ$ by $Y_n^\circ$. 

\subsection{The sheaves above}
By Hartog's theorem, any line bundle on a smooth variety defined away from a space of codimension 2 can be uniquely extended to the whole space. Hence, to define a line bundle on $Y_n^\circ$ is equivalent to defining a line bundle on $Y_n=\BP^{1,\times n}$. 

Define the line bundle on $Y_n$
\[\CV_n\coloneqq \CO_{\BP^1}(n-1)\boxtimes\CO_{\BP^1}(n-2)\boxtimes\cdots\boxtimes\CO_{\BP^1}.\]
Note that the space of global sections
\[V_n\coloneqq \Gamma(\CV_n)\cong\BC\{ x_1^{k_1}\cdots x_n^{k_n}: 0\le k_i\le n-i\}\] 
is a space of dimension $n!$. 

\begin{remark}
    It is worth noting that the involution $\tau$ of equation 13 in \cite{KMZ} can be geometrically interpreted as the transition function of sections of $\CV_n$; this is part of how we guessed the sheaf. 
\end{remark}

\subsection{The sheaves below}
$ X_n^\circ$ is obtained from $X_n=\BP^n$ by removing a space of codimension 2, so as above, by Hartog's theorem it suffices to consider bundles on $X_n$. 

Since the base space is constructed from symmetric polynomials, it is natural to try to construct a free sheaf on $X_n$ of rank $n!$, mimicking the realization of $\Pol_n$ as a free $\Sym_n$-module of rank $n!$. Hence let us define the trivial vector bundle on $X_n$
\[\CP_n\coloneqq \CO_{\BP^n}\otimes \Gamma(\CV_n).\] 
Note $\Gamma(\CV_n)=\Gamma(\CP_n)=V_n$.
\begin{proposition}
    The pushforward sheaf $\pi_*\CV_n$ on $X_n$ is a degree $n!$ vector bundle and the natural evaluation map $\CP_n\lto \pi_*\CV_n$ is an isomorphism.
\end{proposition}
\begin{proof}
    The first claim is clear because $\pi$ is finite of degree $n!$. For the second claim, it suffices to show that $\pi_* \CV_n$ is isomorphic to $\CO_{\BP^n}^{\oplus n!}$, since any surjective map of vector bundles of the same rank are necessarily isomorphic.

    We first use Horrock's criterion to verify that $\pi_* \CV_n$ splits into line bundles, i.e. for all $1\le i \le n-1$ and $k\in \BZ$, we require $H^i(\BP^n,\pi_*\CV_n(k))=0$. By the projection formula, \[\pi_*\CV_n(k) = \pi_*(\CV_n \otimes \pi^*\CO_{\BP^n}(k)) = \pi_*(\CO_{\BP^1}(n+k-1)\boxtimes\CO_{\BP^1}(n+k-2)\boxtimes\cdots\boxtimes\CO_{\BP^1}(k)),\] and using that the higher pushforwards along a finite morphism vanish, we have \[H^i(\BP^n,\pi_*\CV_n(k)) = H^i(\BP^{1,\times n}, \CO_{\BP^1}(n+k-1)\boxtimes\CO_{\BP^1}(n+k-2)\boxtimes\cdots\boxtimes\CO_{\BP^1}(k)).\] By the K\"unneth formula, this cohomology group splits into a tensor product of cohomology groups of line bundles on $\BP^1$, and we claim that there is always a vanishing term. Indeed, for $k\ge 0$, since $i\ge 1$, there is always a cohomology group of degree at least $1$, which vanishes on any line bundle on $\BP^1$ of degree at least $k\ge 0$. For $k \le -n-1$, since $i \le n-1$, there is always a cohomology group of degree zero, which vanishes on any line bundle on $\BP^1$ of negative degree. For $-n \le k \le -1$, the consecutive sequence of integers $n+k-1, \ldots, k$ contains the integer $-1$, and we have $H^j(\BP^1, \CO_{\BP^1}(-1))=0$ for any $j$. 
    
    Since every line bundle on $\BP^n$ is of the form $\CO_{\BP^n}(a)$ for an integer $a$, we may then write $\pi_*\CP_n = \bigoplus_{i=1}^{n!} \CO_{\BP^n}(a_i)$, where the $\CO_{\BP^n}(a_i)$ are line bundles on $\BP^n$. Twisting by $-1$ and taking global sections, we have 
    \begin{align*}
        H^0(\BP^n,\pi_*\CV_n(-1))&=H^0(\BP^1,\CO_{\BP^1}(n-2))\otimes \cdots \otimes H^0(\BP^1,\CO_{\BP^1}(-1))=0,\\
        &=H^0\left(\BP^n,\bigoplus_{i=1}^{n!}\CO_{\BP^n}(a_i-1)\right)=\bigoplus_{i=1}^{n!}H^0(\BP^n,\CO_{\BP^n}(a_i-1)),
    \end{align*}
    which means that $a_i\le 0$ for all $i$. 

    Finally, $H^0(\BP^n,\pi_*\CV_n)$ has rank $n!$, so we must have $a_i\ge 0$ for all $i$ too. The result follows.
\end{proof}

Note well that 
\[\CP_n(X_n^+)\cong \Sym_n\otimes V_n\cong \Pol_n,\] 
which is why we have chosen the symbol $\CP$ for it. 

\section{The nil-Hecke sheaf}
Recalling that $\NH_n=\End_{\Sym_n}(\Pol_n)$, it is then natural to define the sheaf
\[\CE_n\coloneqq\shend_{\CO}(\CP_n)=\CO_{\BP^n}\otimes \End_\BC V_n.\] 
Note that it is automatic that
\[\CE_n(X_n^+)=\End_{\Sym_n}(\Pol_n)=\NH_n,\] 
so that affine sections over the dominant affine chart $X_n^+$ recovers the classical nil-Hecke algebra. In fact, the affine sections over $X_n^-$ recover the \textit{negative} nil-Hecke algebra of \cite{KMZ}:
\[\CE_n(X_n^-)=\End_{\Sym_n^-}(\Pol_n^-)=\NH_n^-.\] 
\begin{theorem}\label{thm:nilhecke old main}
    The global sections of $\CE_n$ as
    \[\CE_n(X_n^\circ)=\CE_n(X_n^+)\cap \CE_n(X_n^-)\subset \CE_n(X_n^+\cap X_n^-)\subset \CE_{n,\eta},\] 
    where the intersection is taken inside the generic stalk, realizes the double nil-Coxeter $\wh\NC_n$ of \cite{KMZ} as a subalgebra,
    \[\wh\NC_n=\NH_n^+\cap\NH_n^-\subseteq \NH_n^\circ.\] 

    On the other hand, there is a surjection
    \[\CE_n(X_n^+)\lsurj \CE_n\rv_0\cong\CE_n(X_n^\circ)\] 
    realizing the double nil-Coxeter as a quotient. 
\end{theorem}
To summarize, one has the picture
\begin{center}
\begin{tikzcd}
    \CE_n(X_n^\circ)\arrow[hookrightarrow]{r}\arrow[swap]{dr}{\sim} & \CE_n(X_n^+)\arrow[twoheadrightarrow]{r} & \CE_n\rv_0\arrow{dl}{\sim}\\
    & \CE_n(X_n^+)^{\BG_m}\cong\Mat_{n!}\BC
\end{tikzcd}
\end{center}
\begin{remark}
    One way to think of things is as follows. There is a $\BG_m$-action on $\CP_n=\CO_{\BP^n}\otimes V_n$ acting only on $\CO_{\BP^n}$ such that both global sections $\CE_n(X_n^\circ)$ and the fiber at zero $\CE_n\rv_0$ can be realized as $\CO_{\BP^n}^{\BG_m}\otimes \End_\BC V_n\cong\Mat_{n!}\BC$, given by the space of $\BG_m$-invariants in the structure sheaf.
\end{remark}
\begin{proof}[Proof of Theorem \ref{thm:nilhecke old main}]
    This is immediate from the identification of the results 2.23 through 2.29 of \cite{KMZ}. 
\end{proof}
\begin{remark}
    It is possible to obtain diagrammatics for this story; this is done in Theorem \ref{thm:laurent KLR}. 
\end{remark}

\section{Monoidal product}\label{sec:1color monoidal}
Let \[\wh{\sha}_{n,m}\colon \BP^{1,\times n}\times\BP^{1,\times m}\xrightarrow{\sim}\BP^{1,\times(n+m)}\]
be concatenation of the ordered factors, and let
\[c_{n,m}\colon\BP^n\times\BP^m\longrightarrow\BP^{n+m}\]be the induced addition map on symmetric powers. Thus, $\pi_{n+m}\circ\wh{\sha}_{n,m}=c_{n,m}\circ(\pi_n\times\pi_m).$ Let
\[\mathcal L_{n,m}=\CO_{\BP^n}(m)\]and
\[\mathcal W_{n,m}=(\CP_n\otimes\mathcal L_{n,m})\boxtimes\CP_m.\]
Since $\pi_n^*\CO_{\BP^n}(1)\cong\CO(1,\ldots,1)$, we have $\wh{\sha}_{n,m}^*\CV_{n+m}\cong (\CV_n\otimes\pi_n^*\mathcal L_{n,m})\boxtimes\CV_m.$ Pushing forward and applying the projection formula gives \[c_{n,m,*}\mathcal W_{n,m}\cong\CP_{n+m}.\]

We then have the natural composition \[\mu_{n,m}\colon c_{n,m,*}(\CE_n\boxtimes\CE_m)\longrightarrow
c_{n,m,*}\shend(\mathcal W_{n,m})\longrightarrow\shend(c_{n,m,*}\mathcal W_{n,m})
\cong\CE_{n+m},\] where the second arrow is induced by the counit of $c_{n,m}^*\dashv c_{n,m,*}$.

\begin{theorem}\label{thm:A1-monoidal}
The morphisms
\[\mu_{n,m}\colon c_{n,m,*}(\CE_n\boxtimes \CE_m)\lto \CE_{n+m}\] 
are $\SL_2$-equivariant and associative, with $\CE_0=\BC$ as their
unit.  On the two standard affine charts they recover the usual
nil-Hecke homomorphisms
\[\NH_n^\pm\otimes\NH_m^\pm\longrightarrow\NH_{n+m}^\pm\]
given by horizontal concatenation.  On global sections they induce
\[\wh\NC_n\otimes\wh\NC_m\longrightarrow\wh\NC_{n+m}.\]
\end{theorem}

\begin{proof}
For either sign $\epsilon$, one has\[c_{n,m}^{-1}(X_{n+m}^\epsilon)=X_n^\epsilon\times X_m^\epsilon.\] On this chart $\mathcal L_{n,m}$ is trivial, and the preceding bundle
isomorphism identifies
\[\Pol_{n+m}^\epsilon\cong\Pol_n^\epsilon\otimes_{\BC}\Pol_m^\epsilon.\]
Under this identification, $\mu_{n,m}(x\otimes y)$ acts by $x$ on the
first $n$ variables and by $y$ on the last $m$ variables.  This is
exactly horizontal concatenation. Associativity follows on the dominant affine chart from associativity of horizontal concatenation, and hence globally because $\CE_{n+m+r}$ is torsion-free (in fact, locally free). The monoidal unit assertion is also clear from this description. Taking global sections gives the final assertion.

Note well that all calculations above in this section are $\SL_2$-equivariant; hence the monoidal product is also $\SL_2$-equivariant. 
\end{proof}

\section{$\sl_2$-action}\label{sec:1color sl2}
It is a classical result that $\sl_2$ acts on the twisting sheaves $\CO(n)$ of $\BP^1$ by
\begin{align*}
    e&\lmto -x^2\pd_x+nx,\\
    h&\lmto 2x\pd_x-n,\\
    f&\lmto \pd_x.
\end{align*}
Hence the sheaf $\CP_n=\CO(n-1)\boxtimes\cdots\boxtimes\CO(0)$, as well as its space of global sections $V_n=\Gamma(\CP_n)$, naturally receives an $\sl_2$-action. 

Via $\CP_n=\CO_{\wh X_n}\otimes V_n$, the sheaf $\CP_n$ then also receives an $\sl_2$-action. 
\begin{remark}
    We must warn the reader that this action $\sl_2\actson\CP_n$ is \textit{not} the usual action of $\sl_2\actson \Pol_n$; instead, it is twisted. To avoid confusion, we will use diagrammatics to denote algebra elements and algebraic symbols to denote module elements.
\end{remark}

The sheaf $\CE_n$ of endomorphisms then also receives an $\sl_2$-action 
    \[\sl_2\actson \CE_n=\shom_\CO(\CP_n,\CP_n)\] 
via $(X\cdot \phi)(\sq)=X\cdot (\phi(\sq))-\phi(X\cdot \sq)$. 

It turns out this action on $\CE_n$ is the geometric version of the action of Elias-Qi. 
\begin{theorem}\label{thm:sl2action}
    The action $\sl_2\actson \CE_n$ restricts to the Elias-Qi action $\sl_2\actson \CE_n(X_n^+)=\NH_n^+$.
\end{theorem}
\begin{proof}
    Since the monoidal product in Section \ref{sec:1color monoidal} is $\SL_2$-equivariant, it suffices to check the $\sl_2$-action on the monoidal generators, namely the dot and the crossing. We emphasize that we use algebraic symbols, such as $x_i$, for vectors in the module $\Pol$, and diagrammatic symbols for elements of the algebra $\CE$. 

    Let us first remark that the $\sl_2$-action on $\CP_1=\CO_{\BP^1}$ is simply the classical one, namely $e\lmto -x^2\pd_x$, $h\lmto 2x\pd_x$, $f\lmto \pd_x$. 

    Let us compute the $\sl_2$-action on the dot. This is the generator for $\NH_1$, which acts upon $\Pol_1$ which is of rank $1$ over $\Sym_1$, so to determine what $X\cdot\tikzdot$ is, it suffices to check the action on the only basis element which is $1\in V_1$. Let us compute:
    \begin{align*}
        (e\cdot \tikzdot)(1)&=e\cdot \tikzdot(1)-\tikzdot(e\cdot 1)=e\cdot x_1=-x_1^2\\
        &\hspace*{275pt}\implies e\cdot\tikzdot\:=-\:\begin{diagram}
            \draw(0,0)--(0,1.5);
            \fill(0,0.5)circle(6pt);
            \fill(0,1)circle(6pt);
        \end{diagram}\:,\\
        (h\cdot\tikzdot)(1)&=h\cdot\tikzdot(1)-\tikzdot(h\cdot 1)=h\cdot x_1=2x_1\\
        &\hspace*{275pt}\implies h\cdot\tikzdot\:=2\:\tikzdot\:,\\
        (f\cdot\tikzdot)(1)&=f\cdot\tikzdot(1)-\tikzdot(f\cdot 1)=f\cdot x_1=1\\
        &\hspace*{275pt}\implies f\cdot\tikzdot\:=\:\begin{diagram}
            \draw(0,0)--(0,1.5);
        \end{diagram}\:.
    \end{align*}

    Next let us compute the $\sl_2$-action on $\CP_2(X_2^+)$. Recall that elements of this are elements of $\pi_*(\CO(1)\boxtimes\CO)(X_2^+)$; in particular the $\sl_2$-action on the two factors will be different. $\CP_2(X_2^+)$ has two basis elements over $\CO_{\wh X_2}$, namely $1=\pi_*(1\boxtimes 1)$ and $x_1=\pi_*(x_1\boxtimes 1)$. We will suppress the $\pi_*$ below for brevity. Let us compute:
    \begin{align*}
        e\cdot 1&=(e\cdot_{\CO(1)} 1)\boxtimes 1+1\boxtimes (e\cdot_{\CO} 1)=x_1\boxtimes1+0\\
        &\hspace*{275pt}\implies e\cdot 1=x_1,\\ 
        e\cdot x_1&=(e\cdot_{\CO(1)} x_1)\boxtimes 1+x_1\boxtimes (e\cdot_{\CO} 1)=0+0\\
        &\hspace*{275pt}\implies e\cdot x_1=0,\\
        h\cdot 1&=(h\cdot_{\CO(1)} 1)\boxtimes 1+1\boxtimes (h\cdot_{\CO} 1)=-1\boxtimes 1+0\\
        &\hspace*{275pt}\implies h\cdot 1=-1,\\
        h\cdot x_1&=(h\cdot_{\CO(1)} x_1)\boxtimes 1+x_1\boxtimes (h\cdot_{\CO} 1)=2x_1\boxtimes 1+x_1\boxtimes (-1)\\
        &\hspace*{275pt}\implies h\cdot x_1=x_1,\\
        f\cdot 1&=(f\cdot_{\CO(1)} 1)\boxtimes 1+1\boxtimes (f\cdot_{\CO} 1)=0+0\\
        &\hspace*{275pt}\implies f\cdot 1=0,\\
        f\cdot x_1&=(f\cdot_{\CO(1)} x_1)\boxtimes 1+x_1\boxtimes (f\cdot_{\CO} 1)=1\boxtimes 1+0\\
        &\hspace*{275pt}\implies f\cdot x_1=1.
    \end{align*}
    
    Then we can compute the $\sl_2$-action on the crossing. Again, to know what $X\cdot\tikzcrossing$ is, it suffices to check the action on $1,x_1$, which are the basis elements in $V_2$. Compute:
    \begin{align*}
        (e\cdot\tikzcrossing)(1)&=e\cdot\tikzcrossing(1)-\tikzcrossing(e\cdot 1)=0-\tikzcrossing(x_1)=-1,\\
        (e\cdot \tikzcrossing)(x_1)&=e\cdot\tikzcrossing(x_1)-\tikzcrossing(e\cdot x_1)=e\cdot 1-\tikzcrossing(0)=x_1\\
        &\hspace*{275pt}\implies e\cdot\tikzcrossing=\:\begin{diagram}
            \draw(0,0)--(1.5,1.5);
            \draw(1.5,0)--(0,1.5);
            \fill (0.35,1.15) circle(6pt);
        \end{diagram}\:+\:\begin{diagram}
            \draw(0,0)--(1.5,1.5);
            \draw(1.5,0)--(0,1.5);
            \fill (1.15,0.35) circle(6pt);
        \end{diagram}\:;\\
        (h\cdot\tikzcrossing)(1)&=h\cdot\tikzcrossing(1)-\tikzcrossing(h\cdot 1)=0-\tikzcrossing(-1)=0,\\
        (h\cdot \tikzcrossing)(x_1)&=h\cdot\tikzcrossing(x_1)-\tikzcrossing(h\cdot x_1)=h\cdot 1-\tikzcrossing(x_1)=-1-1=-2\\
        &\hspace*{275pt}\implies h\cdot\tikzcrossing=-2\:\tikzcrossing\:;\\
        (f\cdot\tikzcrossing)(1)&=f\cdot\tikzcrossing(1)-\tikzcrossing(f\cdot 1)=0-\tikzcrossing(0)=0,\\
        (f\cdot\tikzcrossing)(x_1)&=f\cdot\tikzcrossing(x_1)-\tikzcrossing(f\cdot x_1)=f\cdot 1-\tikzcrossing(1)=0\\
        &\hspace*{275pt}\implies f\cdot \tikzcrossing=0.
    \end{align*}
    This concludes the check.
\end{proof}
The Cartan action above can be exponentiated to obtain a $\BG_m$-action on $\CE_n$ which recovers the KLR grading on $\NH_n=\CE_n(X_n^+)$.
\begin{remark}
    Note in the proof above that the Cartan action on $\Pol_n=\CP_n(X_n^+)$ does \textit{not} recover the standard grading -- indeed, $h\cdot 1=-1\neq 0$ for $n=2$. However, as both copies of $\CP_n$ inside $\shend(\CP_n)$ have this grading shift, they cancel out and the final identification with $\NH_n$ has the correct grading. 
\end{remark}

It turns out this $\sl_2$-action can be extended to a half-Witt action, or even a full-Witt action if one is willing to introduce ``negative dots''. We delay this discussion to the full general case in Section \ref{sec:fullWitt}.  

\subsection{The core is global sections}
One of the main consequences of our coherent construction of KLR is that the $\sl_2$-core admits a very nice geometric description.
\begin{theorem}
Under the identification \[\NH_n=\CE_n(X_n^+),\] we have \[\Cor_{\sl_2}(\NH_n)=\CE_n(X_n)=\wh\NC_n.\]
\end{theorem}

\begin{proof}
In one direction, note that the space $\CE_n(X_n)$ is finite-dimensional and $\sl_2$-stable, so it is contained in $\Cor_{\sl_2}(\NH_n)$. 

Conversely, let $M\subset\CE_n(X_n^+)$ be some finite-dimenisonal $\sl_2$-submodule. Over $\BC$, it is well known that the action on $M$ integrates to a natural $\SL_2$-action. 

Then, for any $p\in \BP^1$, let \[H_p=\{D\in \Sym^n\BP^1\colon p\in \Supp D\}\subset X_n.\] For example, $X_n\setminus X_n^+$ is $H_\infty$. If $s\in M$ does not extend to a function on $X_n$, then, by local freeness of $\mathcal{E}_n$, $s$ must have a pole along $H_\infty$. We may choose $g\in \SL_2$ that does not fix $\infty$. This ensures that $g\cdot s\in M$ has a pole along $H_{g\infty}$. Since $H_{g\infty}\ne H_\infty$, its generic point lies in $X_n^+$, which contradicts the assumption that $g\cdot s\in M \subset \CE_n(X_n^+)$. 

So in fact every element of $M$ must actually extend to $X_n$, which proves the reverse inclusion.
\end{proof}

\part{The $A_N$ case: KLR}\label{part:II}
The main algebraic tool in the 1-color case is that of the isomorphism from $\NH_n$ to $\End_{\Sym_n}(\Pol_n)$. This tool made the 1-color story very simple, perhaps even too simple. We lose this tool when we generalize to more colors. Indeed, \cite{KL09} showed that the KLR algebra acts faithfully upon the projective module $\Pol_\alpha=\bigoplus_\beta \Pol_\beta$, and that the action commutes with the commutative subring $\Sym_\alpha=\pr*{\bigotimes_i \Sym_{n_i}}\pr*{\sum_\beta 1_\beta}$. However, this gives only an injection
\[\CR_\alpha\linj \End_{\Sym_\alpha}(\Pol_\alpha)\cong\End_{n!}(\Sym_\alpha),\] 
which is not a surjection. Our new main algebraic tool will be to identify the image of $\CR_\alpha$ under this map.

Let us fix some conventions. Our convention for the 2-color KLR, corresponding to the quiver
\[\begin{tikzcd}[column sep=small]
  1 \arrow[dash]{r} & \textcolor{red}{2}
\end{tikzcd},\]
will be that
\[\begin{diagram}
        \draw (0,0) ..controls(2.5,2).. (0,4);
        \draw[red] (2,0)..controls(-0.5,2)..(2,4);
    \end{diagram}
    \:=\:\begin{diagram}
        \draw(0,0)--(0,4);
        \draw[red](2,0)--(2,4);
        \fill (0,2) circle (6pt);
    \end{diagram}
    \:-\:
    \begin{diagram}
        \draw(0,0)--(0,4);
        \draw[red](2,0)--(2,4);
        \fill[red] (2,2) circle (6pt);
    \end{diagram}\:
    ,\qquad 
    \begin{diagram}
        \draw (0,0)--(4,4);
    \draw (4,0)--(0,4);
    \draw[red] (2,0) ..controls(-0.5,2).. (2,4);
    \end{diagram}
    \:-\:
    \begin{diagram}
        \draw (0,0)--(4,4);
    \draw (4,0)--(0,4);
    \draw[red] (2,0) ..controls(4.5,2).. (2,4);
    \end{diagram} \:= \:\begin{diagram}
        \draw(0,0)--(0,4);
        \draw[red](2,0)--(2,4);
        \draw(4,0)--(4,4);
    \end{diagram}\:,\]
    \[\begin{diagram}
        \draw[red] (0,0) ..controls(2.5,2).. (0,4);
        \draw (2,0)..controls(-0.5,2)..(2,4);
    \end{diagram}
    \:=\:\begin{diagram}
        \draw[red](0,0)--(0,4);
        \draw(2,0)--(2,4);
        \fill (2,2) circle (6pt);
    \end{diagram}
    \:-\:
    \begin{diagram}
        \draw[red](0,0)--(0,4);
        \draw(2,0)--(2,4);
        \fill[red] (0,2) circle (6pt);
    \end{diagram}\:
    ,\qquad 
    \begin{diagram}
        \draw[red] (0,0)--(4,4);
    \draw[red] (4,0)--(0,4);
    \draw (2,0) ..controls(-0.5,2).. (2,4);
    \end{diagram}
    \:-\:
    \begin{diagram}
        \draw[red] (0,0)--(4,4);
    \draw[red] (4,0)--(0,4);
    \draw (2,0) ..controls(4.5,2).. (2,4);
    \end{diagram} \:= -\:\begin{diagram}
        \draw[red](0,0)--(0,4);
        \draw(2,0)--(2,4);
        \draw[red](4,0)--(4,4);
    \end{diagram}\:.\]
Accordingly, the action on the polynomial representation (whose vectors are drawn with an orange line beneath them, to distinguish them from elements of the KLR algebra) will be such that
\[\begin{diagram}
  \draw[red](0,0)--(2,2);
  \draw(2,0)--(0,2);
  \draw[red](0,-0.5)--(0,0);
  \draw(2,-0.5)--(2,0);
  \draw[orange](-0.5,-0.5)--(2.5,-0.5);
\end{diagram}\:=\:
\begin{diagram}
  \draw(0,-0.5)--(0,1);
  \draw[red](2,-0.5)--(2,1);
  \fill (0,0.5) circle (6pt);
  \draw[orange](-0.5,-0.5)--(2.5,-0.5);
  \draw[white](0,1)--(0,2);
\end{diagram}\:-\:
\begin{diagram}
  \draw(0,-0.5)--(0,1);
  \draw[red](2,-0.5)--(2,1);
  \fill[red] (2,0.5) circle (6pt);
  \draw[orange](-0.5,-0.5)--(2.5,-0.5);
  \draw[white](0,1)--(0,2);
\end{diagram}\:,\qquad \begin{diagram}
  \draw(0,0)--(2,2);
  \draw[red](2,0)--(0,2);
  \draw(0,-0.5)--(0,0);
  \draw[red](2,-0.5)--(2,0);
  \draw[orange](-0.5,-0.5)--(2.5,-0.5);
\end{diagram}\:=\:
\begin{diagram}
  \draw[red](0,-0.5)--(0,1);
  \draw(2,-0.5)--(2,1);
  \draw[orange](-0.5,-0.5)--(2.5,-0.5);
  \draw[white](0,1)--(0,2);
\end{diagram}\:.
\]
Our convention of the polynomial representation is that
\[\deg\:\begin{diagram}
  \draw[red](0,-0.5)--(0,1);
  \draw(2,-0.5)--(2,1);
  \draw[orange](-0.5,-0.5)--(2.5,-0.5);
\end{diagram}\:=+1,\qquad \deg \:\begin{diagram}
  \draw(0,-0.5)--(0,1);
  \draw[red](2,-0.5)--(2,1);
  \draw[orange](-0.5,-0.5)--(2.5,-0.5);
\end{diagram}\:=0.\]

\section{A recollection of KLR and KLM}
In this section we recall some constructions from categorification.

\subsection{The KLR algebra}
For the reader who is not too familiar with KLR algebras, we give here a concise overview. The KLR algebra is defined as follows: the monoidal generators are 
\begin{center}
\begin{tabular}{ c c c c c c }
 & $\begin{diagram}
    \draw (0,0)--(0,2);
    \node at (0,-0.5) {\tiny $i$};
\end{diagram}$ & $\begin{diagram}
    \draw (0,0)--(0,2);
    \fill (0,1) circle (5pt);
    \node at (0,-0.5) {\tiny $i$};
\end{diagram}$ & 
$\begin{diagram}
    \draw (0,0)--(2,2);
    \draw (2,0)--(0,2);
    \node at (0,-0.5) {\tiny $i$};
    \node at (2,-0.5) {\tiny $i$};
\end{diagram}$ & 
$\begin{diagram}
    \draw (0,0)--(2,2);
    \draw (2,0)--(0,2);
    \node at (0,-0.5) {\tiny $i$};
    \node at (2,-0.5) {\tiny $i\pm 1$};
\end{diagram}$ & 
$\begin{diagram}
    \draw (0,0)--(2,2);
    \draw (2,0)--(0,2);
    \node at (0,-0.5) {\tiny $i$};
    \node at (2,-0.5) {\tiny $j$};
\end{diagram}$
\\[1em]
    degree & 0 & $2$ & $-2$ & \hspace*{-6.5pt}$1$ & $0$ 
\end{tabular}
\end{center}
    where $|j-i|>1$.
The important relations are
\begin{align*}
    {}\raisebox{-0.5em}{$\hackcenter{\begin{tikzpicture}[scale=0.375]
    \draw (0,0)--(2,2);
    \draw (2,0)--(0,2);
    \node at (0,-0.5) {\tiny $i$};
    \node at (2,-0.5) {\tiny $i$};
    \fill (0.5,1.5) circle (5pt);
\end{tikzpicture}}$}
-
\raisebox{-0.5em}{$\hackcenter{\begin{tikzpicture}[scale=0.375]
    \draw (0,0)--(2,2);
    \draw (2,0)--(0,2);
    \node at (0,-0.5) {\tiny $i$};
    \node at (2,-0.5) {\tiny $i$};
    \fill (1.5,0.5) circle (5pt);
\end{tikzpicture}}$}
&=
\raisebox{-0.5em}{$\hackcenter{\begin{tikzpicture}[scale=0.375]
    \draw (0,0)--(2,2);
    \draw (2,0)--(0,2);
    \node at (0,-0.5) {\tiny $i$};
    \node at (2,-0.5) {\tiny $i$};
    \fill (0.5,0.5) circle (5pt);
\end{tikzpicture}}$}
-
\raisebox{-0.5em}{$\hackcenter{\begin{tikzpicture}[scale=0.375]
    \draw (0,0)--(2,2);
    \draw (2,0)--(0,2);
    \node at (0,-0.5) {\tiny $i$};
    \node at (2,-0.5) {\tiny $i$};
    \fill (1.5,1.5) circle (5pt);
\end{tikzpicture}}$}
=
\raisebox{-0.5em}{$\hackcenter{\begin{tikzpicture}[scale=0.375]
    \draw (0,0)--(0,2);
    \draw (2,0)--(2,2);
    \node at (0,-0.5) {\tiny $i$};
    \node at (2,-0.5) {\tiny $i$};
\end{tikzpicture}}$},\\ 
{}\raisebox{-0.5em}{$\hackcenter{\begin{tikzpicture}[scale=0.375]
    \draw (0,0)..controls(2.5,2)..(0,4);
    \draw (2,0)..controls(-0.5,2)..(2,4);
    \node at (0,-0.5) {\tiny $i$};
    \node at (2,-0.5) {\tiny $i$};
\end{tikzpicture}}$}
&=0,\\ 
{}\raisebox{-0.5em}{$\hackcenter{\begin{tikzpicture}[scale=0.375]
    \draw (0,0)..controls(2.5,2)..(0,4);
    \draw (2,0)..controls(-0.5,2)..(2,4);
    \node at (0,-0.5) {\tiny $i$};
    \node at (2,-0.5) {\tiny $i\pm 1$};
\end{tikzpicture}}$}
&=\pm 
\raisebox{-0.5em}{$\hackcenter{\begin{tikzpicture}[scale=0.375]
    \draw (0,0)--(0,4);
    \draw (2,0)--(2,4);
    \node at (0,-0.5) {\tiny $i$};
    \node at (2,-0.5) {\tiny $i\pm 1$};
    \fill (0,2) circle (5pt);
\end{tikzpicture}}$}
\mp 
\raisebox{-0.5em}{$\hackcenter{\begin{tikzpicture}[scale=0.375]
    \draw (0,0)--(0,4);
    \draw (2,0)--(2,4);
    \node at (0,-0.5) {\tiny $i$};
    \node at (2,-0.5) {\tiny $i\pm 1$};
    \fill (2,2) circle (5pt);
\end{tikzpicture}}$},\\ 
{}\raisebox{-0.5em}{$\hackcenter{\begin{tikzpicture}[scale=0.375]
    \draw (0,0)--(4,4);
    \draw (4,0)--(0,4);
    \draw (2,0) ..controls(-0.5,2).. (2,4);
    \node at (0,-0.5) {\tiny $i$};
    \node at (2,-0.5) {\tiny $i\pm 1$};
    \node at (4,-0.5) {\tiny $i$};
\end{tikzpicture}}$}
-
\raisebox{-0.5em}{$\hackcenter{\begin{tikzpicture}[scale=0.375]
    \draw (0,0)--(4,4);
    \draw (4,0)--(0,4);
    \draw (2,0) ..controls(4.5,2).. (2,4);
    \node at (0,-0.5) {\tiny $i$};
    \node at (2,-0.5) {\tiny $i\pm 1$};
    \node at (4,-0.5) {\tiny $i$};
\end{tikzpicture}}$}
&=\pm
\raisebox{-0.5em}{$\hackcenter{\begin{tikzpicture}[scale=0.375]
    \draw (0,0)--(0,4);
    \draw (4,0)--(4,4);
    \draw (2,0)--(2,4);
    \node at (0,-0.5) {\tiny $i$};
    \node at (2,-0.5) {\tiny $i\pm 1$};
    \node at (4,-0.5) {\tiny $i$};
\end{tikzpicture}}$};
\end{align*}
the less important relations are 
\begin{align*}
    \raisebox{-0.5em}{$\begin{diagram}
        \draw (0,0)..controls(2.5,2)..(0,4);
    \draw (2,0)..controls(-0.5,2)..(2,4);
    \node at (0,-0.5) {\tiny $i$};
    \node at (2,-0.5) {\tiny $j$};
    \end{diagram}$}
    &=
    \raisebox{-0.5em}{$\begin{diagram}
        \draw (0,0)--(0,4);
        \draw (2,0)--(2,4);
         \node at (0,-0.5) {\tiny $i$};
    \node at (2,-0.5) {\tiny $j$};
    \end{diagram}$}
    \quad\te{ for }|i-j|>1,\\ 
    \raisebox{-0.5em}{$\begin{diagram}
        \draw (0,0)--(2,2);
    \draw (2,0)--(0,2);
    \node at (0,-0.5) {\tiny $i$};
    \node at (2,-0.5) {\tiny $j$};
    \fill (1.5,1.5) circle (5pt);
    \end{diagram}$}
    &=
    \raisebox{-0.5em}{$\begin{diagram}
        \draw (0,0)--(2,2);
    \draw (2,0)--(0,2);
    \node at (0,-0.5) {\tiny $i$};
    \node at (2,-0.5) {\tiny $j$};
    \fill (0.5,0.5) circle (5pt);
    \end{diagram}$}
    \quad\te{ for }i\neq j,\\ 
    \raisebox{-0.5em}{$\begin{diagram}
        \draw (0,0)--(2,2);
    \draw (2,0)--(0,2);
    \node at (0,-0.5) {\tiny $i$};
    \node at (2,-0.5) {\tiny $j$};
    \fill (0.5,1.5) circle (5pt);
    \end{diagram}$}
    &=
    \raisebox{-0.5em}{$\begin{diagram}
        \draw (0,0)--(2,2);
    \draw (2,0)--(0,2);
    \node at (0,-0.5) {\tiny $i$};
    \node at (2,-0.5) {\tiny $j$};
    \fill (1.5,0.5) circle (5pt);
    \end{diagram}$}
    \quad\te{ for }i\neq j,\\ 
    \raisebox{-0.5em}{$\begin{diagram}
        \draw (0,0)--(4,4);
    \draw (4,0)--(0,4);
    \draw (2,0) ..controls(-0.5,2).. (2,4);
    \node at (0,-0.5) {\tiny $i$};
    \node at (2,-0.5) {\tiny $j$};
    \node at (4,-0.5) {\tiny $k$};
    \end{diagram}$}
&=
    \raisebox{-0.5em}{$\begin{diagram}
        \draw (0,0)--(4,4);
    \draw (4,0)--(0,4);
    \draw (2,0) ..controls(4.5,2).. (2,4);
    \node at (0,-0.5) {\tiny $i$};
    \node at (2,-0.5) {\tiny $j$};
    \node at (4,-0.5) {\tiny $k$};
    \end{diagram}$}
    \quad\te{ for }(j,k)\neq (i\pm 1,i).
\end{align*}

Given a nonnegative sum of simple roots $\alpha=\sum_i n_i\alpha_i$, $\CR_\alpha$ is the subalgebra such that there are $n_i$ strings colored by $i$ for each $i$.

Sometimes it is convenient to have algebraic notation for these diagrams. For an ordered sequence of colors $\beta=(\beta_1,\cdotsc,\beta_n)$, we let $e^\beta$ denote the diagram of propagating strings with colors indicated by $\beta$:
\[e^\beta=\raisebox{-0.5em}{$\begin{diagram}
    \draw (0,0)--(0,2);
    \node at (0,-0.5) {\tiny $\beta_1$};
    \draw (1,0)--(1,2);
    \node at (1,-0.5) {\tiny $\beta_2$};
    \node at (2.1,1) {$\cdots$};
    \draw (3,0)--(3,2);
    \node at (3,-0.5) {\tiny $\beta_n$};
\end{diagram}$}.\]
$\psi_i$ denotes the crossing of the $i$-th and $(i+1)$-th strings, and $y_i$ denotes a dot on the $i$-th string.


\subsection{The algebraic affine cellular structure of KLM}
In this section let us quickly recall the affine cellular structure on a type $A$ KLR algebra, due to \cite{kleshchev2013affine}. 


We will adhere to `highest-weight notation' for cellular algebras, namely so that $A^{\ge\lbd}$ is an ideal of $A$. 

For type $A$, let the simple roots be $\{\alpha_i\}$. Recall that the set of positive roots is 
\[\Phi^+=\{\alpha_i+\alpha_{i+1}+\cdots+\alpha_j:i\le j\}.\] 
In what follows we will identify a positive root $\alpha_i+\cdots+\alpha_j$ with a sequence (purely a piece of combinatorial data) as follows:
\[\alpha_i+\cdots+\alpha_j \llra (j,j-1,\cdotsc,i).\] 
By choosing a total order on the set of simple roots, for instance by declaring that $\alpha_1>\alpha_2>\cdots$ (note well that we choose this to be opposite to the way we wrote down positive roots earlier, as $\alpha_i+\cdots+\alpha_j$ for $i<j$!), we can obtain a lexicographic total order on $\Phi^+$ by comparing them as sequences. 
\begin{example}
    In the case of $\sl_3$, the positive roots are
    \[\Phi^+=\{\alpha_1>\alpha_2+\alpha_1>\alpha_2\},\] 
    ordered as above. We can also think of this as
    \[(1)>(21)>(2).\] 
\end{example}

The poset describing the cellular structure on $\CR_\alpha$ is that of the poset of ``root partitions of $\alpha$'', $\Pi_\alpha$, namely the set of ways to write $\alpha$ as an ordered sum (largest on the left) of positive roots; this set is (totally) ordered by lexicographically comparing the sequences associated with each ordered sum of positive roots.
\begin{example}
    For the root $\alpha=2\alpha_1+2\alpha_2$, there are 3 such root partitions, namely
    \[\Pi_{2\alpha_1+2\alpha_2}=\{(1)(1)(2)(2)>(1)(21)(2)>(21)(21)\}.\] 
    Note well that the parenthetical notation is extraneous; given an appropriate sequence of 1's and 2's, there is only one way to interpret it as an ordered sum of positive roots.
\end{example}

Having said what the poset is, we should also describe the basis. To do this we must define some sets and elements. Associated to $\pi$ we have the KLR idempotent $e^\pi$, given by straight strings in an order prescribed by $\pi$; namely, if $\pi=\sum n_i\beta_i$ for $\beta_i$ positive roots, then let
\[e^\pi\coloneqq e^{n_1\beta_1}\otimes\cdots\otimes e^{n_r\beta_r},\] 
where for $\beta=\alpha_j+\cdots+\alpha_i$ $e^{n\beta}$ is defined as 
\[e^{n\beta}\coloneqq \underbrace{e^{j,j-1,\cdotsc,i}\otimes\cdots\otimes e^{j,j-1,\cdotsc,i}}_{n}.\]

Given a root partition $\pi=\sum n_i \beta_i$ for $\beta_i$ positive roots, let $S_\pi$ denote 
\[S_\pi\coloneqq \prod_i S_{|\beta_i|}^{\times n_i};\] 
if we think of each $\beta_i$ in the sum $\pi=\sum n_i \beta_i$ as corresponding to a thick ``cable'' of strings (i.e. each cable is a grouping of the KLR strings $e^{b,b-1,\cdots,a}$, where $\beta_i=\alpha_b+\cdots+\alpha_a$), this $S_\pi$ is the group of permutations corresponding to twisting each thick cable internally. Let $S^\pi$ denote the set of minimal length representatives of the coset $S_n/S_\pi$. For each $w\in S_n/S_\pi$ we can consider the KLR diagram $\psi_w e^\pi$. 

Another crossing diagram we can define is the ``full twists'' crossing $\psi_\pi$, namely the one obtained by taking the idempotent $e^\pi=e^{n_1\beta_1}\otimes\cdots e^{n_r\beta_r}$ and putting the longest permutation in $S_{n_i}$ on each $e^{n_i\beta_i}$, where $S_{n_i}$ attaches to $e^{n_i\beta_i}$ by thinking of permutations as permuting the thick cables corresponding to $\beta_i$. In other words,
\[\psi_\pi\coloneqq (\psi_{n_1\beta_1}\otimes \cdots \psi_{n_r\beta_r})e^\pi,\] 
where $\psi_{n\beta}$ is defined by fulling permuting the $n$ thick cables.

Let $\Sym^\pi$ denote the ring of ``symmetric functions'' associated to $\pi$. This is a ring which is isomorphic to
\[\Sym^\pi\cong \bigotimes_i\Sym_{n_i},\] 
sitting inside KLR by thinking of each $\Sym_{n_i}$ as a copy of symmetric polynomials in $n_i$ variables sitting on the $n_i$ thick cables associated to $\beta_i$, where for each cable we place the dot on the right-most KLR string. Frequently in the following we write $\Lbd_\pi\coloneqq \Sym^\pi$. 

Finally, let $\ul{y}^\pi$ be
\[\ul{y}^\pi\coloneqq (\ul{y}^{n_1\beta_1}\otimes\cdots\otimes \ul{y}^{n_r\beta_r})e^\pi,\] 
where $\ul{y}^{n_i\beta_i}$ is the image of $y_1^{n-1}y_2^{n-2}\cdots y_{n-1}$ inside of $\Pol_{n_i}$, which in turn is inside of $\Pol^\pi$. 

Then the basis theorem of KLM is as follows.
\begin{theorem}[KLM]
    Let $u,w\in S^\pi$ and let $\{f_\lbd\}_\lbd$ form a basis of $\Sym^\pi$. Then the elements
    \[c_{u,w,\lbd}^\pi\coloneqq \psi_u f_\lbd \ul{y}^\pi \psi_\pi \ul{y}^\pi \psi_w^\dag\] 
    form an affine cellular basis for $\CR_\alpha$, which is to say that they form a basis and for any $x\in\CR_\alpha$ we have
    \[x c_{u,w,\lbd}^\pi \equiv \sum_{v,\mu} r_{x,u,v,\mu,\lbd} c_{v,w,\mu}^\pi \mod \CR_\alpha^{>\pi},\] 
    where $r_{x,u,v,\mu}$ is a constant in $\BC$ which depends on $x,v,u,\mu,\lbd$ but not on $w$. Alternatively, we can write \[x c_{u,w,\lbd}^\pi \equiv \sum_{v,\mu} r_{x,u,v} \cdot c_{v,w,\lbd}^\pi \mod \CR_\alpha^{>\pi},\]   where $r_{x,u,v}\in\Sym^\pi$ is a coefficient depending on $x,u,v$ but not on $\lbd$, and the action is $r\cdot c_{v,w,\lbd}^\pi=\psi_u rf_\lbd \ul y^\pi \psi_\pi \ul y^\pi \psi_w^\dag$. 
\end{theorem}

Moving forward, let us fix our notation/convention to be so that for instance in type $A_3$ we have $(1)>(21)>(2)>(321)>(32)>(3)$, and let us fix a labeling on all root partitions of $\alpha$ such that
\[\pi_0>\cdots>\pi_\ell,\] 
where $\ell+1$ is the total number of root partitions. We oftentimes write $[b,a]$ for the sequence $(b,b-1,\cdotsc,a)$, with $[a,a]=(a)$.

\section{A small 2-color example}\label{sect:2color}
In this subsection we discuss the smallest 2-color case, i.e. $\alpha=\alpha_1+\alpha_2$, to illustrate the story and give the diagrammatics.

Thankfully it is not difficult to describe the image if $n_1=n_2=1$. Let us fix the ordered basis $\{1_{21},1_{12}\}$ of $\Pol_{1,1}$ over $\Sym_{1,1}$ to start with.
\begin{lemma}
    At $n_1=n_2=1$, the image of the above morphism $\CR\linj \End_{(1+1)!}(\Sym_{1,1})$ is 
    \[\tbt{\BC[x,y]}{\BC[x,y]}{(x-y)\BC[x,y]}{\BC[x,y]}.\]
    The elements $1_{21}$ and $1_{12}$ of $\CR_{1,1}$ correspond to $E_{11}$ and $E_{22}$ respectively, and the elements $\psi_{21}^{12}$ and $\psi_{12}^{21}$ correspond to $(x-y)E_{21}$ and $E_{12}$ respectively. 
\end{lemma}

Let us first philosophize briefly. The core principle is that the $\sl_2$-action of \cite{EQ23} (and the Witt action of \cite{lauda2025action}) should be obtained from the geometric picture; hence we expect products of $\BP^1$'s to be involved. The KLR degrees should match the $\hlie$-action, which tells us what the sums of twists should be. Using the $e$-action allows us to further pin down the correct twists.

Let us first consider the space 
\[X\coloneqq \BP^1\times\BP^1.\] 
Letting $\BP^1=U\sqcup_{U\cap U^-}U^-$ be the construction of $\BP^1$ from two affine charts, $ X$ is then glued from 4 affine charts, namely $X^{\eps_1\eps_2}\coloneqq U^{\eps_1}\times U^{\eps_2}$ for $\eps_i\in\{+,-\}$. Since there is only one string of each color, we do not expect there to be a non-trivial cover in the sense of e.g. $\NH_2$ having $\BP^1\times\BP^1\lto \BP^2$. Let the first copy of $\BP^1$ be $\BP^1=\Proj\BC[X_0,X_1]$, so that the coordinate on $U$ is $x=X_1/X_0$, and let the second copy of $\BP^1$ be $\BP^1=\Proj\BC[Y_0,Y_1]$, so that the coordinate on $U$ is $y=Y_1/Y_0$. For instance, $U^{++}\subset X$ is then the chart where $X_0=Y_0=1$. 

There are sheaves on $\wh X$,
\[\CO(a,b)\coloneqq \CO(a)\boxtimes\CO(b).\] 
Moving forward, we will compare rational sections of such sheaves on different charts inside the ambient space of the generic stalk, $\CO(a,b)_\eta$, after fixing a frame at infinity. Under this comparison, we have, for $U$ any of the four affine charts, 
\[\CO(a,b)(U)=X_0^{-a}Y_0^{-b}\BC[X_0,X_1,Y_0,Y_1].\]

Consider the section 
\[\Delta\coloneqq  X_1Y_0-X_0Y_1\in H^0(\CO(1,1)),\]
multiplication by which can be considered as a map
\[\Delta\colon \CO(-1,0)\linj \CO(0,1).\]

Consider the sheaf (this sheaf is reverse-engineered by using the Elias-Qi action)
\[\CP\coloneqq \CO(0,-1)\oplus \CO(0,0)\] 
on $ X$. Then $\shend\CP$ looks like
\[\shend\CP=\tbt{\CO(0,0)}{\CO(0,-1)}{\CO(0,1)}{\CO(0,0)}.\] 
This is the geometric analogue of $\End_{\Sym}\Pol$, inside which KLR lives. The geometric analogue of KLR is the subsheaf $\CE$ fixing the subsheaf $\CQ\subset\CP$ given by
\[\id\otimes \Delta\colon \CO(0,-1)\oplus \CO(-1,-1)\linj \CO(0,-1)\oplus \CO(0,0).\] 
$\CE$ can be explicitly described as
\[\CE\coloneqq \tbt{\CO}{\CO(0,-1)}{\Img\Delta}{\CO}\cong \tbt{\CO}{\CO(0,-1)}{\CO(-1,0)}{\CO}.\] 
By taking the embedding $\CO(a,b)(U)\linj \CO_{\BP^1\times\BP^1,\eta}=\BC(x,y)$ via $s\lmto X_0^{-a}Y_0^{-b}s$, we can compute the space of sections of this sheaf over each affine chart (with $X_i=Y_j=1$) by identifying the image with
\[\tbt{A}{Y_0B}{X_0\frac{\Delta}{X_0Y_0}C}{D}\] 
for $A,B,C,D\in\BC[X_0,X_1,Y_0,Y_1]$.
By so doing we may compute that (here e.g. $X^{+-}$ is the chart on which $X_0=Y_1=1$)
\begin{align*}
    \CE(X^{++})
    &=\tbt{\BC[x,y]}{\BC[x,y]}{(x-y)\BC[x,y]}{\BC[x,y]},\\
    \CE(X^{+-})&=\tbt{\BC[x,y^{-1}]}{\frac{1}{y}\BC[x,y^{-1}]}{(x-y)\BC[x,y^{-1}]}{\BC[x,y^{-1}]},\\
    \CE(X^{-+})&=\tbt{\BC[x^{-1},y]}{\BC[x^{-1},y]}{\pr*{1-\frac{y}{x}}\BC[x^{-1},y]}{\BC[x^{-1},y]},\\
    \CE(X^{--})&=\tbt{\BC[x^{-1},y^{-1}]}{\frac{1}{y}\BC[x^{-1},y^{-1}]}{\pr*{1-\frac{y}{x}}\BC[x^{-1},y^{-1}]}{\BC[x^{-1},y^{-1}]}.
\end{align*}
Note that $\CE(X^{++})\cong\CR_{1,1}$ inside of $\End_{2}(\Sym_{1,1})$. Hence $\CE$ is a sheaf version of the KLR $\CR_{1,1}$. 
\begin{proposition}
    The sheaf $\CE$ has $\CE(X^{++})=\CR_{1,1}$, and is hence a sheaf version of $\CR_{1,1}$. Moreover the global sections satisfy $\CE(\wh X)=\BC\oplus\BC\subset \CE(X^{++})$, which is the semisimplification $\wt\CR_{1,1}$ of $\CR_{1,1}$. 
\end{proposition}

We may extend the graphical calculus of $\CR_{1,1}$ using $\CE$ as a rubric. Since $\CE$ is a subsheaf of $\End(\CP)$, we expect the gradings of the polynomial representations to match. In other words, we expect that the action of the usual KLR on the usual polynomial representation to be a specialization of the generic stalk KLR on the generic stalk polynomial representation. 

On the chart $X^{-+}$, the predicted action is:
\[\begin{diagram}
  \draw[red](0,0)--(2,2);
  \draw[squigs](2,0)--(0,2);
  \draw[red](0,-0.5)--(0,0);
  \draw[squigs](2,-0.5)--(2,0);
  \draw[orange](-0.5,-0.5)--(2.5,-0.5);
\end{diagram}\:=\:
\begin{diagram}
  \draw[squigs](0,-0.5)--(0,1);
  \draw[red](2,-0.5)--(2,1);
  \draw[orange](-0.5,-0.5)--(2.5,-0.5);
  \draw[white](0,1)--(0,2);
\end{diagram}\:-\:
\begin{diagram}
  \draw[squigs](0,-0.5)--(0,1);
  \draw[red](2,-0.5)--(2,1);
  \draw[orange](-0.5,-0.5)--(2.5,-0.5);
  \draw[white](0,1)--(0,2);
  \fill[red](2,0.5)circle(6pt);
  \draw[black](0,0.5)circle(6pt);
\end{diagram}\:,\qquad 
\begin{diagram}
  \draw[squigs](0,0)--(2,2);
  \draw[red](2,0)--(0,2);
  \draw[squigs](0,-0.5)--(0,0);
  \draw[red](2,-0.5)--(2,0);
  \draw[orange](-0.5,-0.5)--(2.5,-0.5);
\end{diagram}\:=\:
\begin{diagram}
  \draw[red](0,-0.5)--(0,1);
  \draw[squigs](2,-0.5)--(2,1);
  \draw[orange](-0.5,-0.5)--(2.5,-0.5);
  \draw[white](0,1)--(0,2);
\end{diagram}\:.
\]
Since then $\deg \:\begin{diagram}
  \draw[red](0,0)--(2,2);
  \draw[squigs](2,0)--(0,2);
\end{diagram}\:+\deg \:\begin{diagram}
  \draw[red](0,-0.5)--(0,1);
  \draw[squigs](2,-0.5)--(2,1);
  \draw[orange](-0.5,-0.5)--(2.5,-0.5);
\end{diagram}\:=\deg \:
\begin{diagram}
  \draw[squigs](0,-0.5)--(0,1);
  \draw[red](2,-0.5)--(2,1);
  \draw[orange](-0.5,-0.5)--(2.5,-0.5);
\end{diagram}\:$ and 
$\deg \:\begin{diagram}
  \draw[squigs](0,0)--(2,2);
  \draw[red](2,0)--(0,2);
\end{diagram}\:+\deg \:
\begin{diagram}
  \draw[squigs](0,-0.5)--(0,1);
  \draw[red](2,-0.5)--(2,1);
  \draw[orange](-0.5,-0.5)--(2.5,-0.5);
\end{diagram}\:=\deg \:\begin{diagram}
  \draw[red](0,-0.5)--(0,1);
  \draw[squigs](2,-0.5)--(2,1);
  \draw[orange](-0.5,-0.5)--(2.5,-0.5);
\end{diagram}\:$, we may conclude that
\[\deg \:\begin{diagram}
  \draw[red](0,0)--(2,2);
  \draw[squigs](2,0)--(0,2);
\end{diagram}\:=-1,\qquad \deg \:\begin{diagram}
  \draw[squigs](0,0)--(2,2);
  \draw[red](2,0)--(0,2);
\end{diagram}\:=+1.\]

Similarly on $X^{+-}$ we have
\[\begin{diagram}
  \draw[red,squigs](0,0)--(2,2);
  \draw[black](2,0)--(0,2);
  \draw[red,squigs](0,-0.5)--(0,0);
  \draw[black](2,-0.5)--(2,0);
  \draw[orange](-0.5,-0.5)--(2.5,-0.5);
\end{diagram}\:=\:
\begin{diagram}
  \draw[squigs](0,-0.5)--(0,1);
  \draw[red](2,-0.5)--(2,1);
  \fill (0,0.5) circle (6pt);
  \draw[orange](-0.5,-0.5)--(2.5,-0.5);
  \draw[white](0,1)--(0,2);
\end{diagram}\:-\:
\begin{diagram}
  \draw[squigs](0,-0.5)--(0,1);
  \draw[red](2,-0.5)--(2,1);
  \fill[red] (2,0.5) circle (6pt);
  \draw[orange](-0.5,-0.5)--(2.5,-0.5);
  \draw[white](0,1)--(0,2);
\end{diagram}\:,\qquad 
\begin{diagram}
  \draw[black](0,0)--(2,2);
  \draw[red,squigs](2,0)--(0,2);
  \draw[black](0,-0.5)--(0,0);
  \draw[red,squigs](2,-0.5)--(2,0);
  \draw[orange](-0.5,-0.5)--(2.5,-0.5);
\end{diagram}\:=\:
\begin{diagram}
  \draw[red,squigs](0,-0.5)--(0,1);
  \draw[black](2,-0.5)--(2,1);
  \draw[orange](-0.5,-0.5)--(2.5,-0.5);
  \draw[white](0,1)--(0,2);
  \draw[red](0,0.5)circle(6pt);
\end{diagram}\:.
\]
Again degree considerations show that
\[\deg \:\begin{diagram}
  \draw[red,squigs](0,0)--(2,2);
  \draw[black](2,0)--(0,2);
\end{diagram}\:=+1,\qquad \deg \:\begin{diagram}
  \draw[black](0,0)--(2,2);
  \draw[red,squigs](2,0)--(0,2);
\end{diagram}\:=-1.\]

Note that neither of the above charts given algebras which are isomorphic to $\CR_{1,1}$ up to grading flip. However, the following chart will.

Lastly on $X^{--}$ we have
\[\begin{diagram}
  \draw[red,squigs](0,0)--(2,2);
  \draw[black,squigs](2,0)--(0,2);
  \draw[red,squigs](0,-0.5)--(0,0);
  \draw[black,squigs](2,-0.5)--(2,0);
  \draw[orange](-0.5,-0.5)--(2.5,-0.5);
\end{diagram}\:=\:
\begin{diagram}
  \draw[squigs](0,-0.5)--(0,1);
  \draw[red](2,-0.5)--(2,1);
  \draw[orange](-0.5,-0.5)--(2.5,-0.5);
  \draw[white](0,1)--(0,2);
\end{diagram}\:-\:
\begin{diagram}
  \draw[squigs](0,-0.5)--(0,1);
  \draw[red](2,-0.5)--(2,1);
  \draw[orange](-0.5,-0.5)--(2.5,-0.5);
  \draw[white](0,1)--(0,2);
  \fill[red](2,0.5)circle(6pt);
  \draw[black](0,0.5)circle(6pt);
\end{diagram}\:,\qquad 
\begin{diagram}
  \draw[black,squigs](0,0)--(2,2);
  \draw[red,squigs](2,0)--(0,2);
  \draw[black,squigs](0,-0.5)--(0,0);
  \draw[red,squigs](2,-0.5)--(2,0);
  \draw[orange](-0.5,-0.5)--(2.5,-0.5);
\end{diagram}\:=\:
\begin{diagram}
  \draw[red,squigs](0,-0.5)--(0,1);
  \draw[black,squigs](2,-0.5)--(2,1);
  \draw[orange](-0.5,-0.5)--(2.5,-0.5);
  \draw[white](0,1)--(0,2);
  \draw[red] (0,0.5) circle (6pt);
\end{diagram}\:.
\]
Again degree considerations show that
\[\deg \:\begin{diagram}
  \draw[red,squigs](0,0)--(2,2);
  \draw[black,squigs](2,0)--(0,2);
\end{diagram}\:=-1,\qquad \deg \:\begin{diagram}
  \draw[black,squigs](0,0)--(2,2);
  \draw[red,squigs](2,0)--(0,2);
\end{diagram}\:=-1.\]
Note that the gradings are now flipped from those of $\CR_{1,1}$, which is a good sign. 

Finally we can describe the gluing data of these four algebras by comparing their actions on the generic stalk of the polynomial sheaf. We find that
\begin{align*}
    \:\begin{diagram}
  \draw[red,squigs](0,0)--(2,2);
  \draw[black](2,0)--(0,2);
\end{diagram}\:&=\:\begin{diagram}
  \draw[red](0,0)--(2,2);
  \draw[black](2,0)--(0,2);
\end{diagram}\:,\\
\:\begin{diagram}
  \draw[red](0,0)--(2,2);
  \draw[squigs](2,0)--(0,2);
\end{diagram}\:&=\:\begin{diagram}
  \draw[red](0,0)--(2,2);
  \draw[black](2,0)--(0,2);
  \draw[black](0.5,1.5)circle(6pt);
\end{diagram}\:,\\
\:\begin{diagram}
  \draw[red,squigs](0,0)--(2,2);
  \draw[black,squigs](2,0)--(0,2);
\end{diagram}\:&=\:\begin{diagram}
  \draw[red](0,0)--(2,2);
  \draw[black](2,0)--(0,2);
  \draw[black](0.5,1.5)circle(6pt);
\end{diagram}\:,\\
\:\begin{diagram}
  \draw[black](0,0)--(2,2);
  \draw[red,squigs](2,0)--(0,2);
\end{diagram}\:&=\:\begin{diagram}
  \draw[black](0,0)--(2,2);
  \draw[red](2,0)--(0,2);
  \draw[red](0.5,1.5) circle (6pt);
\end{diagram}\:;\\
\:\begin{diagram}
  \draw[squigs](0,0)--(2,2);
  \draw[red](2,0)--(0,2);
\end{diagram}\:&=\:\begin{diagram}
  \draw[black](0,0)--(2,2);
  \draw[red](2,0)--(0,2);
\end{diagram}\:;\\
\:\begin{diagram}
  \draw[black,squigs](0,0)--(2,2);
  \draw[red,squigs](2,0)--(0,2);
\end{diagram}\:&=\:\begin{diagram}
  \draw[black](0,0)--(2,2);
  \draw[red](2,0)--(0,2);
  \draw[red] (0.5,1.5) circle (6pt);
\end{diagram}\:.
\end{align*}
Note that all dots, hollow or not, commute with all crossings. It is now also diagrammatically clear that none of the crossings are globally defined, which gives a diagrammatic interpretation of Proposition 7.2.

Earlier we had mentioned that the sheaf $\CP$ was chosen by reverse-engineering starting from the Elias-Qi action. Indeed, we may again compute here that for instance
\[e\cdot \:\begin{diagram}
  \draw[black](0,0)--(2,2);
  \draw[red](2,0)--(0,2);
\end{diagram}\:=-\:\begin{diagram}
  \draw[black](0,0)--(2,2);
  \draw[red](2,0)--(0,2);
  \fill[red] (0.5,1.5) circle (6pt);
\end{diagram}\:.\]

\subsubsection{The cellular structure}
``Cellular algebras'', first introduced by \cite{graham1996cellular}, are algebras with a certain filtration by ``cell ideals''. In a certain sense this type of structure can be thought of as an algebraic attempt to mimic stratification techniques from geometry, such as the Bruhat stratification of the flag variety, which gives rise to the highest weight structure of category $\CO$. 

The KLR algebras are affine cellular due to \cite{kleshchev2013affine}. In the case of $\CR=\CR_{1,1}$, the cellular structure is easy to state without introducing too much terminology/machinery. There is a certain ideal $\CR_{\ge(1)(2)}$ defined by
\[\CR_{\ge(1)(2)}\coloneqq \BC[x,y]\set*{ \:\begin{diagram}
    \draw(0,0)--(0,2);
    \draw[red](2,0)--(2,2);
\end{diagram}\:,
\:\begin{diagram}
    \draw(0,0)--(2,2);
    \draw[red](2,0)--(0,2);
\end{diagram}\:,
\:\begin{diagram}
    \draw[red](0,0)--(2,2);
    \draw[black](2,0)--(0,2);
\end{diagram}\:,
\:\begin{diagram}
    \draw[red](0,0)--(0,2);
    \draw[black](2,0)--(2,2);
    \fill (2,1) circle (5pt);
\end{diagram}\:-\:\begin{diagram}
    \draw[red](0,0)--(0,2);
    \draw[black](2,0)--(2,2);
    \fill[red] (0,1) circle (5pt);
\end{diagram}\:,
}\subset \CR.\]
Note that, on the affine patch $X^{++}$, we have (here $R=\BC[x,y]$)
\[\Img \CR_{\ge(1)(2)}=\tbt{(x-y)R}{R}{(x-y)R}{R}.\]
Recall $\CE$ was characterized as the subsheaf of $\shend\CP$ preserving $\CP_{(21)}\coloneqq \CO(0,-1)\oplus\CO(-1,-1)$. Then we can let
\[\CE_{\ge(1)(2)}\coloneqq\te{the subsheaf sending }\CO(0,-1)\oplus \CO(-1,-1)\te{ to }\CO(-1,-2)\oplus\CO(-1,-1).\] 
This subsheaf $\CE_{\ge(1)(2)}\subset\CE$ has the property that
\[\CE_{\ge(1)(2)}(X^{++})=\CR_{\ge(1)(2)}.\]

\section{The spaces and the sheaves}
\subsection{The symmetrizing cover space}
Now let us return to the general type $A_N$ case. Fix $\alpha=\sum_a n_a\alpha_a$. Let
\[Y=\prod_a \BP^{1,\times n_a},\]
and let
\[Y^\circ=\prod_a \pr*{\BP^{1,\times n_a}\mathbin{\bigg\backslash}\bigcup_{n_a(n_a-1)} \BP^{1,\times (n_a-2)}},\]
where the removed space on the $a$-th factor is the union of points $(p_1,\cdotsc,p_{n_a})$ such that $p_i=0,\ p_j=\infty$ for some $i,j$. This space is obtained from $\BP^{1,\times n}$ by removing a codimension-2 subspace, and it is covered by $2^N$ affine charts.

Let $x_{a,1},\cdotsc,x_{a,n_a}$ denote the variables attached to color $a$, and let $\eps\in\{+,-\}^N$ be a sequence of $N$ signs. Let 
\[\Pol_\alpha^\eps=\BC[x_{1,1}^{\eps_1},\cdotsc,x_{1,n_1}^{\eps_1},\cdotsc,x_{N,1}^{\eps_N},\cdotsc,x_{N,n_N}^{\eps_N}],\] 
and let
\[Y^\eps= \Spec \Pol_\alpha^\eps;\] 
these affine charts cover $Y^\circ$.

\subsection{The base space}
Let $\e_{a,i}=\e_{a,i}^+$ denote the $i$-th elementary symmetric polynomial in the variables of color $a$, and let $\e_{a,i}^-$ denote the $i$-th elementary symmetric polynomial in $x_{a,1}^{-1},\cdotsc,x_{a,n_a}^{-1}$. Let $\E_{a,0},\cdotsc,\E_{a,n_a}$ be homogeneous projective coordinates such that $\frac{\E_{a,i}}{\E_{a,0}}=\e_{a,i}$, and let 
\[\Sym_\alpha^\eps=\BC[\e_{1,1}^{\eps_1},\cdotsc,\e_{1,n_1}^{\eps_1},\cdotsc,\e_{N,1}^{\eps_N},\cdotsc,\e_{N,n_N}^{\eps_N}]\] 
denote the tensors of rings of symmetric functions in the variables $x_{a,i}^{\eps_a}$. Let 
\[X=\prod_a \BP^{n_a},\]
and let
\[X^\circ=\prod_a (\BP^{n_a}\sm \BP^{n_a-2}),\] 
where $\BP^{n_a}=\Proj \BC[\E_{a,0},\cdotsc,\E_{a,n_a}]$. This space is obtained from $X$ by removing a codimension-2 subspace, and it is also covered by $2^N$ affine charts, namely
\[X^\eps=\Spec \Sym_\alpha^\eps.\]
Let the ``dominant affine chart'' be $U\coloneqq X^{+^N}$, and let the ``antidominant affine chart'' be $U^-\coloneqq X^{-^N}$. It is not hard to see that $U\cap U^-=\bigcap_\eps X^\eps$.

Recall that the positive roots of type $A_N$ are $\alpha_{[b,a]}\coloneqq \alpha_b+\alpha_{b-1}+\cdots+\alpha_{a+1}+\alpha_a$. Let $\pi$ be a root partition which has $m_{[b,a]}=m_{[b,a]}(\pi)$ copies of the positive root $\alpha_{[b,a]}$ (note $\sum (b-a+1)m_{[b,a]}=n$ and $\sum_{[b,a]\ni c}m_{[b,a]}=n_c$). Let 
\[Z_\pi\coloneqq \prod_{\text{positive roots}} \BP^{m_{[b,a]}}=\prod_{b\ge a} \BP^{m_{[b,a]}},\] 
let
\begin{align*}
    \mu_\pi\colon Z_\pi&\lto X\\
    (f_{[b,a]})_{[b,a]}&\lmto \Bigpr{\prod_{[b,a]\ni c} f_{[b,a]}}_c,
\end{align*}
and let $\CB_\pi=\mu_{\pi,*} \CO_{Z_\pi}$. 

Let $X_\pi=\Img\mu_\pi$ be the scheme-theoretic image, living inside $\prod_a \BP^{n_a}$. By abuse of notation, we also write $\mu_\pi$ to refer to the map $Z_\pi\to X_\pi$. By construction of the map, the (closed points of the) image can be thought of as tuples $(g_c)_c$ such that $\deg\gcd(g_a,g_b)\ge \sum_{[d,c]\ni a,b}m_{[d,c]}$ for any colors $a,b$. This is a necessary but insufficient condition; note that not all tuples with this property are in the image. 

It will be sometimes useful for us to think of  points of $\BP^{n_a}$ as  points of $\Sym^{n_a}\BP^1$, i.e. length-$n_a$ effective divisors on $\BP^1$, so that points of $X$ are tuples of effective divisors on $\BP^1$, with one divisor of length $n_a$ for each color $a$. Similarly we would think of points of $\BP^{m_{[b,a]}}$ as length-$m_{[b,a]}$ effective divisors on $\BP^1$. Under this identification, the map $\mu_\pi$ then becomes 
\begin{align*}
    \mu_\pi\colon Z_\pi&\lto X_\pi\\ 
    (D_{[b,a]})_{[b,a]}&\lmto \Bigpr{\sum_{[b,a]\ni c}D_{[b,a]} }_c;
\end{align*}
then points in the image satisfy that the divisors for colors $a$ and $b$ share at least $\sum_{[d,c]\ni a,b}m_{[d,c]}$ points in common, for any colors $a,b$.

Let
\[X_{\le\pi_k}=\bigcup_{r=k}^\ell X_{\pi_r}\] 
be the scheme-theoretic union. 

\begin{lemma}\label{lem: finiteness_mu}
    The morphism $\mu_\pi\colon Z_\pi\to X_\pi$ is finite and birational. In particular, this means that $X_\pi$ is integral and $Z_\pi$ is its normalization via $\mu_\pi$. Moreover, the morphism is equivariant for the diagonal $\SL_2$-action.
\end{lemma}
\begin{proof}
    We first verify that $\mu_\pi$ has finite fibers. Let $(D_1,\ldots, D_N)$ be a closed point of $X_\pi$. Then, for any tuple of $(f_{[b,a]})_{[b,a]}$ mapping to this point, each $f_{[b,a]}$ is supported on the union of the supports of the $D_c$. Moreover, at each closed point $p$ of $\BP^1$, we have the equality \[\operatorname{mult}_p(D_c)=\sum_{b\geq c\geq a}\operatorname{mult}_p(f_{[b,a]}),\] which evidently has finitely many (non-negative) solutions. So $\mu_\pi$ has finite fibers. Since $Z_\pi$ is projective, we have $\mu_\pi$ is proper with finite fibers, so it is actually finite.

    Next, on the open locus $V$ where distinct factors have pairwise disjoint support and each multiplicity at a point in the support is one, $\mu_\pi$ is an isomorphism from $\mu_{\pi}^{-1}(V)\to V$, i.e. $\mu_\pi$ is birational. Since $Z_\pi$ is integral, its scheme-theoretic image $X_\pi$ is also integral. Moreover, $Z_\pi$ is a product of projective spaces and hence normal. In other words, $Z_\pi$ is the normalization of $X_\pi$ via $\mu_\pi$. 

    Finally, the diagonal action of $\SL_2$ on $\BP^1$ induces actions on all of its symmetric powers and hence on both $Z_\pi$ and $X$. If $\mu_\pi((f_{[b,a]}))=(D_1,\ldots,D_N)$ with $D_c=\sum_{b\ge c\ge a}f_{[b,a]},$ then, for $g\in \SL_2$,
\[\mu_\pi((gf_{[b,a]}))_c=\sum_{b\ge c\ge a}gf_{[b,a]}=g\left(\sum_{b\ge c\ge a}f_{[b,a]}\right)=gD_c.\]
Thus $\mu_\pi$ is $\SL_2$-equivariant.
\end{proof}

\subsection{The disjoint union covers}
Because the KLR algebra has many orthogonal idempotents $e^\beta$ in it, it is convenient to introduce a disjoint union cover of the spaces $X$ and $Y$ to capture this. Let $\Seq=\Seq_\alpha$ be the set of sequences of colors with $n_a$ instance of color $a$, which we call ``$a$-strings''. Let
\begin{align*}
    \wh X&=\bigsqcup_{\beta\in\Seq} X,\\
    \wh Y&=\bigsqcup_{\beta\in\Seq} Y,
\end{align*}
and similarly let $\wh X^\circ=\bigsqcup_\beta X^\circ$ and $\wh Y^\circ=\bigsqcup_\beta Y^\circ$. $\wh X$ and $\wh Y$ have obvious maps to $X$ and $Y$ respectively given by stacking the copies together, and there is an obvious symmetrizing map $\wh\pi\colon\wh Y\lto \wh X$ given by $\pi$ on each copy. We record all this in the following diagram:
\begin{center}
    \begin{tikzcd}
        \wh Y\arrow[twoheadrightarrow]{r}{\wh\sigma}\arrow[twoheadrightarrow,swap]{d}{\wh\pi}\arrow[twoheadrightarrow,"\varpi" description]{dr}& Y\arrow[twoheadrightarrow]{d}{\pi}\\
        \wh X\arrow[twoheadrightarrow,swap]{r}{\sigma}& X
    \end{tikzcd}\ .
\end{center}
Note that each of maps in the diagram above, in particular $\varpi$, is finite (because each is a quotient by a finite group) and flat (by applying miracle flatness, observing that the sources and targets are smooth of the same dimension). Let $\wh U=\varpi^{-1}(U)$ be the ``dominant affine chart'' on the cover $\wh Y$, and similarly let $\wh U^-=\varpi^{-1}(U^-)$.

\subsection{The sheaves above}
Moving forward, the sheaves we consider will be built from line bundles, so we will appeal to Hartog's theorem and conflate e.g. $X$ and $X^\circ$. 

Recall that the dominant root partition is $\pi_0=(1)^{n_1}(2)^{n_2}\cdots (N)^{n_N}$, which we in abuse of notation conflate with the color sequence $1^{n_1}\cdots N^{n_N}$. For any color sequence $\beta\in\Seq$, let $r_i(\beta)$ be the number of $(\beta_i-1)$-strings to the right of the $i$-th string. Let 
\[\wh\CP_{\pi_0}\coloneqq \CO(n_1-1,n_1-2,\cdotsc,0,n_2-1,n_2-2,\cdotsc,0,\cdotsc,n_N-1,n_N-2,\cdotsc,0)\] 
be a sheaf on $\wh Y_{\pi_0}=Y$. Let us define a line bundle on $\wh Y$ by 
\[\wh\CP=\bigsqcup_{\beta\in\Seq}\wh\CP_{\pi_0}\otimes\CO\pr*{-\prod_i y_i^{r_i(\beta)}e^\beta}.\]
Alternatively, letting $r_{a,i}(\beta)$ denote the number of $(a-1)$-strings to the right of the $i$-th $a$-string in $\beta$, we have 
\[\wh\CP=\bigsqcup_\beta \CO\big(n_1-1-r_{1,1}(\beta),n_1-2-r_{1,2}(\beta),\cdotsc,0-r_{1,n_1}(\beta),n_2-1-r_{2,1}(\beta),n_2-2-r_{2,2}(\beta),\cdotsc,0-r_{2,n_2}(\beta)  ,\cdotsc\big).\]
Let $r(\beta)=\sum_i r_i(\beta)=\sum_{a,i}r_{a,i}(\beta)$. 

Note that the black twists are unaffected by changing $\beta$\footnote{This is because we are beginning with the color sequence $\pi_0$ where all blacks are to the left, so all other color sequences are obtained by applying crossings where red runs from bottom-right to top-left; and $e\in\sl_2$ acts on such crossings by adding a red dot, which twists the line bundle down by 1 in the appropriate place.}. Another way to think of this is that $\wh\CP_\beta=\wh\CP_{\pi_0}\otimes \CO\pr*{-\prod_i y_i^{r_i(\beta)}}$.

We say $(a,i)\prec_\beta (b,j)$ if $b=a+1$ and the $i$-th string of color $a$ is to the left of the $j$-th string of color $b$ in $\beta$. Let
\[T_\beta\coloneqq \set*{\big( (a,i),(b,j)\big):(a,i)\prec_\beta (b,j)},\] 
with $t_\beta\coloneqq |T_\beta|$, and let
\[Q_\beta\coloneqq \prod_{(a,i)\prec_\beta (a+1,j) } (x_{a,i}-x_{a+1,j}).\] 
More generally, if $T$ is some set of pairs $((a,i),(b,j))$, we let $Q_T=\prod_{((a,i),(b,j))\in T}(x_{a,i}-x_{b,j})$. Given $T_\beta$ and $w\in S_\alpha$, let $w_a\coloneqq w\rv_{S_{n_a}}$ denote the restriction of $w$ to the $a$-th factor and let
\[wT_\beta=\set*{ \big((a,w_a(i)),(b,w_b(j))\big): ((a,i),(b,j)\in T_\beta  }.\]

Consider the subsheaf of $\wh\CP$ defined by 
\[\wh\CQ=\wh\CP\otimes \bigsqcup_{\beta\in\Seq}\CO\pr*{-\prod_{\substack{ i,j:\beta_i\to \beta_j}}(y_i-y_j)e^\beta}=\wh\CP\otimes\bigsqcup_\beta \CO(-Q_\beta ),\] 
where $Q_\beta =\prod_{\substack{ i,j:\beta_i\to \beta_j}}(y_i-y_j)e^\beta$. $\wh\CP$ and $\wh\CQ$ will be the sheaves from which the KLR sheaf is built.

But we need more sheaves still to describe the geometric cellular filtration on KLR. Let us take $\alpha_1>\alpha_2>\cdots$, so that for instance $(1)(21)(2)$ is a root partition in type $A_2$. Recall that root partitions are totally ordered, so let us order them $\pi_0>\cdots>\pi_\ell$. Let $S^{\pi}$ denote the set of minimal length coset representatives in $S_{n}/S_{\pi}$, and let $\Seq_{\pi}\coloneqq S^{\pi}\cdot\pi$, namely the color sequences obtainable by hitting $\pi$ with $S^{\pi}$. Given $w\in S^{\pi}$, we let $w_a\in S_{n_a}$ denote the element obtained by considering only the $a$-colored strings (or "$a$-strings"). We say $(j(b)i(a))\in\pi$ if the $i$-th $a$-string and the $j$-th $b$-string are in a thick cable of $\pi$. Then let
\[\wh\CI_{\pi} = \bigsqcup_{\beta\in\Seq} \bigcap_{w\in S^{\pi}:w\pi=\beta} \wan{x_{a,w_a(i)}-x_{b,w_b(j)}:(j(b)i(a))\in\pi},\] 
where the intersection over the empty set is defined to be $\CO$. Let
\[\wh\CI_{\le\pi_k}=\bigcap_{r=k}^\ell \wh\CI_{\pi_r},\]
and define
\[\wh\CQ_{\ge\pi_k} = \wh\CQ\otimes \wh\CI_{<\pi_k}.\] Here $\wh\CI_{<\pi_k}=\wh\CI_{\le\pi_{k+1}}$, except at $k=\ell$, where $\wh\CI_{\le\pi_{\ell+1}}=\bigcap_{\emptyset}=\CO$. 
\begin{remark}
    The opposing inequality signs above are an unfortunate consequence of our choice of notation, but we want the filtration of spaces to be $X_{\le\pi_\ell}\subset\cdots\subset X_{\le\pi_0}$, corresponding to ideals $\wh\CI_{\le\pi_0}\subset\cdots\subset\wh\CI_{\le\pi_\ell}$.
\end{remark}

\subsection{The sheaves below}
Consider the sheaf 
\[\CP\coloneqq \varpi_* \wh\CP\] 
on $X$, which is to play the role of the polynomial representation; indeed, its affine sections are certainly isomorphic to the polynomial representation. Also let
\[\CQ_{\ge\pi_k}\coloneqq \varpi_* \wh\CQ_{\ge\pi_k},\]
with 
\[\CQ\coloneqq\varpi_*\wh\CQ.\] The following lemma is clear:\begin{lemma}\label{lem:SL2-equivariance}
The diagonal action of $\SL_2$ on the products of projective lines
induces $\SL_2$-linearizations of $\widehat{\CP}$ and
$\widehat{\CQ}$. The ideals $\widehat{\CI}_\pi$ are
$\SL_2$-stable, and the morphisms $\varpi$ and $\mu_\pi$ are
$\SL_2$-equivariant. Consequently, $\CP$, $\CQ$, and all the
filtration sheaves $\CQ_{\ge\pi_k}$ are $\SL_2$-equivariant.
\end{lemma}
Roughly speaking this is to play the role of the projective representation $\CR e_{\pi_k}$, where $e_{\pi_k}$ is the minimal idempotent associated with $\pi_k$. Note that $\CQ(U)=\bigoplus_\beta Q_\beta\cdot \Pol_\beta$, where $\Pol_\beta= q^{r(\beta)}\cdot\Pol$. 

On the base, let $\CI_{X_\pi}$ denote the subsheaf of $\CO_X$ cutting out $X_\pi$, and let $\CI_{\le\pi_k}$ denote the subsheaf of $\CO_X$ cutting out $X_{\le\pi_k}$. More precisely, we have
\[\CI_{\le\pi_k}=\bigcap_{r=k}^{\ell}\CI_{X_{\pi_r}}.\]
We also write
\[\CI_{<\pi_k}\coloneqq\CI_{\le\pi_{k+1}}=\bigcap_{r=k+1}^{\ell}\CI_{X_{\pi_r}},\]with $\CI_{<\pi_\ell}=\CO_X$. We should warn that $\varpi_*\wh\CI_{\le\pi_k}$ is \textit{not} $\CI_{\le\pi_k}$.

\begin{lemma}\label{lem:Yhat filt vs X filt}
Let $\widehat Y_{\le\pi_k}\subset \widehat Y$ be the closed
subscheme cut out by $\widehat{\CI}_{\le\pi_k}$. Its
scheme-theoretic image under $\varpi$ is $X_{\le\pi_k}$. Similarly, if $\wh Y_\pi$ is the closed subscheme cut out by $\wh\CI_\pi$, its scheme-theoretic image under $\varpi$ is $X_\pi$. 
\end{lemma}
\begin{proof}First fix a root partition $\pi$. On the component of $\widehat Y$
indexed by a color sequence $\beta$, the ideal
$\widehat{\CI}_\pi$ is an intersection of ideals indexed by the
elements $w\in S^\pi$ satisfying $w\pi=\beta$. It follows that its
zero locus is the scheme-theoretic union of the corresponding
diagonal loci, i.e. the variables on each such locus belonging to the
same interval factor of $\pi$ are equal. The resulting morphism to
$X$ therefore factors through $\mu_\pi\colon Z_\pi\longrightarrow X_\pi.$

On the other hand, we have the following finite surjective morphism \[\prod_{b\ge a}(\BP^1)^{m_{[b,a]}}
\longrightarrow\prod_{b\ge a}\Sym^{m_{[b,a]}}\BP^1=Z_\pi.\]Observe that the source maps to one of the diagonal loci cut out by
$\widehat{\CI}_\pi$, and its composite with $\varpi$ is the morphism $\mu_\pi$. More explicitly, a point of the source is a collection of points $p_{[b,a],t}\in\BP^1$, one for each
occurrence $1\le t\le m_{[b,a]}$ of each interval $[b,a]$.
For every such occurrence, assign the same point $p_{[b,a],t}$ to
the variables of colors $b,b-1,\ldots,a$. The variables belonging to each interval occurrence are the same, so
this point lies on the diagonal locus indexed by the identity element
of $S^\pi$, and hence in the subscheme cut out by
$\widehat{\CI}_\pi$. After applying $\varpi$, the corresponding divisor of color $c$ is
\[\sum_{b\ge c\ge a}\sum_{t=1}^{m_{[b,a]}}[p_{[b,a],t}],\] which is exactly the divisor obtained by first symmetrizing the points
$p_{[b,a],t}$ to obtain an element of $Z_\pi$ and then applying
$\mu_\pi$.
It follows that the
scheme-theoretic image of the subscheme cut out by
$\widehat{\CI}_\pi$ is exactly $X_\pi$.

Finally, by definition, \[\widehat{\CI}_{\le\pi_k}=
\bigcap_{r=k}^{\ell}\widehat{\CI}_{\pi_r}\]
cuts out the scheme-theoretic union of the corresponding closed
subschemes. Taking scheme-theoretic images therefore gives the desired result.    
\end{proof}
\begin{proposition}\label{prop:Q as cap}
For every root partition $\pi$, we have $\widehat{\CQ}\cap(\widehat{\CI}_\pi\otimes\widehat{\CP})=\widehat{\CQ}\otimes\widehat{\CI}_\pi$ as subsheaves of $\widehat{\CP}$, and $\varpi_*(\widehat{\CQ}\otimes\widehat{\CI}_\pi)=\CQ\cap(\CI_{X_\pi}\CP)$ as subsheaves of $\CP$. 

Consequently, $\widehat{\CQ}\cap(\widehat{\CI}_{<\pi_k}\otimes\widehat{\CP})=\widehat{\CQ}\otimes\widehat{\CI}_{<\pi_k}$
and $$\CQ_{\ge\pi_k}=\CQ\cap(\CI_{<\pi_k}\CP).$$
\end{proposition}
The key point of this proposition is that, on the component of $\widehat Y$ indexed by a color sequence $\beta$, the pre-image of $X_\pi$ is a union of diagonal loci. The equation $Q_\beta$ vanishes identically on exactly those diagonal loci which are incompatible with $\beta$, and is a non-zero divisor on the union of the compatible loci. Thus taking the colon ideal by $Q_\beta$ removes precisely the incompatible components, giving the asserted identities\footnote{Another way to think of this is that the list of generators of $J_\lbd$ is a regular sequence, and appending $Q_\beta$ to the end of this list gives still a regular sequence.}.
\begin{proof}
The inclusions $\widehat{\CI}_\pi\subset \CO_{\widehat Y}$ and $\widehat{\CQ}\subset \widehat{\CP}$
give a canonical inclusion
\[
\widehat{\CQ}\otimes\widehat{\CI}_\pi
\hookrightarrow\widehat{\CP}.
\]
Since $\varpi$ is finite, pushing forward gives a canonical inclusion
\[\varpi_*(\widehat{\CQ}\otimes\widehat{\CI}_\pi)\hookrightarrow\varpi_*\widehat{\CP}=\CP.\]
Thus for each of the four equalities, both sides are globally defined
subsheaves of the same sheaf, and it is enough to check that their restrictions
agree on an open cover.

On the dominant affine chart, write
\[\Lbd\coloneqq \BC[\e_{a,i}\colon 1\le a\le N,\ 1\le i\le n_a]=\CO_X(U^+)\]
and \[R\coloneqq \BC[x_{a,i}\colon 1\le a\le N,\ 1\le i\le n_a].
\]
For each color sequence $\beta$, the $\beta$-component of $\varpi^{-1}(U)$ has coordinate ring $R$ and the symmetrizing morphism is induced by the inclusion
$\Lbd\subseteq R$. The ring $R$ is finite free over $\Lbd$.

Consider all ways $\lambda$ of grouping the labeled variables
$x_{a,i}$ into the interval factors prescribed by $\pi$, using every
variable exactly once. For such a grouping, let
$J_\lambda\subseteq R$ be the ideal generated by $x_{c,i_c}-x_{c+1,i_{c+1}}$ whenever the two variables belong to the same interval factor. Clearly $R/J_\lambda$ is a polynomial ring, which is a domain, so $J_\lambda$ is a prime ideal.

Let \[I_{X_\pi} \coloneqq \mathcal I_{X_\pi}(U)\subset \Lbd.\] Then, we claim that \[I_{X_\pi} R=\bigcap_\lambda J_\lambda.\] To see this, consider the finite flat map $\Spec R\to \Spec \Lbd$, and observe that the preimage of $X_\pi \cap U$ in $\Spec R$ is the union of closed subsets $V(J_\lambda)$ set-theoretically.

Since each $J_\lambda$ is a prime ideal, $\bigcap_\lambda J_\lambda$ is a radical ideal, so it suffices to show that $V(I_\pi R)$, namely the preimage of $X_\pi \cap U$, is reduced. Since $X_\pi$ is integral, $\Lbd/I_{X_\pi}$ is a domain. Moreover, the symmetrization map $\Spec R\to \Spec \Lbd$ is finite and generically \'etale (it is \'etale over the locus where the roots of each fixed color are pairwise distinct), so writing $R/I_{X_\pi} R=R\otimes_\Lbd (\Lbd/I_{X_\pi})$, it follows that the generic fiber of $R/I_{X_\pi} R$ over $\Lbd/I_{X_\pi}$ is \'etale over a point and hence reduced.

Now, suppose there exists a nilpotent element of $R/I_{X_\pi} R$. After tensoring with the fraction field of $\Lbd/I_{X_\pi}$ it must vanish, which means that it is killed by a non-zero element of $\Lbd/I_{X_\pi}$. But $R/I_{X_\pi} R$ is also free over $\Lbd/I_{X_\pi}$, so this is impossible. So $R/I_{X_\pi} R$ is reduced, showing that $I_{X_\pi} R$ is indeed $I_{X_\pi} R=\bigcap_\lambda J_\lambda$, i.e. $V(I_{X_\pi} R)$ is the scheme-theoretic union of the $V(J_\lambda)$.

For fixed $\beta$, denote by $I_{\pi, \beta}\subset R$ be the ideal defined by $\widehat \CI_\pi$ on the $\beta$-component. The color sequence $\beta$ orders the variables as follows: the $i$th occurrence of color $a$, read from left to right in $\beta$,
corresponds to $x_{a,i}$. Consider a group of type $[b,a]$, containing variables \[
x_{b,i_b},x_{b-1,i_{b-1}},\ldots,x_{a,i_a}.\] We say that this group is compatible with $\beta$ if these variables
occur in $\beta$ in the displayed order. They need not occur
consecutively. A grouping $\lambda$ is \textit{compatible} with $\beta$ if each of its
groups is compatible.

The elements $w\in S^\pi$ satisfying $w\pi=\beta$ describe exactly
these compatible groupings, so\[
I_{\pi,\beta}
=
\bigcap_{\lambda\text{ compatible with }\beta}J_\lambda.\] Suppose $x_{c,i}$ and $x_{c+1,j}$ belong to the same interval
group of a grouping $\lambda$ compatible with $\beta$. By compatibility with $\beta$, the variable of color $c+1$ occurs to the left of the
variable of color $c$:
\[
x_{c+1,j}, \ldots, x_{c,i}.
\]
On the other hand, the difference
\[
x_{c,i}-x_{c+1,j}
\]
occurs as a factor of $Q_\beta$ only when these variables occur in the
opposite order. 

Modulo $J_\lambda$, the variables in every interval group become one
variable, while variables belonging to different groups remain
independent. Hence every factor of $Q_\beta$ remains a non-zero
polynomial in the polynomial ring $R/J_\lambda$. Therefore
$Q_\beta\notin J_\lambda$. Since $J_\lambda$ is a prime ideal, we obtain
\[
(J_\lambda\colon Q_\beta)=J_\lambda.
\] By compatibility of intersections and colon ideals, we obtain \[(I_{\pi,\beta}\colon Q_\beta)=I_{\pi,\beta}.\] Since $R$ is a domain, for any $r\in R$ we have $I\cap rR = r(I\colon rR)$, so \[I_{\pi,\beta}\cap Q_\beta R=Q_\beta I_{\pi,\beta}.\] On the affine chart $U$, this implies the equality \[
\widehat{\CQ}\cap(\widehat{\CI}_\pi\otimes\widehat{\CP})=\widehat{\CQ}\otimes\widehat{\CI}_\pi.\] For $\lambda$ not compatible with $\beta$, the argument is very similar: In some interval factor, the variables do not occur in descending order. Consequently, for some adjacent colors $c$ and $c+1$, the $c$-colored variable occurs before the $(c+1)$-colored one, and this corresponding difference now is a factor of $Q_\beta$ and belong to $J_\lambda$. This means that we have \[(J_\lambda \colon Q_\beta) = R.\] Then, we obtain \[(I_{X_\pi} R\colon Q_\beta) = \bigcap_\lambda(J_\lambda\colon Q_\beta) = \bigcap_{\lambda \text{ compatible with }\beta}J_\lambda = I_{\pi, \beta}.\] So \[Q_\beta R\cap I_{X_\pi} R = Q_\beta(I_{X_\pi} R\colon Q_\beta) = Q_\beta I_{\pi, \beta}.\] Now, recall that
$\CP=\varpi_*\widehat{\CP}$ and $\CQ=\varpi_*\widehat{\CQ}.$
Over $U$, the sections of the pushforward decompose according to the
components of $\widehat Y$:
\[\Gamma(U,\varpi_*(\widehat{\CQ}\otimes\widehat{\CI}_\pi))=\bigoplus_\beta Q_\beta I_{\pi,\beta}\Pol_\beta.\]
On the other hand, $\CQ(U)=\bigoplus_\beta Q_\beta R\Pol_\beta$ and $(\CI_{X_\pi}\CP)(U)=\bigoplus_\beta I_{X_\pi} R\Pol_\beta.$ Hence
\[(\CQ\cap(\CI_{X_\pi}\CP))(U)=\bigoplus_\beta
(Q_\beta R\cap I_{X_\pi} R)\Pol_\beta=
\bigoplus_\beta Q_\beta I_{\pi,\beta}\Pol_\beta.\]
The two sheaves therefore agree over $U$:\[
\varpi_*(\widehat{\CQ}\otimes\widehat{\CI}_\pi)=\CQ\cap(\CI_{X_\pi}\CP).\] The same calculation applies after taking finite intersections over
all $\pi<\pi_k$. Indeed, by compatibility of intersections and colon ideals, and since
extensions of ideals from $\Lbd$ to $R$ commute with finite intersections ($R$ is flat over $\Lbd$), we have \[\widehat{\CQ}\cap (\widehat{\CI}_{<\pi_k}\otimes\widehat{\CP})=\widehat{\CQ}\otimes\widehat{\CI}_{<\pi_k}\] and \[\CQ_{\ge\pi_k}=\CQ\cap(\CI_{<\pi_k}\CP)\]
over $U$.

For $p\in\BP^1$, let $U_p\subset X$ be the open subset on which none of the divisors has $p$ in its support. These opens clearly cover $X$, and the preceding calculation is transported
from $U$ to any other $U_p$ by the diagonal $\SL_2$-action. Thus for each of the four desired equalities, both sides are globally defined
subsheaves of the same sheaf whose restrictions to an open cover agree. The result follows.
\end{proof}

\section{The KLR sheaf}
Just as the KLR algebra is a subalgebra of $\End_{\Sym}\Pol$, we will construct the KLR sheaf as an appropriate subsheaf of $\shend\CP$. 

\subsection{The sheaf corresponding to the algebra}
Let
\[\CE\coloneqq\{\phi\in \shend\CP:\phi(\CQ)\subseteq \CQ\},\]
namely the sheaf of endomorphisms of $\CP$ which preserve $\CQ$. By Lemma \ref{lem:SL2-equivariance}, $\CE$ is an
$\SL_2$-equivariant subsheaf of $\shend(\CP)$. Moreover, since
$\CP$ is locally free, $\shend(\CP)$ is torsion-free, and hence so is
its subsheaf $\CE$.
\begin{theorem}\label{thm:E=R}
    The sheaf $\CE$ deserves the name of the ``KLR sheaf'', namely
    \[\CE(U)=\CR_\alpha.\]
    The fact that $\Sym$ is central in $\CR_\alpha$ corresponds to the fact that $\CO_X(U)$ is central in $\CE(U)$. 
\end{theorem}

For the rest of this section, let $\Lbd=\Sym$ and let $R=\Pol=\BC[x_{1,1},\cdotsc,x_{1,n_1},\cdotsc,x_{N,1},\cdotsc,x_{N,n_N}]$. Let $S_\alpha=S_{n_1}\times\cdots\times S_{n_N}$. 
\begin{lemma}\label{lem:dimklr}
    The graded dimension of the KLR $e^\gamma \CR_\alpha e^\beta$ is
    \[\dim e^\gamma \CR_\alpha e^\beta = \dim R\cdot q^{-|T_\beta|-|T_\gamma|}\sum_{w\in S_\alpha} q^{-2\ell(w)}\cdot q^{2|T_\gamma\cup w T_\beta|}.\]
\end{lemma}
\begin{proof}
    Since $\CR_\alpha$ has a basis consisting of elements $y^m\psi_u e^\beta$ for $u\in S_n$ and $m=(m_1,\cdotsc,m_n)$, we may compute $\dim e^\gamma \CR_\alpha e^\beta$ as the dimension of $\BC\{y^m\}$ times the dimension of the space of crossings $\psi_u$ from $e^\beta$ to $e^\gamma$, the former of which is accounted for by $\dim R$. The permutations $u\in S_n$ with $\psi_u=e^\gamma \psi_u e^\beta$ are in bijection with $w\in S_\alpha$, as the order of the colors is prescribed already, and the grading is contributed only by 1-color and adjacent 2-color crossings. For each $w$, the corresponding $u$ has $\ell(w)$ 1-color crossings, contributing a factor $q^{-2\ell(w)}$ to the dimension count. The adjacent 2-color crossings occur precisely when either $((a,i),(a+1,j))\in T_\beta$ but $w((a,i),(a+1,j))\not\in T_\gamma$, in which case we have monoidally $\psi_1 e_{a,a+1}$ somewhere, or $w((a,i),(a+1,j))\in T_\gamma$ but $((a,i),(a+1,j))\not\in T_\beta$, in which case we have $\psi_1 e_{a+1,a}$. The set of such pairs is given by the symmetric difference $T_\beta\Delta w^{-1} T_\gamma=(T_\beta\cup w^{-1} T_\gamma)\sm (T_\beta\cap w^{-1} T_\gamma)$. Either way, this further contributes a factor of $q^{|T_\beta\Delta w^{-1} T_\gamma|}=q^{-|T_\beta|-|T_\gamma|+2|T_\beta\cup w^{-1} T_\gamma|}$. Summing over all $w$, we have the desired formula.
\end{proof}

\begin{proof}[Proof of Theorem \ref{thm:E=R}]
    It suffices to show that $\CR_{\alpha}$ is precisely the subspace $E$ of $\End_{\Sym}(\Pol)$ preserving $\bigoplus_\beta Q_\beta\cdot\Pol_\beta$. 

    One direction is easy, namely $\CR_\alpha\subseteq E$. It suffices to show that generators of $\CR$ lie in $E$. But this is clear: a 1-color crossing $1\otimes \psi_1 e_{kk}\otimes 1$ clearly preserves $\Pol_\beta$ and therefore also $Q_\beta\cdot\Pol_\beta$, since for the 1-color crossing to exist the 2 strings of color $k$ must be adjacent, whereupon $Q_\beta$ must be symmetric in those two variables. The crossing of adjacent colors $a\in e^{s_i\beta}\CR e^\beta$ is one of two cases,
    \[a=\begin{cases}
        1\otimes \psi_1 e_{\beta_i,\beta_i+1}\otimes 1\colon Q_\beta\cdot 1_\beta\lmto Q_\beta\cdot 1_{s_i\beta}\in Q_{s_i\beta}\cdot \Pol_{s_i\beta}&\te{if }\beta_{i+1}=\beta_i+1,\\
        1\otimes \psi_1 e_{\beta_i,\beta_i-1}\otimes 1\colon Q_\beta\cdot 1_\beta\lmto (y_i-y_{i+1}) Q_\beta\cdot 1_{s_i\beta}=Q_{s_i\beta} \cdot 1_{s_i\beta}\in Q_{s_i\beta}\cdot \Pol_{s_i\beta}&\te{if }\beta_{i+1}=\beta_i-1
    \end{cases},\]
    where $Q_\beta=Q_{s_i\beta}\cdot (y_i-y_{i+1})$ in the first case and $Q_{s_i\beta}=Q_\beta\cdot (y_i-y_{i+1})$ in the second.
    It is clear that crossings of distant colors preserve $\bigoplus_\beta Q_\beta\cdot\Pol_\beta$, as they act on polynomials only by changing the labeling color sequence. 

    As $\CE$ was a subsheaf of $\shend\CP$, it is torsion-free and finite-rank over $\CO_X$, so that $E=\CE(U)$ is torsion-free and finite-rank over $\Lbd=\CO_X(U)$. 

    To show $\CR_\alpha=E$ we now compare the graded dimensions. Namely, we will show that $\dim e^\gamma \CR_\alpha e^\beta\ge \dim E_{\beta\gamma}$, where 
    \[E_{\beta\gamma}=\{\phi\in\hom_\Lbd(\Pol_\beta,\Pol_\gamma):\phi(Q_\beta\cdot\Pol_\beta)\subseteq Q_\gamma\cdot \Pol_\gamma\}.\] 
    For convenience, let $F_{\beta\gamma}=\{\phi\in\hom_\Lbd(R,R):\phi(Q_\beta\cdot R)\subseteq Q_\gamma\cdot R\}$, so that
    \[E_{\beta\gamma}=q^{r(\gamma)-r(\beta)}F_{\beta\gamma}.\] 
    Let $H\coloneqq \hom_\Lbd(R,R)\cong\NH_{n_1}\otimes\cdots\otimes\NH_{n_N}$, so that elements of $H$ can naturally be written as $\sum_{w\in S_\alpha}a_w\pd_w$. There is a natural filtration $(F_{\beta\gamma})_l$ on $F_{\beta\gamma}$ inherited from the filtration on $H$ defined by letting $H_l$ be the space spanned by elements $\sum_{\ell(w)\le l}a_w\pd_w$; since
    \[\pd_w (Q_\beta\cdot f) = w(Q_\beta)\pd_w(f)+\sum_{\ell(v)<\ell(w)}b_v\pd_v(f),\]
    to say $\phi=\sum_w a_w\pd_w\in(F_{\beta\gamma})_l$ is precisely to say
    \[Q_\gamma\mid w(Q_\beta) a_w\quad\te{for each }\ell(w)=l,\] 
    i.e. 
    \[\frac{Q_\gamma}{\gcd(Q_\gamma,w Q_\beta)}\Bigm| a_w\quad\te{for each }\ell(w)=l.\] 
    Hence there is an injection
    \[\gr_l F_{\beta\gamma} \linj \bigoplus_{\ell(w)=l} \frac{Q_\gamma}{\gcd(Q_\gamma,w Q_\beta)}R\cdot \pd_w.\] 
    Note 
    \begin{align*}
        \deg \frac{Q_\gamma}{\gcd(Q_\gamma,w Q_\beta)}&=2|T_\gamma|-2|T_\gamma\cap w T_\beta|\\
        &=2|T_\gamma|-2(|T_\gamma|+|wT_\beta|-|T_\gamma\cup wT_\beta|)\\
        &=-2|T_\beta|+2|T_\gamma\cup wT_\beta|,
    \end{align*}
    so that summing over $l$ we have
    \[\dim F_{\beta\gamma}\le \dim R\cdot \sum_{w\in S_\alpha} q^{-2|T_\beta|+2|T_\gamma\cup wT_\beta|}q^{-2\ell(w)}.\]
    Since $r(\gamma)+|T_\gamma|=r(\beta)+|T_\beta|$, we have $r(\gamma)-r(\beta)-2|T_\beta|+2|T_\gamma\cup wT_\beta|=-|T_\gamma|-|T_\beta|+2|T_\gamma\cup wT_\beta|$, so that
    \[\dim E_{\beta\gamma}\le \dim R\cdot q^{-|T_\beta|-|T_\gamma|}\sum_{w\in S_\alpha} q^{2|T_\gamma\cup w T_\beta|}q^{-2\ell(w)}=\dim e^\gamma \CR_\alpha e^\beta,\] 
    as desired.

\end{proof}

\begin{remark}
    A previous iteration of this proof used a dimension comparison on an Ext group instead, by using the following lemma, which we record in case it can be useful for applying the program in this paper to other settings.
    \begin{lemma}\label{lem:dimlemma}
    Suppose $\Lbd$ is a graded ring and $V,V',W,W'$ are graded finite-rank modules over $\Lbd$ with $W\subseteq V$ and $W'\subseteq V'$. Let $V$ be projective over $\Lbd$. Let $E\subseteq \hom_\Lbd(V,V')$ be the subspace defined by
    \[E=\{\phi\in\hom_\Lbd(V,V'): \phi(W)\subseteq W'\}.\] 
    Then
    \begin{align*}
        \dim E&=\dim \hom_\Lbd(V,V')-\dim\hom_\Lbd(W,V'/W')+\dim\Ext_\Lbd^1(V/W,V'/W')\\
        &=\rnk V' \rnk V^* \dim\Lbd- \rnk W^*(\rnk V'-\rnk W')\dim\Lbd+\dim\Ext_\Lbd^1(V/W,V'/W').
    \end{align*}
\end{lemma}
\begin{proof}
    Consider the composition of maps
    \[\hom_\Lbd(V,V')\oslto{f} \hom_\Lbd(V,V'/W')\oslto{g} \hom_\Lbd(W,V'/W'),\] 
    where $f$ is induced by the quotient $V'\lsurj V'/W'$ and $g$ is restriction to $W$. Note that $E=\ker gf$. Hence we have the exact sequence
    \[0\lto E\lto \hom_\Lbd(V,V')\lto \hom_\Lbd(W,V'/W')\lto \coker gf\lto 0.\] 

    Applying the exact functor $\hom_\Lbd(V,\sq)$ to $0\lto W'\lto V'\lto V'/W'$ and applying the functor $\rhom_\Lbd(\sq,V'/W')$ to $0\lto W\lto V\lto V/W\lto 0$ show that 
    \begin{align*}
        \coker f&=\ker(\Ext_\Lbd^1(V,W')\lto \Ext_\Lbd^1(V,V')),\\
        \coker g&=\ker(\Ext_\Lbd^1(V/W,V'/W')\lto \Ext_\Lbd^1(V,V'/W')).
    \end{align*} 
    But $V$ is projective, so we know $\coker f=0$ and $\coker g=\Ext_\Lbd^1(V/W,V'/W')$. As $\coker f\lto \coker gf\lto \coker g\lto 0$, we then know
    \[\coker gf=\Ext_\Lbd^1(V/W,V'/W').\]

    Now taking the alternating sum of dimensions of the earlier exact sequence gives the claim.
\end{proof}
\end{remark}

\subsection{The geometric monoidal product}\label{sec:geometric monoidal}
Let $\CE$ be the coherent KLR sheaf on $X$ corresponding to $\alpha$, and similarly for $\CE'$ on $X'$ and $\CE''$ on $X''$ corresponding to $\alpha'$ and $\alpha''$, where if $\alpha=\sum_a n_a\alpha_a$ and $\alpha'=\sum_a n'_a\alpha_a$, we let $\alpha''=\alpha+\alpha'=\sum_a n''_a\alpha_a$ and $n''_a=n_a+n'_a.$ We wish to construct a monoidal product.

Write $(X,\CP,\CQ,\CE)$, $(X',\CP',\CQ',\CE')$, and
$(X'',\CP'',\CQ'',\CE'')$ for the spaces and sheaves associated with
$\alpha,\alpha'$, and $\alpha''$, respectively. Let
\[c\colon X\times X'\lto X''\] 
be defined by applying the concatenation map per color, namely sending $\BP^{n_a}\times\BP^{n_a'}\lto \BP^{n_a''}.$ Note that this map is finite.

Next, set \[\mathcal L_{\alpha,\alpha'}\coloneqq \boxtimes_{a=1}^N \CO_{\BP^{n_a}}(n'_a-n'_{a-1})\] on $X$ and \[\mathcal V_{\alpha,\alpha'}\coloneqq(\CP\otimes\mathcal L_{\alpha,\alpha'})\boxtimes\CP'\]on $X\times X'$.

Recall that $\wh Y''=\bigsqcup_{\gamma\in\Seq_{\alpha''}}Y''_\gamma$
is, by definition, a disjoint union of copies of $Y''$ indexed by
color sequences. Thus any union of these components is open and closed. Denote by $\wh Y''_{\alpha,\alpha'}\subseteq\wh Y''$ the (open and closed) union of the components indexed by $\Seq_{\alpha,\alpha'}$, which is the subset of $\Seq_{\alpha''}$ obtained by concatenating sequences from $\Seq_\alpha$ and $\Seq_{\alpha'}$. The subset $\wh Y''_{\alpha,\alpha'}$ geometrically records the ``concatenated idempotent'' $\sum_{\beta\in\Seq_\alpha,\beta'\in\Seq_{\alpha'}}e^{\beta\beta'},$ since the usual KLR monoidal product sends $e^\beta\otimes e^{\beta'}$ to $e^{\beta\beta'}$ and acts by zero on the other summands.

Define \[\CP''_{\alpha,\alpha'}\coloneqq \varpi''_*(\wh\CP''|_{\wh Y''_{\alpha,\alpha'}})\] and write  \[\CP''=\CP''_{\alpha,\alpha'}\oplus\CP''_{\mathrm{rest}}.\] \begin{lemma}\label{lem:monoidal-polynomial-sheaf}
There is a canonical $\SL_2$-equivariant isomorphism
\[\theta_{\alpha,\alpha'}\colon c_*\mathcal V_{\alpha,\alpha'} \xrightarrow{\sim}\CP''_{\alpha,\alpha'}.\]
\end{lemma}
\begin{proof} Concatenation of color sequences, together with concatenation of the
ordered variables within each color, gives an $\SL_2$-equivariant
isomorphism
\[\wh\sha \colon\wh Y\times\wh Y'\xrightarrow{\sim}\wh Y''_{\alpha,\alpha'}.\]
It remains to compare the corresponding line bundles. 

Fix $\beta\in\Seq_\alpha$ and $\beta'\in\Seq_{\alpha'}$. For a color $a$ variable belonging to the first
block (i.e. the ones coming from $\beta$), passing from $\beta$ to $\beta\beta'$ increases the initial
same-color twist by $n'_a$. At the same time, every one of the
$n'_{a-1}$ strings of color $a-1$ belonging to $\beta'$ now lies to
its right. Thus its total twist changes by $n'_a-n'_{a-1}.$

For a variable belonging to the second block (i.e. from $\beta'$), its index is shifted by
$n_a$, while the number of $(a-1)$-strings to its right is unchanged;
hence its twist is exactly the same as in $\wh\CP'_{\beta'}$. Since the symmetrizing morphism pulls
$\CO_{\BP^{n_a}}(1)$ back to $\CO(1,\ldots,1)$, by using the earlier calculation (namely that every color $a$ variable from $\alpha$ gains an extra $n_a'-n_{a-1}'$ in degree and that nothing changes for the the variables in $\alpha'$) and taking the disjoint union over all $\beta,\beta'$ gives the isomorphism
\[\wh\sha^*(\wh\CP''|_{\wh Y''_{\alpha,\alpha'}})\cong
(\wh\CP\otimes\varpi^*\mathcal L_{\alpha,\alpha'})
\boxtimes\wh\CP'.
\] Set
\[q=\varpi\times\varpi' \text{ and }\mathcal F=(\wh\CP\otimes\varpi^*\mathcal L_{\alpha,\alpha'})\boxtimes\wh\CP'.\]
By the projection formula and compatibility of finite pushforward with
external products, we have
\[q_*\mathcal F
\cong
(\CP\otimes\mathcal L_{\alpha,\alpha'})\boxtimes\CP'
=
\mathcal V_{\alpha,\alpha'}.
\]
Since
$c\circ q=\varpi''_{\alpha,\alpha'}\circ\wh\sha$ and
$\wh\sha$ is an isomorphism, the preceding line bundle
identification gives \[c_*\mathcal V_{\alpha,\alpha'}\cong
\varpi''_{\alpha,\alpha',*}\wh\sha_*\mathcal F\cong
\varpi''_{\alpha,\alpha',*}(
\wh\CP''|_{\wh Y''_{\alpha,\alpha'}})=\CP''_{\alpha,\alpha'}.\]All of these calculations are $\SL_2$-equivariant, and so the result follows.\end{proof}

There is a natural morphism 
\begin{align*}
    \iota_{\alpha,\alpha'}\colon \CE\boxtimes\CE'&\longrightarrow\shend(\mathcal V_{\alpha,\alpha'})\\
    \phi\boxtimes\phi'&\longmapsto(\phi\otimes\id_{\mathcal L_{\alpha,\alpha'}})\boxtimes\phi'.
\end{align*}
Consider the composition \[c_*\shend(\mathcal V_{\alpha,\alpha'})\longrightarrow c_*\mathcal{H}om(c^*c_*\mathcal V_{\alpha,\alpha'}, \mathcal V_{\alpha,\alpha'})\simlto\shend(c_*\mathcal V_{\alpha,\alpha'}),\] where the first arrow is induced by pre-composition with the counit of the adjunction $c^*\dashv c_*.$ By Lemma \ref{lem:monoidal-polynomial-sheaf}, this becomes
\[c_*(\CE\boxtimes\CE')\longrightarrow\shend(\CP''_{\alpha,\alpha'}).\] Finally, extend an endomorphism of $\CP''_{\alpha,\alpha'}$ by zero on $\CP''_{\mathrm{rest}}$, which then gives rise to an $\SL_2$-equivariant morphism of sheaves of algebras
\[\Phi_{\alpha,\alpha'}\colon c_*(\CE\boxtimes\CE')\longrightarrow\shend(\CP'').\] 
\begin{theorem}\label{thm:A_N monoidal}
The morphism $\Phi_{\alpha,\alpha'}$ takes values in $\CE''$ and
therefore defines a $\SL_2$-equivariant morphism \[\otimes\colon c_*(\CE\boxtimes\CE')\longrightarrow\CE''.\]Moreover, these morphisms are associative, and $\CE_0\coloneqq \CO_{\Spec \BC}$ acts as their monoidal unit. 

On the dominant affine chart, this is the usual monoidal product
\[\CR_\alpha\otimes\CR_{\alpha'}\longrightarrow\CR_{\alpha+\alpha'}\]
given by horizontal concatenation of KLR diagrams.
\end{theorem}
\begin{proof}
    Let $\iota\colon\CQ''\hookrightarrow\CP''$ and $q\colon\CP''\twoheadrightarrow\CP''/\CQ''$ denote the inclusion and quotient maps. By definition, \[\CE''=\ker(\shend(\CP'')\longrightarrow\shom(\CQ'',\CP''/\CQ'')),\]where \[f\longmapsto q\circ f\circ\iota.\] Therefore, to show that $\Phi_{\alpha,\alpha'}$ takes values in $\CE''$, it is enough to prove that\[q\circ\Phi_{\alpha,\alpha'}(s)\circ\iota=0\]
for every local section $s$ of $c_*(\CE\boxtimes\CE')$.
Let $U''\subset X''$ be the dominant affine chart, and let $U\subset X$ and $U'\subset X'$ be the corresponding charts. We have
\[c^{-1}(U'')=U\times U'\] because these charts parametrize effective divisors supported on $\BA^1=\BP^1\setminus\{\infty\}$.

On $U$, the non-vanishing homogeneous coordinates $\E_{a,0}$ trivialize every factor of $\mathcal L_{\alpha,\alpha'}$, and hence
\[\mathcal L_{\alpha,\alpha'}|_U\cong\CO_U.\] Therefore Lemma \ref{lem:monoidal-polynomial-sheaf}, restricted to the
components indexed by $\beta$ and $\beta'$, gives
\[\wh\sha^*\widehat\CP''_{\beta\beta'}\cong\widehat\CP_\beta\boxtimes\widehat\CP'_{\beta'}.
\]
Taking sections over the corresponding affine charts identifies
\[\Pol_{\beta\beta'}\cong\Pol_\beta\otimes_{\BC}\Pol_{\beta'},\]
where the variables in the second factor are placed after those in the first factor.  Under this identification, the idempotent $e^{\beta\beta'}$ corresponds to $e^\beta\otimes e^{\beta'}.$ Moreover,
$\Phi_{\alpha,\alpha'}(x\otimes y)$ acts on this summand as the
external tensor product of the actions of $x$ and $y$.  Hence
\[e^{\gamma\gamma'}\Phi_{\alpha,\alpha'}(x\otimes y)e^{\beta\beta'}=
(e^\gamma\otimes e^{\gamma'})(x\otimes y)
(e^\beta\otimes e^{\beta'})=(e^\gamma x e^\beta)\otimes(e^{\gamma'}ye^{\beta'}).\] Moreover, any block not indexed by
concatenated sequences is zero. This is exactly the action on the
polynomial representation of the diagram obtained by placing the
diagrams for $x$ and $y$ side by side.  Therefore, using Theorem
\ref{thm:E=R}, the restriction of $\Phi_{\alpha,\alpha'}$ to $U''$
is precisely the usual KLR homomorphism
\[\CR_\alpha\otimes\CR_{\alpha'}\longrightarrow\CR_{\alpha''}.\]
Its image consequently lies in
$\CE''(U'')=\CR_{\alpha''}$, and hence
$\rho\circ\Phi_{\alpha,\alpha'}$ vanishes over $U''.$

For a closed point $p\in\BP^1$, let
\[U''_p=\{(D_1,\ldots,D_N)\in X''\colon p\notin\Supp(D_a)\text{ for every }a\}.\]The opens $U''_p$ cover $X''$, and each is an $\SL_2$-translate of $U''$. The morphisms $\rho$ and $\Phi_{\alpha,\alpha'}$ are $\SL_2$-equivariant. Therefore the vanishing over $U''$ transports to each $U''_p$. Hence \[\rho\circ\Phi_{\alpha,\alpha'}=0\] globally, proving that $\Phi_{\alpha,\alpha'}$ lands in $\CE''$.

For associativity, consider the two morphisms obtained from three
weights $\alpha_1,\alpha_2,\alpha_3$, after making the canonical
identifications between the two iterated pushforwards. On the
dominant affine chart, both are horizontal concatenation of three KLR
diagrams and hence agree. Their difference is a morphism into the
torsion-free sheaf $\CE_{\alpha_1+\alpha_2+\alpha_3}$ and vanishes on
a dense open subset, so it vanishes globally. The assertion about the
monoidal unit is clear by construction.
\end{proof}

\section{Modules over the KLR sheaf}
Let\[\CM_{\pi_k}\coloneqq\varpi_* (\wh\CQ_{\ge\pi_k}/\wh\CQ_{\ge\pi_{k-1}})\]
be a sheaf on $X$. This subsection leads up to proving that this sheaf is the ``Verma sheaf''. 

\begin{lemma}\label{lem:supp M in X}
    The ideal sheaf $I_{X_\pi}$ of $X_\pi$ annihilates $\CM_\pi$, and hence $\CM_\pi$ is a sheaf on $X_\pi$, i.e., $\Supp \CM_\pi\subseteq X_\pi$. 
\end{lemma}
\begin{proof}
    By Proposition \ref{prop:Q as cap}, and by the exactness of $\varpi_*$, we have
    \[\CM_\pi=\mfrac{\CQ\cap\CI_{<\pi}\CP}{\CQ\cap \CI_{\le\pi}\CP}.\] 
    Let $\pi=\pi_k$. Now note $\CI_{X_{\pi_k}}\cdot \CI_{\le\pi_{k+1}}\subseteq \CI_{X_{\pi_k}}\cap\CI_{\le\pi_{k+1}}=\CI_{\le\pi_k}$.
\end{proof}

\begin{lemma}\label{lem:Qk is submod of Q}
    We have $\CE\cdot (\varpi_*(\wh\CQ\otimes \wh\CI_\pi))\subseteq \varpi_*(\wh\CQ\otimes \wh\CI_\pi)$. Intersecting over all $\pi\le \pi_k$ gives $\CE\cdot \CQ_{\ge\pi}\subseteq \CQ_{\ge\pi}$. In particular, $\CQ_{\ge\pi}(U)$ and $\varpi_*(\wh\CQ\otimes \wh\CI_\pi)(U)$ are $\CR_\alpha$-submodules of $\CQ(U)$. 

\end{lemma}
\begin{proof}
    There is a short proof using Proposition \ref{prop:Q as cap}. Indeed, $\CE\cdot \CQ\subseteq\CQ$ by construction, and $\varpi_*(\wh\CQ\otimes \wh\CI_\pi)=\CQ\cap (\CI_\pi \CP)$, so by $\CO_X$-linearity we have $\CE\cdot (\CI_\pi\CP)\subseteq \CI_\pi\cdot \CE\cdot \CP\subseteq\CI_\pi\CP$. As $\wh\CQ$ is an invertible sheaf, tensoring with it is exact and commutes with all finite intersections, so we have $\bigcap(\wh\CQ\otimes\wh\CI_\pi)=\wh\CQ\otimes\bigcap \wh\CI_\pi$, and hence $\CE\cdot\CQ_{\ge\pi}\subseteq\CQ_{\ge\pi}$. The statements on the affine chart follow. 

\end{proof}

Let $\beta_-=(N^{n_N},(N-1)^{n_{N-1}},\cdotsc,1^{n_1})$ and let \[d_\alpha=\sum_{a=1}^{N-1}n_an_{a+1};\]  note that $d_\alpha=\deg 1_{\beta_-}$ is the degree of the idempotent $\beta_-$ in the polynomial representation, and note $r(\beta)+t_\beta=d_\alpha$ for any $\beta$. \begin{proposition}\label{prop:Q is antidom proj}
    Let $\beta_-=(N^{n_N},(N-1)^{n_{N-1}},\cdotsc,1^{n_1})$ be the color sequence going from $N$ down to $1$, and let $e_{\beta_-}=e_{n_N}(\NH)\otimes\cdots\otimes e_{n_1}(\NH)$ be the monoidal product of the primitive nil-Hecke idempotents, where $e_k(\NH)=\pd_{w_0}\ul y_k$. Then there is an isomorphism of left $\CR_\alpha$-modules
    \begin{align*}
        q^{d_\alpha}\CR_\alpha e_{\beta_-}&\simlto \CQ(U)\\
        r e_{\beta_-}&\lmto r 1_{\beta_-},
    \end{align*}
    where $1_{\beta_-}\in\Pol=\CP(U)$ is the vector of the polynomial representation.
\end{proposition}
\begin{proof}
    Note that the nil-Hecke action on polynomials gives $e_{\beta_-}\cdot 1_{\beta_-}=1_{\beta_-}$, so that the map is well-defined.

    Let $w_\beta$ denote the minimal length permutation such that $w_\beta \beta_-=\beta$; in particular there are no 1-color crossings in $w_\beta$. Since each pair of adjacent strings of colors $a+1,a$ which are reversed by $w_\beta$ then contribute a factor of $(x_{a,i}-x_{a+1,j})$ for the appropriate $i,j$, we then have $\psi_{w_\beta}\cdot 1_{\beta_-}=\pm Q_\beta 1_{\beta}$. Multiplication by dots will commute with $w_\beta$ as there are no 1-color crossings. This shows the map is surjective.

    But any permutation $w$ with $w \beta_-=\beta$ can be written as $w=w_\beta\cdot u$, where $u e^{\beta_-}$ consists only of 1-color crossings; as we have chosen $e_k(\NH)=\pd_{w_0}\ul y_k$, we have $u e_{\beta_-}=0$. Hence the graded dimensions of the two sides can be counted: on the left we have $\dim R\cdot q^{d_\alpha}\sum_{\beta} q^{\deg \psi_{w_\beta}}$, and on the right we have $\dim R\cdot \sum_\beta q^{\deg Q_\beta}\cdot q^{r(\beta)}$. As $r(\beta)+t_\beta=d_\alpha$ and $t_\beta=\deg \psi_{w_\beta}$, we have $\deg Q_\beta+r(\beta)=d_\alpha+\deg\psi_{w_\beta}$ and hence the map must be an isomorphism.
\end{proof}

Recall that 
\[\CR_{\ge\pi_k}=\sum_{i=0}^k \CR_\alpha e^{\pi_i} \CR_\alpha\] 
are the cell ideals of $\CR_\alpha$, with $0\subset \CR_{\ge\pi_0}\subset\cdots\subset\CR_{\ge\pi_\ell}=\CR_\alpha$. Let
\[I_\pi=\wh\CI_\pi(\wh U).\] 
Let $P_-\coloneqq \CR_\alpha e_{\beta_-}$, and consider the filtrations
\begin{align*}
    P_k&=\CR_{\ge\pi_k} e_{\beta_-}=\CR_{\ge\pi_k}\cdot P_-,\\
    Q_k&=q^{-d_\alpha}(\wh \CQ\otimes\wh\CI_{<\pi_k})(\wh U)=\bigcap_{r=k+1}^\ell I_{\pi_r} P_-;
\end{align*}
the former $0\subset P_0\subset\cdots\subset P_\ell=P_-$ is clearly a filtration of $P_-$, and by Proposition \ref{prop:Q is antidom proj} we have the latter is as well, as $\wh\CI_{<\pi_\ell}=\CO$ and $q^{-d_\alpha}\varpi_*(\wh\CQ\otimes\wh\CI_\pi)(U)=q^{-d_\alpha}\CQ(U)\otimes \wh\CI_\pi(\wh U)=I_\pi P_-$. 
\begin{claim}\label{claim:M is Verma 1}
    One has $\tau\in S^\pi\cdot \pi\implies \tau\le\pi$. 
\end{claim}
\begin{proof}
    This follows from the module theory of KLR: if $\tau\in S^\pi\cdot\pi$, then $L_\tau$ appears as a composition factor in $\Delta_\pi$, so necessarily $\tau\le\pi$.
\end{proof}

\begin{claim}\label{claim:M is Verma 2}
    For each $k$, $P_k\subseteq Q_k$. 
\end{claim}
\begin{proof}
    Let $\pi_r<\pi_k$, i.e. $r\ge k+1$. Claim \ref{claim:M is Verma 1} tells us that all $\pi\ge\pi_k$, $\pi\not\in S^{\pi_r}\cdot \pi_r$. From Lemma \ref{lem:Qk is submod of Q} and Proposition \ref{prop:Q is antidom proj} we know that $q^{-d_\alpha}\varpi_*(\wh Q\otimes\wh\CI_{\pi_r})(U)=I_{\pi_r} P_-$ is a $\CR_\alpha$-submodule, so the quotient $P_-/I_{\pi_r}P_-$ receives a $\CR_\alpha$-action. Since our convention for $I_{\pi_r}$ is that the empty intersection gives $\CO$, we know
    \[e^{\pi}\cdot P_-/I_{\pi_r}P_-=0;\] 
    as $\CR_{\ge\pi_k}$ is generated by $e^\pi$ for all $\pi\ge\pi_k$, it follows that
    \[\CR_{\ge\pi_k}\cdot P_-\subseteq I_{\pi_r} P_-.\] 
    Intersecting over all $r\ge k+1$ gives the desired claim.
\end{proof}
Let $s_\pi=\sum_{[b,a]} \binom{m_{[b,a]}(\pi)}{2}$. We take as definition $\Delta_\pi\coloneqq q^{-s_\pi} C_\pi$, where $C_\pi=\CR_{\ge\pi}/\CR_{>\pi} e_\pi$ ($e_\pi\coloneqq \psi_\pi \ul y^\pi e^\pi$) is the left cell module; this grading shift is so that the unique simple head of $\ol\Delta_\pi$ is self-dual.
\begin{lemma}
    We have
    \[P_k/P_{k-1}\cong q^{t_\pi+s_\pi}\Delta_{\pi_k}\]
    as left $\CR_\alpha$-modules. Recall that $t_\pi=|T_\pi|$ is the number of pairs $((a,i),(b,j))$ such that $(a,i)\prec_\pi(b,j)$, and $s_\pi=\sum_{[b,a]} \binom{m_{[b,a]}}{2}$.
\end{lemma}
\begin{proof}
    This is also clear from the module theory of KLR. As $P_k=\CR_{\ge\pi_k}\cdot P_-$, we have $P_k/P_{k-1}=(\CR_{\ge\pi_k}/\CR_{\ge \pi_{k-1}})e_{\beta_-}$ is isomorphic to $\Delta_{\pi_k}$ up to grading shift. The shift $s_\pi$ is standard from KLR theory, e.g. Section 2.4 of \cite{kleshchev2013affine}. The additional shift by $t_\pi$ comes from the choice of idempotent.
    
\end{proof}

By Claim \ref{claim:M is Verma 2}, we may construct a map
\[\vphi\colon q^{t_\pi+s_\pi}\Delta_{\pi_k}=P_k/P_{k-1}\lto Q_k/Q_{k-1}=q^{-d_\alpha}\CM_{\pi_k}(U)\]
commuting with the action of $\Lbd=\CO_X(U)$. 
\begin{proposition}\label{prop:avatars of verma inject}
    The above map is injective.
\end{proposition}
We need some lemmas to help us prove this. Let $K_\pi=\Frac\Lbd_\pi$. By \cite{kleshchev2013affine} we know $\Delta_\pi$ is free of rank $|S^\pi|$ as a right $\Lbd_\pi$-module. Moreover, as $\CR_{>\pi}$ kills $\Delta_\pi$ and all dots on the same $\pi$-cable are identified with each inside $\CR/\CR_{>\pi}$, we know $I_{X_\pi}=\CI_{X_\pi}(U)$ kills $\Delta_\pi$. As 
\[I_{X_\pi}\cdot\Delta_\pi=0,\]
we know the action of $\Lbd\coloneqq \CO_X(U)$ on $\Delta_\pi$ factors through $\Lbd/I_{X_\pi}$:
\[\mu_\pi^\# \colon \Lbd/I_{X_\pi}\lto \Lbd_\pi.\] 
Geometrically this corresponds to the inclusion $\CO_{X_\pi}\subset \mu_{\pi,*} \CO_{Z_\pi}$. By Lemma \ref{lem: finiteness_mu}, we have
\[K_\pi=\Frac\Lbd_\pi=\Frac(\Lbd/I_{X_\pi}).\]
By Lemma \ref{lem:supp M in X}, $I_{X_\pi}$ also kills $\CM_{\pi}(U)$, so the map $\vphi$ is $\Lbd/I_{X_\pi}$-linear. 
\begin{lemma}\label{lem:generic points of strata}
    Let $\eta_\pi$ be the generic point of $X_\pi$. Then $\tau<\pi\implies \eta_\pi\not\in X_\tau$. 
\end{lemma}
\begin{proof}
    First let us show that
    \[\tau<\pi\implies\existss b>a: \sum_{[d,c]\ni b,a}m_{[d,c]}(\tau)>\sum_{[d,c]\ni b,a}m_{[d,c]}(\pi).\] 
    We will exhibit such a pair. Let $[b,a]$ (with $b\ge a$) be the largest root partition such that $m_{[b,a]}(\tau)\neq m_{[b,a]}(\pi)$. Then it is not hard to check that $\tau<\pi$ forces $m_{[b,a]}(\tau)<m_{[b,a]}(\pi)$. Letting $\delta_{[b,a]}=m_{[b,a]}(\tau)-m_{[b,a]}(\pi)$, since $\tau,\pi$ are root partitions of the same $\alpha$, matching the coefficients in front of $\alpha_a$ gives $\sum_{[d,c]\ni a}\delta_{[d,c]}=0$. By the choice of $a,b$, we know $\delta_{[d,c]}=0$ if $d<b$ or if $c<a$. Hence this equation becomes
    \[\delta_{[b,a]}+\sum_{\substack{d>b\\c\le a}}\delta_{[d,c]}=0,\]
    which by $m_{[b,a]}(\tau)<m_{[b,a]}(\pi)$ means
    \[\sum_{\substack{d>b\\c\le a}} \delta_{[d,c]}>0.\]
    Note that we must have $b<N$, since else the sum is empty. Now note $[d,c]\ni a,b+1$ precisely when $d>b$ and $c\le a$. Hence the choice of $b'=b+1,a$ (where $b'>a$) gives $\sum_{[d,c]\ni b',a}m_{[d,c]}(\tau)>\sum_{[d,c]\ni b',a}m_{[d,c]}(\pi)$.

    Recall that points of $X_\pi$ are tuples of length $n_a$ effective divisors on $\BP^1$, one for each color, such that the divisors $D_a,D_b$ for colors $a,b$ share at least $\sum_{[d,c]\ni a,b}m_{[d,c]}(\pi)$ points in common. Hence a general point of $X_\pi$ corresponds to a tuple of divisors such that exactly $\sum_{[d,c]\ni a,b}m_{[d,c]}(\pi)$ points are shared by $D_a,D_b$. In order for $\eta_\pi$ to lie on $X_\tau$, we would hence need $\sum_{[d,c]\ni a,b}m_{[d,c]}(\pi)\ge \sum_{[d,c]\ni a,b}m_{[d,c]}(\tau)$, for all $b>a$. However we have seen in the previous paragraph that this is impossible.



\end{proof}

Following \cite{kleshchev2013affine}, let us denote $\ol\CR=\CR/\CR_{>\pi}$, and let $e_\pi=\psi_\pi\ul y^\pi e^\pi$. Then by \cite[Theorem 5.23]{kleshchev2013affine}, we have 
    \begin{align*}
        e_\pi\ol\CR e_\pi&\cong\Lbd_\pi,\\
        \ol\CR e_\pi\ol\CR&\cong \CR_{\ge\pi}/\CR_{>\pi},\\
        \ol\CR e_\pi&\cong q^{s_\pi}\Delta_\pi,
    \end{align*}
    Let $\eta=\eta_\pi$ be the generic point of $X_\pi$. We denote $\sq_\eta=\sq\otimes_\Lbd K_\pi$; for instance $\ol\CR_{\eta}=\ol\CR\otimes_\Lbd K_\pi$, where the tensor could equivalently be taken over $\Lbd/I_{X_\pi}$.

\begin{lemma}\label{lem:verma is generically simple}
    $\Delta_{\pi}$ is generically simple; that is, $\Delta_{\pi,\eta}=\Delta_{\pi}\otimes_{\Lbd_\pi}K_\pi$ is simple over $\ol\CR_\eta=(\CR_\alpha/\CR_{>\pi})\otimes_{\Lbd} K_\pi$. We have 
    \[\ol\CR_\eta=\ol\CR_\eta e_\pi \ol\CR_\eta = (\CR_{\ge\pi}/\CR_{>\pi})_\eta,\]
    and the cellular form on $\Delta_\pi$ is nondegenerate after base changing to $K_\pi$. In particular $(\CR_{\ge\pi}/\CR_{>\pi})_\eta$ acts faithfully on $\Delta_{\pi,\eta}$. 
\end{lemma}
\begin{proof}
    Fix $\pi=\pi_k$ and let $\eta=\eta_\pi$.  By Lemma \ref{lem:generic points of strata}, for $\tau<\pi$ we have $\eta_\pi\not\in X_\tau$, so that $I_{X_\tau}\not\subseteq I_{X_\pi}$, so that we can find $f_\tau\in I_{X_\tau}\sm I_{X_\pi}$. We have seen that $I_{X_\tau}\cdot\Delta_\tau=0$, so $f_\tau\cdot \CR_{\ge\tau}/\CR_{>\tau}=0$; but $f_\tau$ is nonzero and hence invertible after tensoring up to $K_\pi$. Hence $\CR_{\ge\tau}/\CR_{>\tau}$ must be zero after tensoring to $K_\pi$,
    \[\tau<\pi\implies (\CR_{\ge\tau}/\CR_{>\tau})_\eta=0.\] In fact, as $f_\tau$ is invertible after localizing at $I_{X_\pi}$, we have \[\tau<\pi\implies (\CR_{\ge\tau}/\CR_{>\tau})_{I_{X_\pi}}=0.\]

    As 
    \[\mfrac{\ol\CR}{\ol\CR e_\pi\ol\CR}\cong \CR/\CR_{\ge\pi}\] 
    is filtered with associated gradeds $\CR_{\ge\tau}/\CR_{>\tau}$ for $\tau<\pi$, and since tensoring is right exact and $(\CR_{\ge\tau}/\CR_{>\tau})_\eta=0$, we know
    \[(\CR/\CR_{\ge\pi})_\eta=0.\] Similarly, as localization is exact and $(\CR_{\ge\tau}/\CR_{>\tau})_{I_{X_\pi}}=0$, we know \[(\CR/\CR_{\ge\pi})_{I_{X_\pi}}=0.\] 
    Since $0\lto \CR_{\ge\pi}/\CR_{>\pi}\lto \ol\CR\lto \CR/\CR_{\ge\pi}\lto 0$, we may localize at $I_{X_\pi}$, which is exact, to get $(\CR_{\ge\pi}/\CR_{>\pi})_{I_{X_\pi}}\cong \ol\CR_{I_{X_\pi}}$. Tensoring up to $K_\pi$, this forces $\ol\CR_\eta\cong (\CR_{\ge\pi}/\CR_{>\pi})_\eta=(\ol\CR e_\pi\ol\CR)_\eta$, so that
    \[\ol\CR_\eta=\ol\CR_\eta e_\pi\ol\CR_\eta.\] 
    Hence $e_\pi$ is a full idempotent of $\ol\CR_\eta$. Moreover, 
    \[e_\pi \ol\CR_\eta e_\pi =(\Lbd_\pi)_\eta=K_\pi,\] 
    so that by Morita theory we have
    \[\ol\CR_\eta\cong \End_{K_\pi}(\ol\CR_\eta e_\pi)=\Mat_{|S^\pi|} K_\pi,\] 
    where $\ol\CR_\eta e_\pi$ is of rank $|S^\pi|$ over $K_\pi$. Hence $(\Delta_\pi)_\eta = q^{-s_\pi}\ol\CR_\eta e_\pi$ is simple over $\ol\CR_\eta$. 

    The cellular form after base change is given by
    \[\Delta_{\pi,\eta}\times\Delta_{\pi,\eta}\lto \Delta_{\pi,\eta}^\top\times\Delta_{\pi,\eta}\lto K_\pi,\] 
    where the first map is given by the flipping-upside-down involution and the last map is given by multiplication. As the radical is a $\ol\CR_\eta$-submodule and $\Delta_{\pi,\eta}$ is simple over $\ol\CR_\eta$, the form is either nondegenerate or identically zero; but it cannot be identically zero as $\wan{e_\pi;\ul y^\pi \psi_\pi \ul y^\pi e^\pi}=1$. Hence it is nondegenerate. 
\end{proof}

\begin{proof}[Proof of Proposition \ref{prop:avatars of verma inject}]
    Let $\pi=\pi_k$, $\eta=\eta_\pi$. First let us show that $$(Q_k/Q_{k-1})_{\eta}=(Q_k/Q_{k-1})\otimes_{\Lbd/I_{X_\pi}} K_{\pi}\neq 0.$$
    Recall that $\wh Y_\tau\subset \wh Y$ denotes the subscheme cut out by $\wh I_\tau$. Proposition \ref{prop:Q is antidom proj} and the exactness of $\varpi_*$ imply
    \[P_-/I_\tau P_-\cong \varpi_*(\wh\CQ\otimes\CO_{\wh Y_\tau})(U);\]
    Lemma \ref{lem:Yhat filt vs X filt} then gives
    \[\Supp P_-/I_\tau P_-= X_\tau\cap U.\] 
    Lemma \ref{lem:generic points of strata} says $r>k$ has $\eta\not\in X_{\pi_r}$, so $(I_{\pi_r}\cdot P_-)_\eta=(P_-)_\eta$. As localization commutes with finite intersections, we then have
    \begin{align*}
        (Q_k)_\eta&=\bigcap_{r=k+1}^\ell (I_{\pi_r}\cdot P_-)_\eta=\bigcap_{r=k+1}^\ell (P_-)_\eta=(P_-)_\eta,\\
        (Q_{k-1})_\eta&=(I_\pi P_-)_\eta\cap (P_-)_\eta = (I_\pi P_-)_\eta,
    \end{align*}
    so that right exactness of tensor products gives
    \[(Q_k/Q_{k-1})_\eta=(P_-/I_{\pi_k} P_-)_\eta\neq 0,\] 
    where the last equality follows from $\Supp P_-/I_\pi P_-=X_\pi\cap U$.

    Second let us show that the localized map
    \[\vphi_\eta\colon (P_k/P_{k-1})_{\eta}\lsurj (Q_k/Q_{k-1})_{\eta}\] 
    is surjective. Indeed, note that $\CR_{>\pi}\cdot Q_k\subseteq \CR_{\ge\pi_{k-1}}\cdot P_-=P_{k-1}\subseteq Q_{k-1}$, so that $\ol\CR=\CR/\CR_{>\pi}\actson Q_k/Q_{k-1}$, so that $\ol\CR_\eta\actson (Q_k/Q_{k-1})_\eta$. If $r\in R_{\ge\pi},\ q\in Q_k$, then $rq\in \CR_{\ge\pi}\cdot P_-=P_k$, so that
    \[(\CR_{\ge\pi}/\CR_{>\pi})_\eta\cdot (Q_k/Q_{k-1})_\eta\subseteq\Img\vphi_\eta.\] 
    However, we have seen in the course of Lemma \ref{lem:verma is generically simple} that $(\CR_{\ge\pi}/\CR_{>\pi})_\eta=\ol\CR_\eta$; hence $\Img\vphi_\eta=(Q_k/Q_{k-1})_\eta$.

    As $\vphi_\eta$ is a surjection from a simple onto a nonzero space, it must be an isomorphism and in particular injective. As localization is exact, we have $\ker(\vphi)\otimes_{\Lbd/I_{X_{\pi_k}}}K_{\pi_k}=(\ker(\vphi))_{\eta}=\ker(\vphi_\eta)=0$, so that $\ker\vphi$ has torsion over $\Lbd/I_{X_{\pi_k}}$. Yet $\ker\vphi$ is also a submodule of the torsion-free $\Lbd/I_{X_{\pi_k}}$-module $P_k/P_{k-1}$, and hence $\ker\vphi=0$.  
\end{proof}

We are now ready to see why the sheaf $\CM_\pi$ should be called the ``Verma sheaf''.
\begin{theorem}\label{thm:M is Verma on affine}
    Let $\pi=\pi_k$. The sheaf $\CM_{\pi}$ recovers the Verma modules. Namely, the affine sections recover
    \[\Phi_{\pi}\colon\CM_{\pi}(U)\simlto q^{d_\alpha+t_\pi+s_\pi}\Delta_{\pi}\]
    as $(\CR_\alpha,\Lbd_{\pi})$-bimodules.
    In particular, under the natural identification
\[\Lbd_{\pi}\cong \Gamma(\mu_{\pi}^{-1}(U),\CO_{Z_{\pi}})=\mu_{\pi,*}\CO_{Z_\pi}(U),\]the right $\Lbd_{\pi}$-action on $\Delta_{\pi}$ is compatible
with ordinary scalar multiplication from $\CO_{X_{\pi}}$, i.e. for $f\in\Gamma(X_{\pi}\cap U,\CO_{X_{\pi}})$ and $m\in\CM_{\pi}(U),$ we have 
\[
\Phi_{\pi}(fm)=\Phi_{\pi}(m)\,\mu_{\pi}^{\#}(f).\]
\end{theorem}
\begin{proof}
    Proposition \ref{prop:avatars of verma inject} gives an injection $q^{t_\pi+s_\pi}\Delta_{\pi}\linj q^{-d_\alpha}\CM_\pi(U)$. As $P_\blt,Q_\blt$ are filtrations of $P_-$, we have the equality of graded dimensions
    \[\dim P_-=\sum_k \dim P_k/P_{k-1}=\sum_k \dim Q_k/Q_{k-1}.\] 
    However Proposition \ref{prop:avatars of verma inject} implies a priori that each $\dim P_k/P_{k-1}\le\dim Q_k/Q_{k-1}$; hence the dimensions must be equal, and the injection must be an isomorphism:
    \[q^{t_\pi+s_\pi}\Delta_{\pi}\simlto q^{-d_\alpha}\CM_\pi(U).\] 

    So far we have matched the $\CR_\alpha$-action. Let $\CM_\pi(U)\ractson \Lbd_\pi$ via the isomorphism $\Phi$, in which case $\Phi$ is a morphism of bimodules by definition. Recall $\ol\CR=\CR/\CR_{>\pi}$; as $\Lbd_\pi=e_\pi \ol\CR e_\pi$ and in this quotient all dots on the same cable are identified, we have that this action sends $\e_{c,i}\in\Lbd$ to 
    \[\e_i(z_{[b,a],j}:[b,a]\ni c,\ j\in[1,m_{[b,a]}]),\]
    where $z_{[b,a],j}$ is the variable on the $j$-th $[b,a]$-cable. It remains to check this action is compatible with the natural $\CM_\pi(U)\ractson \CO_{X_\pi}(U)=\Lbd/I_{X_\pi}$ action, i.e. that this agrees with $\mu_\pi^\#$. Let $\f_{[b,a],i}=e_i(z_{[b,a],i}:i\in[1,m_{[b,a]}])$, and let
    \begin{align*}
        F_{[b,a]}(t)&=\prod_{i=1}^{m_{[b,a]}} (t+z_{[b,a],i})=\sum_{i=0}^{m_{[b,a]}}\f_{[b,a],i}t^{m_{[b,a]}-i},\\
        E_c(t)&=\prod_{i=1}^{n_c} (t+x_{c,i})=\sum_{i=0}^{n_c}\e_{c,i} t^{n_c-i}.
    \end{align*}
    Then, since $\mu_\pi(D)_c=\sum_{[b,a]\ni c} D_{[b,a]}$, we know
    \[\mu_\pi^\# E_c(t)=\prod_{[b,a]\ni c}F_{[b,a]}(t).\] 
    Comparing the coefficient of $t^{n_c-i}$ gives
    \[\mu_\pi^\#\e_{c,i}=\sum_{\substack{(i_{[b,a]})_{[b,a]\ni c}:\\0\le i_{[b,a]}\le m_{[b,a]}\\ \sum_{[b,a]\ni c} i_{[b,a]}=i}}\prod_{[b,a]\ni c}\f_{[b,a],i_{[b,a]}};\] 
    as the outer sum is over all ways to split the $i$ powers across the different cables containing $c$, this is equal to 
    \[\mu_\pi^\#(\e_{c,i})=\e_i(z_{[b,a],j}:[b,a]\ni c,\ j\in[1,m_{[b,a]}]),\] 
    as claimed.
\end{proof}

\begin{corollary}
    The sheaf $\CQ_{\ge\pi}$ recovers the cellular filtration pieces of the antidominant projective. In particular, 
    \[\CQ_{\ge\pi}(U)\cong q^{d_\alpha}\CR_{\ge\pi} e_{\beta_-}.\]
\end{corollary}
\begin{proof}
    A similar argument to the proof of Theorem \ref{thm:M is Verma on affine} shows that in fact the inclusion $P_k\linj Q_k$ of Claim \ref{claim:M is Verma 2} must be an isomorphism.
\end{proof}

\begin{proposition}\label{prop:Mpi-pushforward} The $\mathcal{O}_{X_{\pi_k}}$-module structure on $\mathcal{M}_{\pi_k}$ extends uniquely to a $\mathcal{B}_{\pi_k}$-module structure. Moreover, $\mathcal{M}_{\pi_k}$ is locally free over $\mathcal{B}_{\pi_k}$. Let the corresponding sheaf on $Z_{\pi_k}$ be denoted $\widetilde{\CM}_{\pi_k}$. 

Then, there is a canonical isomorphism \[(\mu_{\pi_k})_*\widetilde{\CM}_{\pi_k}\cong \CM_{\pi_k}\] and $\mathcal{\widetilde{\CM}}_{\pi_k}$ is moreover locally free.

Consequently, we also have \[\Supp \CM_{\pi_k} = X_{\pi_k}.\]
\end{proposition}
\begin{proof}
    By Theorem \ref{thm:M is Verma on affine}, restricted to $U$, we know that $\CM_{\pi_k}$ is equipped with a $\mathcal{B}_{\pi_k}$-action that extends that of the natural $\CO_{X_{\pi_k}}$-action. Since $\Delta_{\pi_k}$ is finite free over $\Lbd_{\pi_k}$, it is also clear that $\CM_{\pi_k}$ is locally free over $\CB_{\pi_k}$ on this chart. 

    Now, for any closed point $p\in\BP^1$, let \[U_p=\{(D_1,\ldots, D_N)\in X : p\not\in \Supp D_a \text{ for all }a.\}\] Then, the union of $U_p$ over $p$ covers $X$ and each is a translate of $U$ under the diagonal $\SL_2$-action. Therefore, by $\SL_2$-equivariance, it is clear that $U_p\cap X_{\pi_k}$ also inherits the $\SL_2$-action and local freeness, which shows that $\CM_{\pi_k}$ is locally free over $\CB_{\pi_k}$.

    Next, we show that the transported $\SL_2$-actions agree on overlaps. In fact, there can actually only be at most one $\mathcal{B}_{\pi_k}$-action that extends the ordinary scalar action: Working locally, take $b$ a meromorphic function of $X_{\pi_k}$, and clear denominators by choosing a non-zero local function $s$ such that $sb\in \CO_{X_{\pi_k}}$. Then, the difference between any two actions of $b$ is killed by $s$, which by torsion-freeness means that the two actions must be the same. 

    Finally, by Lemma \ref{lem: finiteness_mu}, we have $\mu_{\pi_k}\colon Z_{\pi_k}\to X_{\pi_k}$ is finite, hence affine, so \[Z_{\pi_k} = \underline{\Spec}_{X_{\pi_k}}\CB_{\pi_k}\]by \cite[Tag 01S8]{stacks}. Also, by the proof of \cite[Tag 01SB]{stacks}, there is an equivalence between quasicoherent modules over $\CO_{Z_{\pi_k}}$ and quasicoherent modules over $\CB_{\pi_k}$ with \[\widetilde{\CM}_{\pi_k}\coloneqq \mu_{\pi_k}^*\CM_{\pi_k}
\otimes_{\mu_{\pi_k}^*\CB_{\pi_k}}
\CO_{Z_{\pi_k}}\] such that there is a canonical isomorphism \[\mu_{\pi_k,*}\widetilde{\CM}_{\pi_k}\cong \CM_{\pi_k}.\] By construction and because $\CM_{\pi_k}$ is locally free, it is clear that $\widetilde{\CM}_{\pi_k}$ is also locally free. 
\end{proof}

It is moreover possible to recover the small Verma, or proper standard module.
\begin{theorem}\label{thm:small verma}
    The ``fiber at the origin'' recovers
    \[\big(\CM_{\pi}\otimes_{\CB_\pi} \mu_{\pi,*}\bk_0\big)(U)=\mu_{\pi,*}\big(\wt\CM_\pi\otimes_{\CO_{Z_\pi}}\bk_0\big)(U)=q^{d_\alpha+t_\pi+s_\pi}\ol\Delta_{\pi},\]
    where $\bk_0$ is the skyscraper at the origin of $Z_\pi$. 
\end{theorem}
\begin{proof}
    Since $\mu_\pi$ is finite (\ref{lem: finiteness_mu}), we have $\mu_{\pi,*}(\wt\CM_\pi\otimes_{\CO_{Z_\pi}}\bk_0)=\mu_{\pi,*}\wt\CM_\pi\otimes_{\mu_{\pi,*}\CO_{Z_\pi}} \mu_{\pi,*}\bk_0$, giving the first equality. Both of these in turn give the small Verma (up to shift) because $(\CM_{\pi}\otimes_{\CB_\pi} \mu_{\pi,*}\bk_0)(U)=\CM_\pi(U)\otimes_{\CB_\pi(U)} \BC=q^{d_\alpha+t_\pi+s_\pi}\Delta_\pi\otimes_{\Lbd_\pi} \BC=q^{d_\alpha+t_\pi+s_\pi}\ol\Delta_\pi$. 
\end{proof}

\section{The geometric cellular structure}
In fact, moreover the cellular filtration of \cite{kleshchev2013affine} can be recovered geometrically in this framework. Let
\[\CE_{\ge\pi_k}\coloneqq\{\phi\in \shend\CP:\phi(\CQ)\subseteq \CQ_{\ge\pi_k}\}.\]
Then the sheaves $\CE_{\ge\pi_k}$ are a geometric version of the cellular filtration of \cite{kleshchev2013affine}. To be more precise,
\begin{theorem}\label{thm:cellular E}
    We have
    \[\CE_{\ge \pi}(U)=\CR_{\ge\pi}.\] 
\end{theorem}
To prove this, let us first develop some lemmas.
\begin{claim}\label{claim:cellular E 1}
    $\CE_{\ge\pi_k}$ is a filtration of $\CE$ by 2-sided ideals.
\end{claim}
\begin{proof}
    By Lemma \ref{lem:Qk is submod of Q}, we have $(\CE\cdot \CE_{\ge\pi_k})\cdot\CQ\subseteq\CE\cdot\CQ_{\ge\pi_k}\subseteq\CQ_{\ge\pi_k}$, which shows $\CE_{\ge\pi_k}$ is a left ideal; and $(\CE_{\ge\pi_k}\cdot\CE)\cdot\CQ\subseteq\CE_{\ge\pi_k}\cdot\CQ\subseteq\CQ_{\ge\pi_k}$ is obvious, which shows $\CE_{\ge\pi_k}$ is a right ideal.
\end{proof}

\begin{claim}\label{claim:cellular E 2}
    We have a natural inclusion $\CR_{\ge\pi_k}\subseteq \CE_{\ge\pi_k}(U)$.
\end{claim}
    \begin{proof}
    Indeed, from KLR theory/cellular theory, we know that $\CR_{\ge\pi_k}$ kills $\Delta_{\pi_r}$ for $r>k$; but by Theorem \ref{thm:M is Verma on affine} such Vermas are $\CQ_{\ge\pi_r}(U)/\CQ_{\ge\pi_{r-1}}(U)$, so that $\CR_{\pi_k}\cdot\CQ_{\ge\pi_r}(U)\subseteq\CQ_{\ge\pi_{r-1}}(U)$ for $r>k$. Applying this several times gives us $\CR_{\ge\pi_k}^{\ell-k}\cdot \CQ (U)\subseteq\CQ_{\ge\pi_k}(U)$; however, as KLR is affine quasihereditary, all cell ideals are idempotent, so this implies $\CR_{\ge\pi_k}\cdot \CQ (U)\subseteq\CQ_{\ge\pi_k}(U)$. By definition of $\CE_{\ge\pi_k}$ this gives the claim.
\end{proof}

\begin{claim}\label{claim:cellular E 3}
    The natural action map $$\CR_{\ge\pi_k}/\CR_{\ge\pi_{k-1}}\linj \End_{\Lbd_{\pi_k}}(\Delta_{\pi_k})$$ is injective.
\end{claim}
\begin{proof}
    Let $\pi=\pi_k$. As $\Delta_\pi$ is finite free over $\Lbd_\pi$, we have
    \[\End_{K_\pi}(\Delta_{\pi,\eta})=(\End_{\Lbd_\pi}(\Delta_\pi))_\eta.\] 
    By Lemma \ref{lem:verma is generically simple}, we have an isomorphism 
    \[\ol\CR_\eta\simlto \End_{K_\pi}(\Delta_{\pi,\eta})=(\End_{\Lbd_\pi}(\Delta_\pi))_\eta\] 
    where the map comes from the action map $\rho$. Since $(\ker\rho)_\eta=0$, $\ker\rho$ must be a torsion $\Lbd_\pi$-submodule of $\CR_{\ge\pi}/\CR_{>\pi}$, which is free over $\Lbd_\pi$; hence $\ker\rho=0$. 
\end{proof}

\begin{claim}\label{claim:cellular E 4}
    The natural maps $$\CR_{\ge\pi_k}/\CR_{\ge\pi_{k-1}}\linj \CE_{\ge\pi_k}(U)/\CE_{\ge\pi_{k-1}}(U)$$ constructed from inclusions of Claim \ref{claim:cellular E 2} are injections. Hence we have the inequality of graded dimensions $$\dim \CR_{\ge\pi_k}/\CR_{\ge\pi_{k-1}}\le\dim \CE_{\ge\pi_k}(U)/\CE_{\ge\pi_{k-1}}(U).$$
\end{claim}
\begin{proof}
    The kernel of such a map is $(\CR_{\ge\pi_k}\cap \CE_{\ge\pi_{k-1}}(U))/\CR_{\ge\pi_{k-1}}$, where the intersection is taken in $\CE_{\ge\pi_k}(U)/\CR_{\ge\pi_{k-1}}$. Suppose there an element $r$ in this kernel; as $r$ is the image of something in $\CE_{\ge\pi_{k-1}}(U)$, we have $r\cdot\CQ_{\ge\pi_k}\subseteq\CQ_{\ge\pi_{k-1}}$, and hence $r$ kills $\CQ_{\ge\pi_k}/\CQ_{\ge\pi_{k-1}}=\CM_{\pi_k}(U)=\Delta_{\pi_k}$. On the other hand, $r\in\CR_{\ge\pi_k}/\CR_{\ge\pi_{k-1}}$, so if it were nonzero then it must not kill $\Delta_{\pi_k}$. Hence $r$ must be zero.
\end{proof}
\begin{proof}[Proof of Theorem \ref{thm:cellular E}]
    Let $\pi=\pi_k$. We have that $\CR_{\ge\pi_k}$ and $\CE_{\ge\pi_k}(U)$ give filtrations of $\CR_\alpha=\CE(U)$, so that
    \[\dim \CR_\alpha=\sum_k \dim \CR_{\ge\pi_k}/\CR_{\ge\pi_{k-1}}=\sum_k \dim \CE_{\ge\pi_k}(U)/\CE_{\ge\pi_{k-1}}(U).\]
    But we already knew $\dim \CR_{\ge\pi_k}/\CR_{\ge\pi_{k-1}}\le\dim \CE_{\ge\pi_k}(U)/\CE_{\ge\pi_{k-1}}(U)$ from Claim \ref{claim:cellular E 4}; hence those inequalities are all equalities, and the injections 
    \[\CR_{\ge\pi_k}/\CR_{\ge\pi_{k-1}}\simlto \CE_{\ge\pi_k}(U)/\CE_{\ge\pi_{k-1}}(U)\]
    from \ref{claim:cellular E 4} are isomorphisms. 

    Then $\dim \CE_{\ge\pi_k}(U)=\sum_{r=0}^k \dim \CE_{\ge\pi_r}(U)/\CE_{\ge\pi_{r-1}}(U)=\sum_{r=0}^k\dim\CR_{\ge\pi_r}/\CR_{\ge\pi_{r-1}}=\dim\CR_{\ge\pi_k}$, so that the inclusions $\CR_{\ge\pi_k}\subseteq \CE_{\ge\pi_k}(U)$ are equalities. This concludes.
\end{proof}

\begin{corollary}
    Each $\CE_{\ge\pi}$ contains an idempotent $e^\pi\in\CE_{\ge\pi}$ and hence $\CE_{\ge\pi}^2=\CE_{\ge\pi}$. 
\end{corollary}
\begin{proof}
    On the affine level we already know $\CE_{\ge\pi}(U)=\CR_{\ge\pi}$ contains the idempotent $e^\pi$, so $\CE_{\ge\pi}(U)^2=\CE_{\ge\pi}(U)$. It is easy to see that $e^\pi$ can be extended to a global section, for example by using the diagrammatics in Theorem \ref{thm:laurent KLR}; alternatively one could use a $\SL_2$-globalization trick, similarly to the proof of Proposition \ref{prop:Mpi-pushforward}.
\end{proof}

Let
\[\CM_{\pi}^\top=\shom_{\CE/\CE_{>\pi}}(\CM_{\pi},\CE_{\ge\pi}/\CE_{>\pi}),\] and recall that $\Delta_\pi^\tau\coloneqq q^{s_\pi}e_\pi \CR/\CR_{>\pi}$ is the right (big) Verma module.
\begin{proposition}\label{prop:M top is right Verma}
    $\CM_{\pi}^\top(U)=q^{-d_\alpha-t_\pi-s_\pi}\Delta_{\pi}^\top$ recovers the right Verma module of KLR.
\end{proposition}
\begin{proof}
    We make free use of the affine identifications we have made so far. Taking affine sections, we have $\CM_\pi(U)^\top=\hom_{\CR/\CR_{>\pi}}(q^{d_\alpha+t_\pi+s_\pi}\Delta_\pi,\CR_{\ge\pi}/\CR_{>\pi})=q^{-d_\alpha-t_\pi-s_\pi}\hom_{\ol\CR}(\ol\CR e_\pi, \ol\CR e_\pi\ol\CR)=q^{-d_\alpha-t_\pi-s_\pi}e_\pi\ol\CR=q^{-d_\alpha-t_\pi-s_\pi}\Delta_\pi^\top$. 
\end{proof}

We can now get the geometric version of the cell decomposition theorem:
\begin{theorem}
    $\CE_{\ge\pi}/\CE_{>\pi}$ is also supported on $X_{\le\pi}$, and we have
    \[\CE_{\ge\pi}/\CE_{>\pi}\cong \CM_{\pi}\otimes_{\CB_{\pi}} \CM_{\pi}^\top\] 
    as sheaves on $X_{\le\pi}$. 
\end{theorem}
\begin{proof}
    Let $\pi=\pi_k$. By Theorem \ref{thm:M is Verma on affine}, we have compatible identifications $\CM_{\pi_k}(U)\cong\Delta_{\pi_k}$ and $\CB_{\pi_k}(U)\cong\Lbd_{\pi_k}.$ Moreover, by Proposition \ref{prop:M top is right Verma}, we have $\CM_{\pi_k}^{\top}(U)\cong\Delta_{\pi_k}^{\top},$ and under these identifications the restriction of the evaluation morphism \[\operatorname{ev}_k\colon\CM_{\pi_k}\otimes_{\CB_{\pi_k}}\CM_{\pi_k}^{\top}
\to\CE_{\ge \pi_k}/\CE_{\ge \pi_{k-1}},\]is the affine cellular isomorphism \[\Delta_{\pi_k}\otimes_{\Lbd_{\pi_k}}\Delta_{\pi_k}^{\top}\xrightarrow{\sim}\CR_{\geq\pi_k}/\CR_{\geq\pi_{k-1}}.\] For every $p\in \BP^1$, recall $U_p$ from the proof of Proposition \ref{prop:Mpi-pushforward}, which is in particular an $\SL_2$-translate of $U$. Equivariance therefore shows that
$\operatorname{ev}_k$ is an isomorphism over every $U_p$. Since the
opens $U_p$ cover $X$, the evaluation map $\operatorname{ev}$ is an isomorphism
globally.

Finally, the right-hand side is scheme-theoretically supported on
$X_{\pi_k}$, and therefore on $X_{\leq\pi_k}$. The same is consequently
true of the left-hand side.
\end{proof}

\section{Action by $\sl_2$ and Witt}\label{subsect:sl2actionKLR}
In this section we give some diagrammatic consequences of our work thus far. We begin by recovering some known phenomena.

\subsection{The Elias-Qi $\sl_2$-action}
As the sheaves $\wh\CQ_{\ge\pi_k}$ and $\wh\CP$ are all built out of line bundles on products of $\BP^1$'s, they automatically receive an action of $\sl_2$ (cf. Lemma \ref{lem:SL2-equivariance}). There is then a natural induced $\sl_2$-action on their pushforwards $\CP$ and $\CQ_{\ge\pi_k}$ on $X$. As the sheaf $\CE$ is built from these sheaves, it also admits a $\sl_2$-action.
\begin{theorem}\label{thm:A_N sl2 action}
    The natural action $\sl_2\actson \CE$ arising from the action $\sl_2\actson \wh\CP$ agrees with the Elias-Qi action, i.e. $\sl_2\actson \CE(U)=\CR_\alpha$. 

    In fact, the geometric cell ideals $\CE_{\ge\pi}$ are $\sl_2$-subsheaves of $\CE$, and hence $\CR_{\ge\pi}\subseteq \CR_\alpha$ are $\sl_2$-submodules of KLR.
\end{theorem}
\begin{proof}
    As the geometric monoidal product is $\SL_2$-equivariant, it suffices to check on monoidal generators. 

    The 1-color case has already been checked in Theorem \ref{thm:sl2action}. To check the multi-color case, it remains to check the action on the multi-color monoidal generators. Keeping the conventions from the proof of \ref{thm:sl2action}, we use algebraic symbols for vectors in the polynomial representation and diagrams for elements of the algebra $\CE$. 

    Let us first do nearby crossings. For $\alpha=\alpha_1+\alpha_2$ the sheaf $\CP$ involved is $\CO(0,-1)1_{21}\oplus\CO(0,0)1_{12}$, where the idempotent is added for clarity. Then one can compute for instance (we choose the only non-trivial examples here to demonstrate) that $e\cdot 1_{21}=(e\cdot_{\CO(0)} 1)\boxtimes 1+1\boxtimes (e\cdot_{\CO(-1)} 1)=0+1\boxtimes (-y)=-y 1_{21}$ and $h\cdot 1_{21}=(h\cdot_{\CO(0)} 1)\boxtimes 1+1\boxtimes (h\cdot_{\CO(-1)} 1)=0+1=1_{21}$. In all, the relevant computations are that
    \begin{align*}
        e\cdot 1_{12}&=h\cdot 1_{12}=f\cdot 1_{12}=0,\\
        e\cdot 1_{21}&=-y1_{21},\\
        h\cdot 1_{21}&=1_{21},\\
        f\cdot 1_{21}&=0.
    \end{align*}
    It is not hard to check that $e((x-y)1_{12})=(-x^2+y^2)1_{12}$, either by using the Leibniz rule or by computing directly. 

    To check the action on diagrams, it again suffices to check the action on a basis of the representation over the center. In this case there is only basis element for each idempotent. 
    \begin{align*}
        (e\cdot\tikzcrossbr)(1_{12})&=e\cdot\tikzcrossbr(1_{12})-\tikzcrossbr(e\cdot 1_{12})=e\cdot 1_{21}-0=-y1_{21},\\
        &\hspace*{275pt}\implies e\cdot\tikzcrossbr\:=-\:\begin{diagram}
            \draw[black](0,0)--(1.5,1.5);
            \draw[red](1.5,0)--(0,1.5);
            \fill[red] (0.35,1.15) circle(6pt);
        \end{diagram}\:;\\
        (h\cdot\tikzcrossbr)(1_{12})&=h\cdot\tikzcrossbr(1_{12})-\tikzcrossbr(h\cdot 1_{12})=h\cdot 1_{21}-0=1_{21},\\
        &\hspace*{275pt}\implies h\cdot\tikzcrossbr\:=\:\tikzcrossbr\:;\\
        (f\cdot\tikzcrossbr)(1_{12})&=f\cdot\tikzcrossbr(1_{12})-\tikzcrossbr(f\cdot 1_{12})=f\cdot 1_{21}-0=0,\\
        &\hspace*{275pt}\implies f\cdot \tikzcrossbr\:=0.
    \end{align*}
    For the other crossing, 
    \begin{align*}
        (e\cdot\tikzcrossrb)(1_{21})&=e\cdot\tikzcrossrb(1_{21})-\tikzcrossrb(e\cdot 1_{21})=e\cdot ((x-y)1_{12})-\tikzcrossrb(-y1_{21})\\
        &=(-x^2+y^2)1_{12}-(-y)(x-y)1_{12}=-x(x-y)1_{12}\\
        &\hspace*{275pt}\implies e\cdot\tikzcrossrb\:=-\:\begin{diagram}
            \draw[red](0,0)--(1.5,1.5);
            \draw[black](1.5,0)--(0,1.5);
            \fill[black] (0.35,1.15) circle(6pt);
        \end{diagram}\:;\\
        (h\cdot\tikzcrossrb)(1_{21})&=h\cdot\tikzcrossrb(1_{21})-\tikzcrossrb(h\cdot 1_{21})=h\cdot((x-y)1_{21})-\tikzcrossrb(1_{21}),\\
        &=(2x-2y)1_{21}-(x-y)1_{21}=(x-y)1_{21}\\
        &\hspace*{275pt}\implies h\cdot\tikzcrossrb\:=\:\tikzcrossrb\:;\\
        (f\cdot\tikzcrossrb)(1_{21})&=f\cdot\tikzcrossrb(1_{21})-\tikzcrossrb(f\cdot 1_{21})=f\cdot ((x-y)1_{21})-0=0,\\
        &\hspace*{275pt}\implies f\cdot \tikzcrossrb\:=0.
    \end{align*}

    Finally let us do the distant crossing case. Let black and blue correspond to distant colors; WLOG let black be $\alpha_1$ and blue be $\alpha_3$. For $\alpha=\alpha_1+\alpha_3$ the sheaf $\CP$ is $\CO(0,0)1_{31}\oplus \CO(0,0)1_{13}$, and one can compute easily that
    \begin{align*}
        e\cdot 1_{13}&=h\cdot 1_{13}=f\cdot 1_{13}=0,\\
        e\cdot 1_{31}&=h\cdot 1_{31}=f\cdot 1_{31}=0.
    \end{align*}
    It follows easily from this and imitating the above computations that
    \[e\cdot \tikzcrossdistant\: = h\cdot \tikzcrossdistant\: = f\cdot \tikzcrossdistant\: =0 .\] 

    This is precisely the action of Elias-Qi.

    That the cellular ideals are preserved follows from the $\sl_2$-stability of $\CQ_{\ge\pi}$. 
\end{proof}


\subsection{A Witt action}
In fact, this can be taken further. Recall that the Witt algebra $\Wlie$ is the Lie algebra generated by $L_n$ for $n\in\BZ$, with Lie bracket structure given by
\[[L_m,L_n]=(m-n)L_{m+n}.\] 
It is worth noting that $L_{-1},L_0,L_1$ generate a copy of $\sl_2$ inside the Witt, identified via 
\[L_{-1}=-f,\qquad L_0=-\frac{1}{2}h,\qquad L_1=e.\] 
In \cite{lauda2025action} it was shown that the ``positive half'' of this algebra, namely $\Wlie_{\ge -1}$, acts upon the KLR algebras, and in particular upon $\NH_n$. We will also give a Witt algebra action on KLR. For type $A_1$, i.e. for nil-Hecke, this action recovers the action of \cite{lauda2025action}. However for multi-colors the action appears to be different; indeed, the action of \cite{lauda2025action} for 2-color KLR does not even agree with that of \cite{EQ23}. 

In for example \cite[Equation 4.10]{schlichenmaier2016krichever}, there is an action $\Wlie\actson \CO_{\BP^1}(k)(U\cap U^-)$, given by 
\[L_n\cdot f=-x^{n+1}\pd_x f+\frac{k}{2}(n+1)x^nf.\]
If one is willing to restrict to just the half Witt, $\Wlie_{\ge -1}$, then this action exists on $\CO_{\BP^1}(k)(U)$. 


\begin{proposition}\label{prop:halfwitt preserves Qk}
    There is a natural action $\Wlie_{\ge-1}\actson \CP(U)$ induced from the coherent geometry; under this action, the subspaces $\CQ_{\ge\pi}(U)$ are $\Wlie_{\ge-1}$-submodules.

    Similarly there is an action of the whole Witt algebra $\Wlie\actson \CP(U\cap U^-)$ induced from the coherent geometry; under this action, the subspaces $\CQ_{\ge\pi}(U\cap U^-)$ are similarly $\CW$-submodules
\end{proposition}
\begin{proof}
    For the operators $L_{-1},L_0,L_1$ this is already known due to the $\sl_2$-action, so it suffices to check $L_n$ for $n\ge 2$. The argument is identical however. 

    Let
    \[d_{a,i}(\beta)=n_a-i-r_{a,i}(\beta)\] 
    be the $i$-th twist of the $a$-th color in the $\beta$-th copy of $\wh\CP$. Then the action $\Wlie_{\ge-1}\actson \CP(U)$ is given by
    \[L_n\lmto \sum_\beta \pr*{-\sum_{a,i} x_{a,i}^{n+1}\pd_{x_{a,i}}+\frac{n+1}{2}\sum_{a,i}d_{a,i}(\beta)x_{a,i}^n}1_\beta.\] 
    Note that the action of this on polynomials preserves all ideals of form $\wan{x_{a,i}-x_{b,j}}$. Indeed, the first term acts by
    \[\pr*{-\sum_{c,k}x_{c,k}^{n+1}\pd_{x_{c,k}}}(x_{a,i}-x_{b,j})=-x_{a,i}^{n+1}+x_{b,j}^{n+1}=-(x_{a,i}-x_{b,j}) \h_n(x_{a,i},x_{b,j}),\]
    while the second term acts simply by multiplication. As $\CQ(U)$ and $\CI_{<\pi}\CP$ are built from intersections or products of ideals generated by such diagonal terms, it follows that $L_n$ preserves $\CQ_{\ge\pi}(U)$.

    For the whole Witt action, the argument is identical.
\end{proof}

\begin{theorem}\label{thm:halfwittacts}
    There is a natural action $\Wlie_{\ge -1}\actson \CR_\alpha$ induced from the coherent geometry, respecting the cellular filtration. It acts as follows:
    \[L_{-1}=-f,\qquad L_0=-\frac{1}{2}h,\qquad L_1=e.\] 
    For $n\ge -1$, the action is locally given by
    \begin{align*}
        L_n\cdot \tikzdot\:&=-\:\begin{diagram}
            \draw(0,0)--(0,1.5);
            \fill (0,0.75) circle (6pt);
            \node at (1.25,1) {\tiny$n+1$};
        \end{diagram},\\
        L_n\cdot \tikzcrossing\: &= \frac{n+1}{2}\pr*{\:\begin{diagram}
            \draw(0,0)--(1.5,1.5);
            \draw(0,1.5)--(1.5,0);
            \fill (1.15,0.35) circle(6pt);
            \node at (1.7,0.6) {\tiny $n$};
        \end{diagram}\!-\:\begin{diagram}
                \draw(0,0)--(1.5,1.5);
                \draw(0,1.5)--(1.5,0);
                \fill (1.15,1.15) circle(6pt);
                \node at (1.7,0.9) {\tiny $n$};
            \end{diagram}\!\!}+\sum_{i+j=n}\begin{diagram}
            \draw(0,0)--(1.5,1.5);
            \draw(0,1.5)--(1.5,0);
            \fill[black] (0.35,1.15) circle(6pt);
            \fill[black] (1.15,1.15) circle(6pt);
            \node at (1.7,0.9) {\tiny $j$};
            \node at (-0.2,0.9) {\tiny $i$};
        \end{diagram},\\
        L_n\cdot\tikzcrossbr\:&=-\frac{n+1}{2}\begin{diagram}
            \draw[black](0,0)--(1.5,1.5);
            \draw[red](0,1.5)--(1.5,0);
            \fill[red] (0.35,1.15) circle(6pt);
            \node at (-0.2,0.9) {\tiny $n$};
        \end{diagram}\:,\\
        L_n\cdot \tikzcrossrb\:&=\frac{n+1}{2}\begin{diagram}
            \draw[red](0,0)--(1.5,1.5);
            \draw[black](0,1.5)--(1.5,0);
            \fill[red] (1.15,1.15) circle(6pt);
            \node at (1.7,0.9) {\tiny $n$};
        \end{diagram}\:-\sum_{i+j=n} \begin{diagram}
            \draw[red](0,0)--(1.5,1.5);
            \draw[black](0,1.5)--(1.5,0);
            \fill[black] (0.35,1.15) circle(6pt);
            \fill[red] (1.15,1.15) circle(6pt);
            \node at (1.7,0.9) {\tiny $j$};
            \node at (-0.2,0.9) {\tiny $i$};
        \end{diagram},\\
        L_n\cdot \tikzcrossdistant \:&=0.
    \end{align*}
\end{theorem}
\begin{proof}
    Let $\CE,\CE'$ be coherent KLR data associated to $\alpha,\alpha'$. Let $\CE''$ be the image of the monoidal product, so that $\alpha''=\alpha+\alpha'$; in what follows, $\sq,\sq',\sq''$ will refer to data attached to those three sheaves, respectively. Let
    \[l_a=n_a'-n_{a-1}'\]
    be the twists of the sheaf $\CL_{a,a'}$ from Section \ref{sec:geometric monoidal}, and let $\rho=\rho_\alpha$ denote the action of $L_n$ on $\CP=\CP_\alpha$ given in the proof of Proposition \ref{prop:halfwitt preserves Qk}, and let $d_{a,i}(\beta)=n_a-i-r_{a,i}(\beta)$ in that same proof. Note that $d_{a,i}''(\beta\beta')=d_{a,i}(\beta)+l_a$ for $i\in[1,n_a]$, and $d_{a,n_a+j}''(\beta\beta')=d_{a,j}'(\beta')$ for $j\in[1,n_a']$. Under the natural identification of $\Pol_{\beta\beta'}=\Pol_\beta\otimes\Pol_{\beta'}$, we then have
    \[\rho''(L_n)_{\beta\beta'}=\rho(L_n)_\beta\otimes 1+1\otimes\rho'(L_n)_{\beta'}+\frac{n+1}{2}\sum_{a,i}l_a x_{a,i}^n;\]
    as the last term is multiplication by a polynomial in the center of KLR, it commutes with multiplication by KLR diagrams and disappears in the Leibniz rule bracket defining the action of $L_n$ on KLR diagrams:
    \[L_n\cdot (x\otimes x')=(L_n\cdot x)\otimes x'+x\otimes (L_n\cdot x')+\Big[\times \frac{n+1}{2}\sum_{a,i}l_a x_{a,i}^n,x\otimes x'\Big]=(L_n\cdot x)\otimes x'+x\otimes (L_n\cdot x').\]
    Hence the action is well-behaved with respect to monoidal product, and it suffices to check on monoidal generators.


    That the action is well-defined and preserves the cellular ideals follows from Proposition \ref{prop:halfwitt preserves Qk}.The rest is again a computation. Let us first do the 1-color generators. Recall that $\wh\CP$ for 1 string is $\CO(0)$, and for 2 strings it is $\CO(1,0)\oplus \CO(0,0)$. Let us first check the action of the (half-)Witt on the basis vectors of the polynomial representation:
    \begin{align*}
        L_n\cdot 1_1&=0\\
        &\hspace*{275pt}\implies L_n\cdot 1_1=0,\\ 
        L_n\cdot 1_{11}&=(L_n\cdot_{\CO(1)} 1)\boxtimes 1+1\boxtimes (L_n\cdot_{\CO(0)} 1)=\frac{n+1}{2}x^n\boxtimes1+0\\
        &\hspace*{275pt}\implies L_n\cdot 1_{11}=\frac{n+1}{2}x_1^n 1_{11},\\ 
        L_n\cdot x_11_{11}&=(L_n\cdot_{\CO(1)} x_1)\boxtimes 1+x_1\boxtimes (L_n\cdot_{\CO(0)} 1)=(-x^{n+1}+\frac{n+1}{2}x^{n+1})\boxtimes 1+0\\
        &\hspace*{275pt}\implies L_n\cdot x_11_{11}=\frac{n-1}{2}x_1^{n+1}1_{11}.
    \end{align*}
    Now to check the action on diagrams we use the Leibniz rule for morphisms, just as before. Let $\h_n(x,y)=\sum_{i+j=n}x^iy^j$ denote the complete homogeneous symmetric polynomial. 
    \begin{align*}
        (L_n\cdot \tikzdot)(1_1)&=L_n\cdot \tikzdot(1_1)-\tikzdot(L_n\cdot 1_1)=L_n\cdot x1_1=-x^{n+1}1_1\\
        &\hspace*{250pt}\implies L_n\cdot\tikzdot\:=-\:\begin{diagram}
            \draw(0,0)--(0,1.5);
            \fill (0,0.75) circle (6pt);
            \node at (1.25,1) {\tiny$n+1$};
        \end{diagram},\\
        (L_n\cdot \tikzcrossing)(1_{11})&=L_n\cdot \tikzcrossing(1_{11})-\tikzcrossing(L_n\cdot 1_{11})\\
        &=0-\tikzcrossing\bigpr{\frac{n+1}{2}x_1^n1_{11}}=-\frac{n+1}{2}\h_{n-1}(x_1,x_2)1_{11},\\
        (L_n\cdot \tikzcrossing)(x_11_{11})&=L_n\cdot \tikzcrossing(x_11_{11})-\tikzcrossing(L_n\cdot x_11_{11})\\
        &=L_n\cdot 1_{11}-\tikzcrossing\bigpr{\frac{n-1}{2}x_1^{n+1}1_{11}}=\pr*{\frac{n+1}{2}x_1^n-\frac{n-1}{2}\h_n(x_1,x_2)}1_{11}\\
        &\hspace*{250pt}\implies L_n\cdot\tikzcrossing\:=\frac{n+1}{2}\pr*{\:\begin{diagram}
            \draw(0,0)--(1.5,1.5);
            \draw(0,1.5)--(1.5,0);
            \fill (1.15,0.35) circle(6pt);
            \node at (1.7,0.6) {\tiny $n$};
        \end{diagram}\!-\:\begin{diagram}
                \draw(0,0)--(1.5,1.5);
                \draw(0,1.5)--(1.5,0);
                \fill (1.15,1.15) circle(6pt);
                \node at (1.7,0.9) {\tiny $n$};
            \end{diagram}\!\!}\\
            &\hspace*{370pt}+\sum_{i+j=n}\begin{diagram}
            \draw(0,0)--(1.5,1.5);
            \draw(0,1.5)--(1.5,0);
            \fill[black] (0.35,1.15) circle(6pt);
            \fill[black] (1.15,1.15) circle(6pt);
            \node at (1.7,0.9) {\tiny $j$};
            \node at (-0.2,0.9) {\tiny $i$};
        \end{diagram},
    \end{align*}
    where the last implication must be checked. It suffices to check that the claimed action $L_n\cdot\tikzcrossing$ acts as expected on the vectors $1_{11}$ and $x_11_{11}$. As $\psi_1(1_{11})=0$, it is clear that only the first term survives the action on $1_{11}$ and
    \[\pr*{\frac{n+1}{2}\pr*{\:\begin{diagram}
            \draw(0,0)--(1.5,1.5);
            \draw(0,1.5)--(1.5,0);
            \fill (1.15,0.35) circle(6pt);
            \node at (1.7,0.6) {\tiny $n$};
        \end{diagram}\!-\:\begin{diagram}
                \draw(0,0)--(1.5,1.5);
                \draw(0,1.5)--(1.5,0);
                \fill (1.15,1.15) circle(6pt);
                \node at (1.7,0.9) {\tiny $n$};
            \end{diagram}\!\!}+\sum_{i+j=n}\begin{diagram}
            \draw(0,0)--(1.5,1.5);
            \draw(0,1.5)--(1.5,0);
            \fill[black] (0.35,1.15) circle(6pt);
            \fill[black] (1.15,1.15) circle(6pt);
            \node at (1.7,0.9) {\tiny $j$};
            \node at (-0.2,0.9) {\tiny $i$};
        \end{diagram}}\cdot 1_{11}=-\frac{n+1}{2}\h_{n-1}(x_1,x_2)1_{11}.\]
    For the action on $x_11_{11}$, one may check that
    \begin{align*}
        \frac{n+1}{2}\:\begin{diagram}
            \draw(0,0)--(1.5,1.5);
            \draw(0,1.5)--(1.5,0);
            \fill (1.15,0.35) circle(6pt);
            \node at (1.7,0.6) {\tiny $n$};
        \end{diagram}\!\cdot x_11_{11}&=\frac{n+1}{2}x_1x_2\psi_1x_2^{n-1}\cdot 1_{11}=-\frac{n+1}{2}(x_1x_2)\h_{n-2}1_{11},\\
        -\frac{n+1}{2}\:\begin{diagram}
                \draw(0,0)--(1.5,1.5);
                \draw(0,1.5)--(1.5,0);
                \fill (1.15,1.15) circle(6pt);
                \node at (1.7,0.9) {\tiny $n$};
            \end{diagram}\!\cdot x_11_{11}&=-\frac{n+1}{2}x_2^n,\\
        \pr*{\sum_{i+j=n}\begin{diagram}
            \draw(0,0)--(1.5,1.5);
            \draw(0,1.5)--(1.5,0);
            \fill[black] (0.35,1.15) circle(6pt);
            \fill[black] (1.15,1.15) circle(6pt);
            \node at (1.7,0.9) {\tiny $j$};
            \node at (-0.2,0.9) {\tiny $i$};
        \end{diagram}}\cdot x_11_{11}&=\h_n(x_1,x_2),
    \end{align*}
    so that summing up we find
    \[\pr*{\frac{n+1}{2}\pr*{\:\begin{diagram}
            \draw(0,0)--(1.5,1.5);
            \draw(0,1.5)--(1.5,0);
            \fill (1.15,0.35) circle(6pt);
            \node at (1.7,0.6) {\tiny $n$};
        \end{diagram}\!-\:\begin{diagram}
                \draw(0,0)--(1.5,1.5);
                \draw(0,1.5)--(1.5,0);
                \fill (1.15,1.15) circle(6pt);
                \node at (1.7,0.9) {\tiny $n$};
            \end{diagram}\!\!}+\sum_{i+j=n}\begin{diagram}
            \draw(0,0)--(1.5,1.5);
            \draw(0,1.5)--(1.5,0);
            \fill[black] (0.35,1.15) circle(6pt);
            \fill[black] (1.15,1.15) circle(6pt);
            \node at (1.7,0.9) {\tiny $j$};
            \node at (-0.2,0.9) {\tiny $i$};
        \end{diagram}}\cdot x_11_{11}=\pr*{-\frac{n+1}{2}(x_1x_2)\h_{n-2}(x_1,x_2)-\frac{n+1}{2}x_2^n+\h_n(x_1,x_2)}1_{11}.\]
    Now it is an easy identity of symmetric functions to see that $-\frac{n+1}{2}(x_1x_2)\h_{n-2}(x_1,x_2)-\frac{n+1}{2}x_2^n+\h_n(x_1,x_2)=\frac{n+1}{2}x_1^n-\frac{n-1}{2}\h_n(x_1,x_2)$. Hence the action on 1-color is as claimed.

    Let us second check the nearby 2-color crossings. Recall that in this case $\CP=\CO(0,-1)_{21}\oplus\CO(0,0)_{12}$. Let us check the action of $L_n$ on some vectors:
    \begin{align*}
        L_n\cdot 1_{12}&=(L_n\cdot_{\CO(0)} 1)\boxtimes 1+1\boxtimes (L_n\cdot_{\CO(0)} 1)=0\\
        &\hspace*{275pt}\implies L_n\cdot 1_{12}=0,\\ 
        L_n\cdot 1_{21}&=(L_n\cdot_{\CO(0)} 1)\boxtimes 1+1\boxtimes (L_n\cdot_{\CO(-1)} 1)=0+1\boxtimes \pr*{-\frac{n+1}{2}y^n}\\
        &\hspace*{275pt}\implies L_n\cdot 1_{21}=-\frac{n+1}{2}y^{n}1_{21}.
    \end{align*}
    Now let us check the action on diagrams. 
    \begin{align*}
        (L_n\cdot \tikzcrossbr)(1_{12})&=L_n\cdot \tikzcrossbr(1_{12})-\tikzcrossbr(L_n\cdot 1_{12})=L_n\cdot 1_{21}=-\frac{n+1}{2}y^n1_{21}\\
        &\hspace*{275pt}\implies L_n\cdot \tikzcrossbr=-\frac{n+1}{2}\begin{diagram}
            \draw[black](0,0)--(1.5,1.5);
            \draw[red](0,1.5)--(1.5,0);
            \fill[red] (0.35,1.15) circle(6pt);
            \node at (-0.2,0.9) {\tiny $n$};
        \end{diagram}\:,\\ 
        (L_n\cdot \tikzcrossrb)(1_{21})&=L_n\cdot \tikzcrossrb(1_{21})-\tikzcrossrb(L_n\cdot 1_{21})=L_n\cdot (x-y)1_{12}-\tikzcrossbr\pr*{-\frac{n+1}{2}y^n1_{21}}\\
        &=(-x^{n+1}+y^{n+1})1_{12}+\frac{n+1}{2}y^n(x-y)1_{12}=\bigpr{-\h_n(x,y)+\frac{n+1}{2}y^n}(x-y)1_{12}\\
        &\hspace*{275pt}\implies L_n\cdot \tikzcrossrb=\frac{n+1}{2}\begin{diagram}
            \draw[red](0,0)--(1.5,1.5);
            \draw[black](0,1.5)--(1.5,0);
            \fill[red] (1.15,1.15) circle(6pt);
            \node at (1.7,0.9) {\tiny $n$};
        \end{diagram}\:\\
        &\hspace*{360pt}-\sum_{i+j=n} \begin{diagram}
            \draw[red](0,0)--(1.5,1.5);
            \draw[black](0,1.5)--(1.5,0);
            \fill[black] (0.35,1.15) circle(6pt);
            \fill[red] (1.15,1.15) circle(6pt);
            \node at (1.7,0.9) {\tiny $j$};
            \node at (-0.2,0.9) {\tiny $i$};
        \end{diagram}.\\ 
    \end{align*}
    
    Let us finally check the distant 2-color crossings. Recall that in this case $\CP=\CO(0,0)_{31}\oplus\CO(0,0)_{13}$. It is easy to compute that
    \[L_n\cdot 1_{13}=L_n\cdot 1_{31}=0,\]
    and so
    \begin{align*}
        (L_n\cdot \tikzcrossdistant)(1_{13})&=L_n\cdot \tikzcrossdistant(1_{13})-\tikzcrossdistant(L_n\cdot 1_{13})=L_n\cdot 1_{31}-0=0\\
        &\hspace*{250pt}\implies L_n\cdot \tikzcrossdistant=0.
    \end{align*}

    This concludes. 
\end{proof}
This recovers the action of \cite{lauda2025action} for 1-color KLR, with Witt sequence parameter $\mu_n=\frac{n+1}{4}$; however, note well that it does \textit{not} recover their action on multi-color KLR! Indeed, whereas we have
\[L_1\cdot\tikzcrossbr=-\:\begin{diagram}
    \draw(0,0)--(1.5,1.5);
    \draw[red](0,1.5)--(1.5,0);
    \fill[red] (0.35,1.15)circle(6pt);
\end{diagram}\:,\]
the action of \cite{lauda2025action} has
\[L_1\cdot_\te{GL}\tikzcrossbr=\frac{1}{2}\:\begin{diagram}
    \draw(0,0)--(1.5,1.5);
    \draw[red](0,1.5)--(1.5,0);
    \fill[black] (1.15,1.15)circle(6pt);
\end{diagram}\:-\frac{3}{2}\:\begin{diagram}
    \draw(0,0)--(1.5,1.5);
    \draw[red](0,1.5)--(1.5,0);
    \fill[red] (0.35,1.15)circle(6pt);
\end{diagram}\:.\]


\subsection{A full Witt action on the Laurent KLR algebra}\label{sec:fullWitt}
In order to obtain an action of the whole Witt algebra, we must pass to a smaller open set in $X$. Let $U^\times= U\cap U^-\subset X$, and let
\[\CR_\alpha^\times\coloneqq \CE(U\cap U^-).\] 
This is the ``Laurent KLR algebra''. 

The 1-color case of this was essentially worked out in \cite{KMZ}, but perhaps it is briefly worth saying something again in the proof of the following.

Recall that there were $2^N$ affine charts $X^\eps$ on the space $X^\circ$, depending on a choice of sign for each color, $\eps=(\eps_1,\cdotsc,\eps_N)$. By the identification of $\CE$ inside $\shend\CP$, we can give diagrammatics on each chart; the choice of sign will be reflected in the way we draw strings of each color, so that a straight regular string of color $i$ corresponds to $\eps_i=+$ (i.e. nonnegative powers of dots are allowed on strings of color $i$, as they have no poles on $X^\eps$), and a squiggly string of color $i$ corresponds to $\eps_i=-$ (i.e. nonpositive powers of dots are allowed on strings of color $i$, as they have no poles on $X^\eps$). 
\begin{theorem}\label{thm:laurent KLR}
    There is a Laurent KLR algebra $\CR^\times_\alpha=\CE(U^\times)$, which is cellular with cell ideals $\CR_{\ge\pi}^\times=\CE_{\ge\pi}(U^\times)$. It is diagrammatically generated monoidally by the usual KLR generators in addition to a ``negative dot'', namely 
    \begin{center}
\begin{tabular}{ c c c c c c }
 & $\begin{diagram}
    \draw (0,0)--(0,2);
    \node at (0,-0.5) {\tiny $i$};
    \draw (0,1) circle (6pt);
\end{diagram}$ & $\begin{diagram}
    \draw (0,0)--(0,2);
    \fill (0,1) circle (6pt);
    \node at (0,-0.5) {\tiny $i$};
\end{diagram}$ & 
$\begin{diagram}
    \draw (0,0)--(2,2);
    \draw (2,0)--(0,2);
    \node at (0,-0.5) {\tiny $i$};
    \node at (2,-0.5) {\tiny $i$};
\end{diagram}$ & 
$\begin{diagram}
    \draw (0,0)--(2,2);
    \draw (2,0)--(0,2);
    \node at (0,-0.5) {\tiny $i$};
    \node at (2,-0.5) {\tiny $i\pm 1$};
\end{diagram}$ & 
$\begin{diagram}
    \draw (0,0)--(2,2);
    \draw (2,0)--(0,2);
    \node at (0,-0.5) {\tiny $i$};
    \node at (2,-0.5) {\tiny $j$};
\end{diagram}$
\\[1em]
    $\mathrm{degree}$ & $-2$ & $2$ & $-2$ & \hspace*{-6.5pt}$1$ & $0$ 
\end{tabular}
\end{center}
where $|j-i|>1$. The local relations, in addition to the usual KLR relations, are
\begin{align*}
    \:\begin{raisediagram}[-5pt]
    \draw (0,0)--(0,2);
    \fill (0,0.6) circle (6pt);
    \draw (0,1.4) circle (6pt);
    \node at (0,-0.5) {\tiny $i$};
\end{raisediagram}\:&=\:\begin{raisediagram}[-5pt]
    \draw (0,0)--(0,2);
    \node at (0,-0.5) {\tiny $i$};
\end{raisediagram}\:=\:\begin{raisediagram}[-5pt]
    \draw (0,0)--(0,2);
    \draw (0,0.6) circle (6pt);
    \fill (0,1.4) circle (6pt);
    \node at (0,-0.5) {\tiny $i$};
\end{raisediagram}\:,\\
    \:\begin{raisediagram}[-5pt]
        \draw(0,0)--(2,2);
        \draw(0,2)--(2,0);
        \draw (0.5,1.5) circle (6pt);
        \node at (0,-0.5) {\tiny $i$};
        \node at (2,-0.5) {\tiny $i$};
    \end{raisediagram}\:-\:\begin{raisediagram}[-5pt]
        \draw(0,0)--(2,2);
        \draw(0,2)--(2,0);
        \draw (1.5,0.5) circle (6pt);
        \node at (0,-0.5) {\tiny $i$};
        \node at (2,-0.5) {\tiny $i$};
    \end{raisediagram}\:&=-\:\begin{raisediagram}[-5pt]
        \draw(0,0)--(0,2);
        \draw(2,0)--(2,2);
        \draw (0,1) circle (6pt);
        \draw (2,1) circle (6pt);
        \node at (0,-0.5) {\tiny $i$};
        \node at (2,-0.5) {\tiny $i$};
    \end{raisediagram}\:=\:\begin{raisediagram}[-5pt]
        \draw(0,0)--(2,2);
        \draw(0,2)--(2,0);
        \draw (0.5,0.5) circle (6pt);
        \node at (0,-0.5) {\tiny $i$};
        \node at (2,-0.5) {\tiny $i$};
    \end{raisediagram}\:-\:\begin{raisediagram}[-5pt]
        \draw(0,0)--(2,2);
        \draw(0,2)--(2,0);
        \draw (1.5,1.5) circle (6pt);
        \node at (0,-0.5) {\tiny $i$};
        \node at (2,-0.5) {\tiny $i$};
    \end{raisediagram}\:,\\
    \:\begin{raisediagram}[-5pt]
        \draw(0,0)--(2,2);
        \draw(0,2)--(2,0);
        \draw (0.5,1.5) circle (6pt);
        \node at (0,-0.5) {\tiny $i$};
        \node at (2,-0.5) {\tiny $j$};
    \end{raisediagram}\:&=\:\begin{raisediagram}[-5pt]
        \draw(0,0)--(2,2);
        \draw(0,2)--(2,0);
        \draw (1.5,0.5) circle (6pt);
        \node at (0,-0.5) {\tiny $i$};
        \node at (2,-0.5) {\tiny $j$};
    \end{raisediagram}\: \qquad (i\neq j),\\
    \:\begin{raisediagram}[-5pt]
        \draw(0,0)--(2,2);
        \draw(0,2)--(2,0);
        \draw (0.5,0.5) circle (6pt);
        \node at (0,-0.5) {\tiny $i$};
        \node at (2,-0.5) {\tiny $j$};
    \end{raisediagram}\:&=\:\begin{raisediagram}[-5pt]
        \draw(0,0)--(2,2);
        \draw(0,2)--(2,0);
        \draw (1.5,1.5) circle (6pt);
        \node at (0,-0.5) {\tiny $i$};
        \node at (2,-0.5) {\tiny $j$};
    \end{raisediagram}\:\qquad (i\neq j).
\end{align*}
The ``negative crossings'' have degrees 
\begin{center}
\begin{tabular}{ c c c c c }
 & $\begin{raisediagram}[-5pt]
        \draw[squigs](0,0)--(2,2);
        \draw[squigs](0,2)--(2,0);
        \node at (0,-0.5) {\tiny $i$};
        \node at (2,-0.5) {\tiny $i$};
    \end{raisediagram}$ & $\begin{raisediagram}[-5pt]
        \draw[dash pattern=on 3pt off 1.75pt](0,0)--(2,2);
        \draw[squigs](0,2)--(2,0);
        \node at (0,-0.5) {\tiny $i$};
        \node at (2,-0.5) {\tiny $i\pm 1$};
    \end{raisediagram}$ & 
$\begin{raisediagram}[-5pt]
        \draw[dash pattern=on 3pt off 1.75pt](0,0)--(2,2);
        \draw(0,2)--(2,0);
        \node at (0,-0.5) {\tiny $i$};
        \node at (2,-0.5) {\tiny $i\pm 1$};
    \end{raisediagram}$ & 
$\begin{raisediagram}[-5pt]
        \draw[dash pattern=on 3pt off 1.75pt](0,0)--(2,2);
        \draw[dash pattern=on 3pt off 1.75pt](0,2)--(2,0);
        \node at (0,-0.5) {\tiny $i$};
        \node at (2,-0.5) {\tiny $j $};
    \end{raisediagram}$
\\[1em]
    $\mathrm{degree}$ & $2$ & \hspace*{-6.5pt}$-1$ & \hspace*{-6.5pt}$1$ & $0$ 
\end{tabular}
\end{center}
where $|j-i|>1$ and the dashed line indicates either solid straight lines or squiggly lines, and the transition functions/relations are
\begin{align*}
    \:\begin{raisediagram}[-5pt]
        \draw[squigs](0,0)--(2,2);
        \draw[squigs](0,2)--(2,0);
        \node at (0,-0.5) {\tiny $i$};
        \node at (2,-0.5) {\tiny $i$};
    \end{raisediagram}\:&=-\:\begin{raisediagram}[-5pt]
        \draw(0,0)--(2,2);
        \draw(0,2)--(2,0);
        \fill (1.5,0.5) circle (6pt);
        \fill (0.5,1.5) circle (6pt);
        \node at (0,-0.5) {\tiny $i$};
        \node at (2,-0.5) {\tiny $i$};
    \end{raisediagram}\:,\\
    \:\begin{raisediagram}[-5pt]
        \draw(0,0)--(2,2);
        \draw[squigs](0,2)--(2,0);
        \node at (0,-0.5) {\tiny $i$};
        \node at (2,-0.5) {\tiny $i\pm 1$};
    \end{raisediagram}\:&=\:\begin{raisediagram}[-5pt]
        \draw(0,0)--(2,2);
        \draw(0,2)--(2,0);
        \draw (0.5,1.5) circle (6pt);
        \node at (0,-0.5) {\tiny $i$};
        \node at (2,-0.5) {\tiny $i\pm 1$};
    \end{raisediagram}\:,\\
    \:\begin{raisediagram}[-5pt]
        \draw[squigs](0,0)--(2,2);
        \draw(0,2)--(2,0);
        \node at (0,-0.5) {\tiny $i$};
        \node at (2,-0.5) {\tiny $i\pm 1$};
    \end{raisediagram}\:&=\:\begin{raisediagram}[-5pt]
        \draw(0,0)--(2,2);
        \draw(0,2)--(2,0);
        \node at (0,-0.5) {\tiny $i$};
        \node at (2,-0.5) {\tiny $i\pm1 $};
    \end{raisediagram}\:,\\
    \:\begin{raisediagram}[-5pt]
        \draw[squigs](0,0)--(2,2);
        \draw[squigs](0,2)--(2,0);
        \node at (0,-0.5) {\tiny $i$};
        \node at (2,-0.5) {\tiny $i\pm1 $};
    \end{raisediagram}\:&=\:\begin{raisediagram}[-5pt]
        \draw(0,0)--(2,2);
        \draw(0,2)--(2,0);
        \draw (0.5,1.5) circle (6pt);
        \node at (0,-0.5) {\tiny $i$};
        \node at (2,-0.5) {\tiny $i\pm1 $};
    \end{raisediagram}\:,\\
    \:\begin{raisediagram}[-5pt]
        \draw[squigs](0,0)--(2,2);
        \draw[squigs](0,2)--(2,0);
        \node at (0,-0.5) {\tiny $i$};
        \node at (2,-0.5) {\tiny $j $};
    \end{raisediagram}\:=\:\begin{raisediagram}[-5pt]
        \draw[squigs](0,0)--(2,2);
        \draw(0,2)--(2,0);
        \node at (0,-0.5) {\tiny $i$};
        \node at (2,-0.5) {\tiny $j $};
    \end{raisediagram}\:&=\:\begin{raisediagram}[-5pt]
        \draw(0,0)--(2,2);
        \draw[squigs](0,2)--(2,0);
        \node at (0,-0.5) {\tiny $i$};
        \node at (2,-0.5) {\tiny $j $};
    \end{raisediagram}\:=\:\begin{raisediagram}[-5pt]
        \draw(0,0)--(2,2);
        \draw(0,2)--(2,0);
        \node at (0,-0.5) {\tiny $i$};
        \node at (2,-0.5) {\tiny $j $};
    \end{raisediagram}\:\qquad(|j-i|>1).
\end{align*}
\end{theorem}
\begin{proof}
    Let us check the 1-color diagrammatics. On the cover, $\wh\CP=\CO(1,0)\sqcup\CO(0,0)$. As in Section \ref{sect:2color}, we identify 
    \[\CO(a,b)(\wh U)=\BC[X_0,X_1,Y_0,Y_1]\]
    and embed $\CO(a,b)(\wh U)\linj \CO_\eta$ via $s\lmto X_0^{-a}Y_0^{-b}s$. Recall $\wh U$ is the chart where $X_0=Y_0=1$ and $\wh U^-$ is the chart where $X_1=Y_1=1$, and $x_1=X_1/X_0$ and $x_2=Y_1/Y_0$. Hence
    \begin{center}
        \begin{tikzcd}
            \wh\CP(\wh U)=\BC[x_1,x_2] \arrow[hookrightarrow]{r}&\BC(x_1,y_1)\arrow[hookleftarrow]{r}{\cdot x_1} &  \wh\CP(\wh U^-)=\BC[1/x_1,1/x_2]
        \end{tikzcd}
    \end{center}
    and the transition function is 
    \begin{align*}
        \wh\CP(\wh U)&\lto \wh\CP(\wh U^-)\\
        x_1&\lmto 1\\
        1&\lmto x_1^{-1}.
    \end{align*}
    Diagrammatically this means
    \begin{align*}
        \:
\begin{diagram}
  \draw[black](0,-0.5)--(0,1);
  \draw[black](2,-0.5)--(2,1);
  \draw[orange](-0.5,-0.5)--(2.5,-0.5);
  \draw[white](0,1)--(0,2);
  \fill[black](0,0.5)circle(6pt);
\end{diagram}\:&=\:
\begin{diagram}
  \draw[black,squigs](0,-0.5)--(0,1);
  \draw[black,squigs](2,-0.5)--(2,1);
  \draw[orange](-0.5,-0.5)--(2.5,-0.5);
  \draw[white](0,1)--(0,2);
\end{diagram}\:,\\
        \:\begin{diagram}
  \draw[black](0,-0.5)--(0,1);
  \draw[black](2,-0.5)--(2,1);
  \draw[orange](-0.5,-0.5)--(2.5,-0.5);
  \draw[white](0,1)--(0,2);
\end{diagram}\:&=\:
\begin{diagram}
  \draw[black,squigs](0,-0.5)--(0,1);
  \draw[black,squigs](2,-0.5)--(2,1);
  \draw[orange](-0.5,-0.5)--(2.5,-0.5);
  \draw[white](0,1)--(0,2);
  \draw[black](0,0.5)circle(6pt);
\end{diagram}\:
    \end{align*}
    As 
    \[\CE=\tbt{\CO(0)}{\CO(0)}{\CO(0)}{\CO(0)},\] 
    we can then compute the action on the polynomial representation to be
    \[0=\:\begin{diagram}
  \draw[black,squigs](0,0)--(2,2);
  \draw[squigs](2,0)--(0,2);
  \draw[black,squigs](0,-1)--(0,0);
  \draw[squigs](2,-1)--(2,0);
  \draw[orange](-0.5,-1)--(2.5,-1);
\end{diagram}\:=\:\begin{diagram}
  \draw[black,squigs](0,0)--(2,2);
  \draw[squigs](2,0)--(0,2);
  \draw[black](0,-1)--(0,0);
  \draw(2,-1)--(2,0);
  \fill (0,-0.5) circle(6pt);
  \draw[orange](-0.5,-1)--(2.5,-1);
\end{diagram}\:,\qquad 
\:
\begin{diagram}
  \draw[black](0,-0.5)--(0,1);
  \draw[black](2,-0.5)--(2,1);
  \draw[orange](-0.5,-0.5)--(2.5,-0.5);
  \draw[white](0,1)--(0,2);
  \fill[black](0,0.5)circle(6pt);
\end{diagram}\: = \:
\begin{diagram}
  \draw[black,squigs](0,-0.5)--(0,1);
  \draw[black,squigs](2,-0.5)--(2,1);
  \draw[orange](-0.5,-0.5)--(2.5,-0.5);
  \draw[white](0,1)--(0,2);
\end{diagram}\: =\:\begin{diagram}
  \draw[black,squigs](0,0)--(2,2);
  \draw[squigs](2,0)--(0,2);
  \draw[black,squigs](0,-1)--(0,0);
  \draw[squigs](2,-1)--(2,0);
  \draw[orange](-0.5,-1)--(2.5,-1);
  \draw (0,-0.5) circle(6pt);
\end{diagram}\:=\:\begin{diagram}
  \draw[black,squigs](0,0)--(2,2);
  \draw[squigs](2,0)--(0,2);
  \draw[black](0,-1)--(0,0);
  \draw(2,-1)--(2,0);
  \draw[orange](-0.5,-1)--(2.5,-1);
\end{diagram}\:,
\]
    which forces $\:\begin{raisediagram}[-5pt]
        \draw[squigs](0,0)--(2,2);
        \draw[squigs](0,2)--(2,0);
        \node at (0,-0.5) {\tiny $i$};
        \node at (2,-0.5) {\tiny $i$};
    \end{raisediagram}\:=-\:\begin{raisediagram}[-5pt]
        \draw(0,0)--(2,2);
        \draw(0,2)--(2,0);
        \fill (1.5,0.5) circle (6pt);
        \fill (0.5,1.5) circle (6pt);
        \node at (0,-0.5) {\tiny $i$};
        \node at (2,-0.5) {\tiny $i$};
    \end{raisediagram}\:$, as advertised.

    The adjacent 2-color case diagrammatic transition functions were worked out in Section \ref{sect:2color}. 
    
    In the distant 2-color case, as $\CP=\CO(0,0)\oplus\CO(0,0)$, it is easy to check that the diagrams behave as advertised.

    Now that we have checked the transition functions, the relations hold automatically. For instance to obtain that $y_1^{-1}\psi_1 e^{11}-\psi_1 y_2^{-1}e^{11}=-y_1^{-1}y_2^{-1}e^{11}$ one can simply take the relation $y_1\psi_1 e^{11}-\psi_1 y_2e^{11}=e^{11}$ and multiply on the left by $y_1^{-1}$ and on the right by $y_2^{-1}$. 
\end{proof}
\begin{example}
    In the case $\alpha=2\alpha_1+\alpha_2$, it is possible to manually check that the following subalgebra of KLR is isomorphic to the semisimplification:
    \[
\wt\CR_{2\alpha_1+\alpha_2}=\thbth{\hackcenter{\begin{tikzpicture}[scale=0.375]
    \draw[red] (1,0)--(1,2);
    \draw (2,0)--(3,2);
    \draw (3,0)--(2,2);
    \fill (2.2,1.6) circle (5pt);
\end{tikzpicture}}}{-\ \hackcenter{\begin{tikzpicture}[scale=0.375]
    \draw[red] (1,0)--(1,2);
    \draw (2,0)--(3,2);
    \draw (3,0)--(2,2);
    \fill (2.2,1.6) circle (5pt);
    \fill (2.8,0.4) circle (5pt);
\end{tikzpicture}}}{\hackcenter{\begin{tikzpicture}[scale=0.375]
    \draw[red] (2,0)--(1,2);
    \draw (1,0)--(3,2);
    \draw (3,0)--(2,2);
    \fill (2.2,1.6) circle (5pt);
\end{tikzpicture}}}{\hackcenter{\begin{tikzpicture}[scale=0.375]
    \draw[red] (1,0)--(1,2);
    \draw (2,0)--(3,2);
    \draw (3,0)--(2,2);
\end{tikzpicture}}}{-\ \hackcenter{\begin{tikzpicture}[scale=0.375]
    \draw[red] (1,0)--(1,2);
    \draw (2,0)--(3,2);
    \draw (3,0)--(2,2);
    \fill (2.8,0.4) circle (5pt);
\end{tikzpicture}}}{\hackcenter{\begin{tikzpicture}[scale=0.375]
    \draw[red] (2,0)--(1,2);
    \draw (1,0)--(3,2);
    \draw (3,0)--(2,2);
\end{tikzpicture}}}{\hackcenter{\begin{tikzpicture}[scale=0.375]
    \draw[red] (1,0)--(2,2);
    \draw (2,0)--(3,2);
    \draw (3,0)--(1,2);
\end{tikzpicture}}}{-\ \hackcenter{\begin{tikzpicture}[scale=0.375]
    \draw[red] (1,0)--(2,2);
    \draw (2,0)--(3,2);
    \draw (3,0)--(1,2);
    \fill (2.7,0.3) circle (5pt);
\end{tikzpicture}}}{\hackcenter{\begin{tikzpicture}[scale=0.375]
    \draw[red] (2,0) to[in=-150,out=150] (2,2);
    \draw (1,0)--(3,2);
    \draw (3,0)--(1,2);
\end{tikzpicture}}}
\oplus
\thbth{-\ \hackcenter{\begin{tikzpicture}[scale=0.375]
    \draw[red] (2,0) to[in=-30,out=30] (2,2);
    \draw (1,0)--(3,2);
    \draw (3,0)--(1,2);
\end{tikzpicture}}}{\hackcenter{\begin{tikzpicture}[scale=0.375]
    \draw (1,2)--(2,0);
    \draw[red] (2,2)--(3,0);
    \draw (3,2)--(1,0);
\end{tikzpicture}}}
{-\ \hackcenter{\begin{tikzpicture}[scale=0.375]
    \draw (1,2)--(2,0);
    \draw[red] (2,2)--(3,0);
    \draw (3,2)--(1,0);
    \fill (1.8,0.4) circle (5pt);
\end{tikzpicture}}}
{-\ \hackcenter{\begin{tikzpicture}[scale=0.375]
    \draw (2,2)--(1,0);
    \draw (1,2)--(3,0);
    \draw[red] (3,2)--(2,0);
    \fill (1.3,2-0.3) circle (5pt);
\end{tikzpicture}}}{\hackcenter{\begin{tikzpicture}[scale=0.375]
    \draw[red] (4,0)--(4,2);
    \draw (2,0)--(3,2);
    \draw (3,0)--(2,2);
    \fill (2.2,1.6) circle (5pt);
\end{tikzpicture}}}
{-\ \hackcenter{\begin{tikzpicture}[scale=0.375]
    \draw[red] (4,0)--(4,2);
    \draw (2,0)--(3,2);
    \draw (3,0)--(2,2);
    \fill (2.2,1.6) circle (5pt);
    \fill (2.8,0.4) circle (5pt);
\end{tikzpicture}}}
{-\ \hackcenter{\begin{tikzpicture}[scale=0.375]
    \draw (2,2)--(1,0);
    \draw (1,2)--(3,0);
    \draw[red] (3,2)--(2,0);
\end{tikzpicture}}}
{\hackcenter{\begin{tikzpicture}[scale=0.375]
    \draw[red] (4,0)--(4,2);
    \draw (2,0)--(3,2);
    \draw (3,0)--(2,2);
\end{tikzpicture}}}
{-\ \hackcenter{\begin{tikzpicture}[scale=0.375]
    \draw[red] (4,0)--(4,2);
    \draw (2,0)--(3,2);
    \draw (3,0)--(2,2);
    \fill (2.8,0.4) circle (5pt);
\end{tikzpicture}}}=\wt\CR^{(1)(21)}\oplus\wt\CR^{(1)(1)(2)}.\]
    The diagrammatics of Theorem \ref{thm:laurent KLR} show that the elements of the above matrix algebras
    \begin{align*}
        &E_{1,1}^{(1)(21)},E_{1,2}^{(1)(21)},E_{2,1}^{(1)(21)},E_{2,2}^{(1)(21)},E_{3,1}^{(1)(21)},E_{3,2}^{(1)(21)},\\
        &E_{3,3}^{(1)(21)}+E_{1,1}^{(1)(1)(2)},\\
        &E_{2,1}^{(1)(1)(2)},E_{2,2}^{(1)(1)(2)},E_{2,3}^{(1)(1)(2)},E_{3,1}^{(1)(1)(2)},E_{3,2}^{(1)(1)(2)},E_{3,3}^{(1)(1)(2)}
    \end{align*}
    are precisely the globally defined sections of $\CR_{2\alpha_1+\alpha_2}$. For instance, ${\hackcenter{\begin{tikzpicture}[scale=0.375]
    \draw[red] (1,0)--(2,2);
    \draw (2,0)--(3,2);
    \draw (3,0)--(1,2);
\end{tikzpicture}}}$ is a global section because 
\begin{align*}
    \:\begin{diagram}
    \draw[red] (1,0)--(2,2);
    \draw (2,0)--(3,2);
    \draw (3,0)--(1,2);
\end{diagram}\:
&= 
\:\begin{diagram}
    \draw[red,squigs] (1,0)--(2,2);
    \draw (2,0)--(3,2);
    \draw (3,0)--(1,2);
\end{diagram}\:\\
&=
\:\begin{diagram}
    \draw[red] (1,0)--(2,2);
    \draw[squigs] (2,0)--(3,2);
    \draw[squigs] (3,0)--(1,2) node[pos=0.85, circle, fill=black, inner sep=1.5pt] {} node[pos=0.5, circle, draw=black, inner sep=1.5pt] {} node[pos=0.15, circle, draw=black, inner sep=1.5pt] {} ;
\end{diagram}\:=
\:\begin{diagram}
    \draw[red] (1,0)--(2,2);
    \draw[squigs] (2,0)--(3,2);
    \draw[squigs] (3,0)--(1,2) node[pos=0.15, circle, draw=black, inner sep=1.5pt] {} ;
\end{diagram}\:\\
&=
\:\begin{diagram}
    \draw[red,squigs] (1,0)--(2,2);
    \draw[squigs] (2,0)--(3,2);
    \draw[squigs] (3,0)--(1,2) node[pos=0.15, circle, draw=black, inner sep=1.5pt] {} ;
\end{diagram}\:;
\end{align*}
each of these four local descriptions on the four local affine charts have no poles. So the core is 13-dimensional. We invite the reader to check that 
\[{\hackcenter{\begin{tikzpicture}[scale=0.375]
    \draw[red] (2,0)--(1,2);
    \draw (1,0)--(3,2);
    \draw (3,0)--(2,2);
    \fill (2.2,1.6) circle (5pt);
\end{tikzpicture}}}
,
{\hackcenter{\begin{tikzpicture}[scale=0.375]
    \draw[red] (2,0)--(1,2);
    \draw (1,0)--(3,2);
    \draw (3,0)--(2,2);
\end{tikzpicture}}}
,
{\hackcenter{\begin{tikzpicture}[scale=0.375]
    \draw[red] (2,0) to[in=-30,out=30] (2,2);
    \draw (1,0)--(3,2);
    \draw (3,0)--(1,2);
\end{tikzpicture}}},
{\hackcenter{\begin{tikzpicture}[scale=0.375]
    \draw (1,2)--(2,0);
    \draw[red] (2,2)--(3,0);
    \draw (3,2)--(1,0);
\end{tikzpicture}}},
{-\ \hackcenter{\begin{tikzpicture}[scale=0.375]
    \draw (1,2)--(2,0);
    \draw[red] (2,2)--(3,0);
    \draw (3,2)--(1,0);
    \fill (1.8,0.4) circle (5pt);
\end{tikzpicture}}}\] 
each have poles on some affine chart. So for this example we can see that the $\sl_2$-core strictly injects into the semisimplification, though we do not prove this in general in this paper.

Note that the core is 13-dimensional and not semisimple. Its Grothendieck group is rank 3, in particular bigger than that of the KLR. We also do not investigate the decategorification of the core in this paper, though it receives a monoidal product automatically from Theorem \ref{thm:A_N monoidal}. 
\end{example}

Now that we have restricted to $U\cap U^-$, the entire Witt algebra $\Wlie$ is able to act on $\CO_{\BP^1}(U\cap U^-)$ and therefore on $\CR_\alpha^\times=\CE(U\cap U^-)$. For convenience, let us define
\[\h_n(x,y)\coloneqq \frac{x^{n+1}-y^{n+1}}{x-y}=\begin{cases}
    \beau\sum_{\substack{i+j=n\\i,j\ge 0}} x^i y^j & n\ge 0\\
    0 & n=-1\\
    -\beau\sum_{\substack{i+j=n\\i,j<0}} x^i y^j & n\le -2
\end{cases}\ ;\]
then it is not hard to see that that the exact same computations as in the proof of Theorem \ref{thm:halfwittacts} gives the following.
\begin{theorem}\label{thm:fullwittacts}
    There is a full Witt algebra action on the Laurent KLR algebra, $\Wlie\actson \CR_\alpha^\times$, respecting the cellular filtration $\CR_{\ge\pi}^\times$. The action is locally given by, for any $n\in\BZ$,
    \begin{align*}
        L_n\cdot \tikzdot\:&=-\:\begin{diagram}
            \draw(0,0)--(0,1.5);
            \fill (0,0.75) circle (6pt);
            \node at (1.25,1) {\tiny$n+1$};
        \end{diagram},\\
        L_n\cdot \tikzcrossing\: &= \frac{n+1}{2}\pr*{\:\begin{diagram}
            \draw(0,0)--(1.5,1.5);
            \draw(0,1.5)--(1.5,0);
            \fill (1.15,0.35) circle(6pt);
            \node at (1.7,0.6) {\tiny $n$};
        \end{diagram}\!-\:\begin{diagram}
                \draw(0,0)--(1.5,1.5);
                \draw(0,1.5)--(1.5,0);
                \fill (1.15,1.15) circle(6pt);
                \node at (1.7,0.9) {\tiny $n$};
            \end{diagram}\!\!}+\h_n(y_1,y_2)\cdot\tikzcrossing,\\
        L_n\cdot\tikzcrossbr\:&=-\frac{n+1}{2}\begin{diagram}
            \draw[black](0,0)--(1.5,1.5);
            \draw[red](0,1.5)--(1.5,0);
            \fill[red] (0.35,1.15) circle(6pt);
            \node at (-0.2,0.9) {\tiny $n$};
        \end{diagram}\:,\\
        L_n\cdot \tikzcrossrb\:&=\frac{n+1}{2}\begin{diagram}
            \draw[red](0,0)--(1.5,1.5);
            \draw[black](0,1.5)--(1.5,0);
            \fill[red] (1.15,1.15) circle(6pt);
            \node at (1.7,0.9) {\tiny $n$};
        \end{diagram}\:-\h_n(y_1,y_2)\cdot\tikzcrossrb,\\
        L_n\cdot \tikzcrossdistant \:&=0.
    \end{align*}
    Here for instance if $n\ge 0$ then $\h_n(y_1,y_2)\cdot\tikzcrossrb=\sum_{i+j=n} \begin{diagram}
            \draw[red](0,0)--(1.5,1.5);
            \draw[black](0,1.5)--(1.5,0);
            \fill[black] (0.35,1.15) circle(6pt);
            \fill[red] (1.15,1.15) circle(6pt);
            \node at (1.7,0.9) {\tiny $j$};
            \node at (-0.2,0.9) {\tiny $i$};
        \end{diagram}$. 
\end{theorem}
\begin{proof}
    The proof is the same computation as the proof of Theorem \ref{thm:halfwittacts}. 
\end{proof}

\section{The core is global sections}
The main objective of this section is to prove that the global sections recovers the core.
\begin{theorem}\label{thm:globalsectioniscore}
    Using the $\sl_2$-action from Section \ref{subsect:sl2actionKLR}, we have
    \[\CE(X)=\Cor_{\sl_2}(\CR_{\alpha}).\] 
\end{theorem}

We prove a significantly more general result that may be useful in situations beyond KLR, namely:
\begin{proposition}\label{prop: torsion_free_core_is_glob_sect}
    Let $G$ be a connected algebraic group over $\BC$ acting algebraically on a normal projective variety $Y$. Let $\mathfrak{g}$ be the associated Lie algebra of $G$. Let $Z \subset Y$ be a proper closed subset, and define $Z'$ as the largest $G$-stable closed subset contained in $Z$ (explicitly, $Z' = \bigcap_{g\in G} gZ$). Let $\CF$ be a $G$-equivariant, torsion-free coherent sheaf on $Y$.
    
    Then the space of sections $\CF(Y \setminus Z)$ inherits a natural $\mathfrak{g}$-action, and the restriction map $\CF(Y \setminus Z') \to \CF(Y \setminus Z)$ is injective. Furthermore, the locally finite part of $\CF(Y \setminus Z)$ (the union of all its finite-dimensional $\mathfrak{g}$-submodules) is exactly the image of the restriction map $\CF(Y \setminus Z') \to \CF(Y \setminus Z)$.
    
    In particular, if $Z$ contains no non-empty $G$-stable closed subsets ($Z' = \emptyset$), the maximal finite-dimensional $\mathfrak{g}$-submodule of $\CF(Y \setminus Z)$ can be identified canonically with the space of global sections $\CF(Y)$.
\end{proposition}
This is enough already to prove that global sections is the core.
\begin{proof}[Proof of Theorem \ref{thm:globalsectioniscore}]
Apply Proposition \ref{prop: torsion_free_core_is_glob_sect} with \[G=\SL_2,Y=X=\prod_{a=1}^N\BP^{n_a},\CF=\CE, \text{ and }Z=X\setminus U.\]Here $Z$ is the union of the hyperplanes on which at least one of the colored effective divisors contains $\infty$. By Lemma \ref{lem:SL2-equivariance}, the sheaf $\CE$ is $\SL_2$-equivariant and torsion-free. It remains to show that $Z$ contains no non-empty closed $\SL_2$-stable subset.

Suppose that $W\subseteq Z$ were such a subset. Since $W$ is projective, it contains a closed $\SL_2$-orbit. A closed orbit contains a point fixed by a Borel subgroup. For the upper-triangular Borel $B\subseteq\SL_2$, the unique $B$-fixed point of $\Sym^{n_a}\BP^1$ is the divisor $n_a[\infty]$. Hence the unique $B$-fixed point of $X$ is \[x_\infty=(n_1[\infty],\ldots,n_N[\infty]).\] It follows that $W$ contains the orbit\[\Delta=\{(n_1[p],\ldots,n_N[p])\colon p\in\BP^1\}.\]But for $p\ne\infty$, the corresponding point of $\Delta$ belongs to $U$, contradicting $W\subset Z$. Thus the largest $\SL_2$-stable closed subset contained in $Z$ is indeed empty.

Now, Proposition \ref{prop: torsion_free_core_is_glob_sect} gives 
\[\Cor_{\sl_2}(\CE(U))=\CE(X),\] so using $\CE(U)=\CR_\alpha$ proves the result.
\end{proof}

\begin{proof}[Proof of Proposition \ref{prop: torsion_free_core_is_glob_sect}]
    The $G$-equivariant structure on $\CF$ differentiates to a natural action of $\mathfrak{g}$, which induces a well-defined $\mathfrak{g}$-action on the sections $\CF(U)$ for any open set $U \subset Y$, including the open set $Y \setminus Z$.

    Since $Y$ is an integral scheme and $\CF$ is a torsion-free sheaf, $\CF$ possesses no non-zero sections supported on proper closed subsets. Consequently, for any open set $V$ satisfying $Y \setminus Z \subset V \subset Y$, the restriction map $\CF(V) \to \CF(Y \setminus Z)$ is injective. It follows that any section $s \in \CF(Y \setminus Z)$ has a uniquely defined maximal domain of definition $U_s \subset Y$ and a uniquely defined locus of non-regularity $S_s = Y \setminus U_s$ satisfying $S_s \subset Z$.

    Next, since $\CF$ is torsion-free and $Y$ is integral, the natural restriction map from the space of sections $\CF(Y \setminus Z)$ to the generic stalk $\CF_\eta$ is injective. Therefore, we may view any finite-dimensional $\mathfrak{g}$-submodule $M \subset \CF(Y \setminus Z)$ as a finite-dimensional $\BC$-subspace of $\CF_\eta$. 
    
    The $G$-equivariant structure of $\CF$ induces an action of $G$ on the infinite-dimensional $\BC$-vector space of rational sections $\CF_\eta$. We claim that $M$ is $G$-stable. Indeed, for any $\xi\in \mathfrak{g}$, comparison of Taylor expansions at $t=0$ gives \[\exp(t\xi)\cdot s = \exp(t\xi\rvert_M)s\] for $s\in M$ and $t$ near zero. So every one-parameter subgroup near the identity preserves $M$. Since $G(\BC)$ is connected, these one-parameter subgroups generate $G(\BC)$, and hence $M$ is $G$-stable.
    
    As a result, for any $s \in M$ and $g \in G$, the geometric translate $g \cdot s$ remains in $M$, meaning it remains regular on $Y \setminus Z$. Because the action of $g$ on $Y$ is an automorphism, the maximal domain of definition of the translated section $g \cdot s$ is exactly $g U_s$, and its locus of non-regularity is $S_{g \cdot s} = g S_s$. However, since $g \cdot s \in M \subset \CF(Y \setminus Z)$, its locus of non-regularity must also be contained in $Z$, yielding $g S_s \subset Z$.

    Now, since $g S_s \subset Z$ for all $g \in G$, it follows by inversion that $S_s \subset g^{-1}Z$ for all $g \in G$. Taking the intersection over all elements of the group, we obtain:
    \[
        S_s \subset \bigcap_{g \in G} g^{-1}Z = \bigcap_{g \in G} gZ = Z'.
    \]
    It follows that the maximal domain of definition $U_s = Y \setminus S_s$ contains $Y \setminus Z'$, meaning the section $s$ extends uniquely to a section in $\CF(Y \setminus Z')$. Consequently, any finite-dimensional $\mathfrak{g}$-submodule $M \subset \CF(Y \setminus Z)$ is contained entirely in the image of $\CF(Y \setminus Z')$. 

    Conversely, we must now demonstrate that every section in $\CF(Y \setminus Z')$ resides inside a finite-dimensional $\mathfrak{g}$-module. Equipping $Z'$ with the reduced induced scheme structure, we may express the space of sections on the open set $Y \setminus Z'$ using Tag 01PQ:
    \[
        \CF(Y \setminus Z') \simeq \varinjlim_{n \ge 1} \operatorname{Hom}_{\CO_Y}(\CI_{Z'}^n, \CF),
    \]
    where $\CI_{Z'}$ is the ideal sheaf defining $Z'$. 
    
    Because $Z'$ is $G$-stable, $\CI_{Z'}$ is canonically a $G$-equivariant coherent sheaf, and this equivariance naturally extends to its powers $\CI_{Z'}^n$. Because $Y$ is a projective variety, each space of global sheaf homomorphisms $\operatorname{Hom}_{\CO_Y}(\CI_{Z'}^n, \CF) = H^0(Y, \shom(\CI_{Z'}^n, \CF))$ is a finite-dimensional vector space. These finite-dimensional spaces are then equipped with the structure of $G$-modules.
    
    Finally, the transition maps in the colimit, as well as the natural restriction map $\CF(Y \setminus Z') \to \CF(Y \setminus Z)$, are $G$-equivariant. Therefore, $\CF(Y \setminus Z')$ arises as a union of finite-dimensional $G$-modules, and its image in $\CF(Y \setminus Z)$ forms the entirety of the locally finite part under the action of $\mathfrak{g}$. 
\end{proof}

\bibliographystyle{plain}
\bibliography{bibliography}

@online{stacks,
    author       = {{The Stacks Project Authors}},
    title        = {\textit{Stacks Project}},
    url = {https://stacks.math.columbia.edu},
    year         = {2024},
  }

@article {EQ23,
    AUTHOR = {Elias, Ben and Qi, You},
     TITLE = {Actions of {$sl_2$} on algebras appearing in
              categorification},
   JOURNAL = {Quantum Topol.},
  FJOURNAL = {Quantum Topology},
    VOLUME = {14},
      YEAR = {2023},
    NUMBER = {4},
     PAGES = {733--806},
      ISSN = {1663-487X,1664-073X},
   MRCLASS = {16W25 (17B60)},
  MRNUMBER = {4668562},
MRREVIEWER = {Zahra\ Riyahi},
       DOI = {10.4171/qt/181},
       URL = {https://doi.org/10.4171/qt/181},
}

@article{KL09,
  title={A diagrammatic approach to categorification of quantum groups {I}},
  author={Khovanov, Mikhail and Lauda, Aaron},
  journal={Represent. Theory},
  volume={13},
  number={14},
  pages={309--347},
  year={2009}
}

@article{rouquier20082,
  title={2 {K}ac-{M}oody algebras},
  author={Rouquier, Rapha{\"e}l},
note={\url{https://arxiv.org/abs/0812.5023}},
  year={2008},
    journal={},
}

@inproceedings{rouquier2012quiver,
  title={Quiver {H}ecke algebras and 2-{L}ie algebras},
  author={Rouquier, Rapha{\"e}l},
  booktitle={Algebra colloquium},
  volume={19},
  number={02},
  pages={359--410},
  year={2012},
  organization={World Scientific}
}

@article{lauda2025action,
  title={Action of the {W}itt algebra on categorified quantum groups},
  author={Grlj, Jernej and Lauda, Aaron D},
  journal={arXiv preprint arXiv:2507.01877},
  year={2025}
}

@article{kleshchev2013affine,
  title={Affine cellularity of {K}hovanov--{L}auda--{R}ouquier algebras in type {A}},
  author={Kleshchev, Alexander S and Loubert, Joseph W and Miemietz, Vanessa},
  journal={Journal of the London Mathematical Society},
  volume={88},
  number={2},
  pages={338--358},
  year={2013},
  publisher={Wiley Online Library}
}

@article{graham1996cellular,
  title={Cellular algebras},
  author={Graham, John J and Lehrer, Gustav I},
  journal={Inventiones mathematicae},
  volume={123},
  number={1},
  pages={1--34},
  year={1996},
  publisher={Springer}
}

@book{schlichenmaier2016krichever,
  title={Krichever-Novikov type algebras. An introduction},
  author={Schlichenmaier, Martin},
  volume={92},
  year={2016},
  publisher={American Mathematical Society Providence, RI}
}

@article{KLII,
   title={A diagrammatic approach to categorification of quantum groups {II}},
   volume={363},
   ISSN={0002-9947},
   url={http://dx.doi.org/10.1090/S0002-9947-2010-05210-9},
   DOI={10.1090/s0002-9947-2010-05210-9},
   number={5},
   journal={Transactions of the American Mathematical Society},
   publisher={American Mathematical Society (AMS)},
   author={Khovanov, Mikhail and Lauda, Aaron},
   year={2010},
   month=Nov, pages={2685–2700} }

@article{varagnolo2011canonical,
  title={Canonical bases and {KLR}-algebras.},
  author={Varagnolo, Michela and Vasserot, Eric},
  journal={Journal f{\"u}r die reine und angewandte Mathematik},
  volume={2011},
  number={659},
  year={2011}
}

@article{Kato_2014,
   title={Poincaré–{B}irkhoff–{W}itt bases and {K}hovanov–{L}auda–{R}ouquier algebras},
   volume={163},
   ISSN={0012-7094},
   url={http://dx.doi.org/10.1215/00127094-2405388},
   DOI={10.1215/00127094-2405388},
   number={3},
   journal={Duke Mathematical Journal},
   publisher={Duke University Press},
   author={Kato, Syu},
   year={2014},
   month=Feb }

@article{McNamaraKLRI,
url = {https://doi.org/10.1515/crelle-2013-0075},
title = {Finite dimensional representations of {K}hovanov–{L}auda–{R}ouquier algebras {I}: Finite type},
title = {},
author = {Peter J. McNamara},
pages = {103--124},
volume = {2015},
number = {707},
journal = {Journal für die reine und angewandte Mathematik (Crelles Journal)},
doi = {doi:10.1515/crelle-2013-0075},
year = {2015},
lastchecked = {2026-08-29}
}

@article{BrundanKleshchevMcNamara_2014,
   title={Homological properties of finite-type {K}hovanov–{L}auda–{R}ouquier algebras},
   volume={163},
   ISSN={0012-7094},
   url={http://dx.doi.org/10.1215/00127094-2681278},
   DOI={10.1215/00127094-2681278},
   number={7},
   journal={Duke Mathematical Journal},
   publisher={Duke University Press},
   author={Brundan, Jonathan and Kleshchev, Alexander and McNamara, Peter J.},
   year={2014},
   month=May }

@article{Kleshchev_2015,
   title={Affine highest weight categories and affine quasihereditary algebras},
   volume={110},
   ISSN={0024-6115},
   url={http://dx.doi.org/10.1112/plms/pdv004},
   DOI={10.1112/plms/pdv004},
   number={4},
   journal={Proceedings of the London Mathematical Society},
   publisher={Wiley},
   author={Kleshchev, Alexander S.},
   year={2015},
   month=Mar, pages={841–882} }

@article{kleshchevloubert2015affine,
  title={Affine cellularity of {K}hovanov--{L}auda--{R}ouquier algebras of finite types},
  author={Kleshchev, Alexander S and Loubert, Joseph W},
  journal={International Mathematics Research Notices},
  volume={2015},
  number={14},
  pages={5659--5709},
  year={2015},
  publisher={Oxford University Press}
}

@article{koenig2012affine,
  title={Affine cellular algebras},
  author={Koenig, Steffen and Xi, Changchang},
  journal={Advances in Mathematics},
  volume={229},
  number={1},
  pages={139--182},
  year={2012},
  publisher={Elsevier}
}

@article{qi2022symmetries,
  title={Symmetries of $\mathfrak{gl}_N$-foams},
  author={Qi, You and Robert, Louis-Hadrien and Sussan, Joshua and Wagner, Emmanuel},
  journal={arXiv preprint arXiv:2212.10106},
  year={2022}
}

@article{qi2023symmetries,
  title={Symmetries of equivariant {K}hovanov-{R}ozansky homology},
  author={Qi, You and Robert, Louis-Hadrien and Sussan, Joshua and Wagner, Emmanuel},
  journal={arXiv preprint arXiv:2306.10729},
  year={2023}
}

@article{guerinroz2025action,
  title={An action of the {W}itt algebra on {K}hovanov-{R}ozansky homology},
  author={Gu{\'e}rin, Alexis and Roz, Felix},
  journal={arXiv preprint arXiv:2501.19096},
  year={2025}
}

@article{khovanov2016positive,
  title={Positive half of the {W}itt algebra acts on triply graded link homology},
  author={Khovanov, Mikhail and Rozansky, Lev},
  journal={Quantum Topology},
  volume={7},
  number={4},
  pages={737--795},
  year={2016}
}

@misc{KMZ,
      title={Even and odd minimal categorifications of the nilpotent part of classical and quantum sl(2)}, 
      author={Khovanov, Mikhail and Martinez, Alvaro L. and Zhou, Fan},
      year={2026},
      eprint={2609.04567},
      archivePrefix={arXiv},
      primaryClass={math.QA},
      url={https://arxiv.org/abs/2609.04567}, 
      note={\url{https://arxiv.org/abs/2609.04567}}
}

@article{zhou2026bgg,
  title={B{GG} resolutions, {K}oszulity, and stratifications, part {II}: the {J}acobi-{T}rudi algebra},
  author={Zhou, Fan},
  journal={arXiv preprint arXiv:2605.09261},
  year={2026},
  note={\url{https://arxiv.org/abs/2605.09261}}
}

@article{zhou2024bgg,
  title={B{GG} Resolutions, {K}oszulity, and Stratifications, Part {I}: the nil{B}rauer Algebra},
  author={Zhou, Fan},
  journal={arXiv preprint arXiv:2402.06890},
  year={2024},
  note={\url{https://arxiv.org/abs/2402.06890}}
}

\end{document}